\documentclass[11pt]{amsart}

\usepackage{latexsym}
\usepackage{amssymb}
\usepackage{amsmath}
\usepackage{color}

\newtheorem{theorem}{Theorem}[section]
\newtheorem{lemma}[theorem]{Lemma}
\newtheorem{proposition}[theorem]{Proposition}
\newtheorem{corollary}[theorem]{Corollary}
\newtheorem{definition}[theorem]{Definition}

\newtheorem{remark}[theorem]{Remark}

\newtheorem{question}[theorem]{Question}

\newcommand\supp{\mathop{\rm supp}}

\newcommand\esssup{\mathop{\rm esssup}}

\newcommand\id{\mathop{\rm id}}

\newcommand\mm{\mathop{\rm m}}
\newcommand\hh{\mathop{\rm h}}

\newcommand\nph{\varphi}

\newcommand\cb{\mathop{\rm cb}}

\newcommand\vn{\mathop{\rm VN}}

\newcommand{\cl}[1]{\mathcal{#1}}
\newcommand{\bb}[1]{\mathbb{#1}}

\newcommand\Bd{\mathop{\cl B(H)}}
\newcommand{\ve}{\Vert}

\newcommand{\ip}[2]{\ensuremath{\left\langle #1 , #2\right\rangle}}

\newcommand{\vip}[2]{\ensuremath{\left\langle #1 | #2\right\rangle}}

\begin{document}

\title{Herz-Schur multipliers of dynamical systems}

\author[A. McKee]{A. McKee}
\address{Pure Mathematics Research Centre,
  Queen's University Belfast, Belfast BT7 1NN, United Kingdom}
\email{amckee240@qub.ac.uk}

\author[I. Todorov]{I. G. Todorov}
\address{Pure Mathematics Research Centre,
  Queen's University Belfast, Belfast BT7 1NN, United Kingdom}
\email{i.todorov@qub.ac.uk}

\author[L. Turowska]{L. Turowska}
\address{Department of Mathematical Sciences,
Chalmers University of Technology and  the University of Gothenburg,
Gothenburg SE-412 96, Sweden}
\email{turowska@chalmers.se}

\date{}
\maketitle

\tableofcontents

\section{Introduction}

The notion of a Schur multiplier has its origins in the work of I. Schur
in the early 20$^{\rm th}$ century, and is based on entry-wise (or Hadamard) product of matrices.
More specifically, a bounded function
$\nph:\mathbb{N}\times\mathbb{N}\rightarrow\mathbb{C}$ is called a
Schur multiplier if $(\nph(i,j)a_{i,j})$ is the matrix of a bounded
linear operator on $\ell^2$ whenever $(a_{i,j})$ is such.
A concrete description of Schur multipliers, which found numerous
applications thereafter, was given by A. Grothendieck in his
{\it R$\acute{e}$sum$\acute{e}$} \cite{Gro} (see also \cite{Pi}).
A measurable version of Schur multipliers was
developed by M. S. Birman and M. Z. Solomyak (see \cite{BS4} and the
references therein) and V. V. Peller \cite{peller}. More
concretely, given standard measure spaces $(X,\mu)$ and $(Y,\nu)$
and a function $\nph : X\times Y \rightarrow \bb{C}$, one defines a
linear transformation $S_{\varphi}$ on the space of all
Hilbert-Schmidt operators from $H_1 = L^2(X,\mu)$ to $H_2 =
L^2(Y,\nu)$ by multiplying their integral kernels by $\varphi$; if
$S_{\varphi}$ is bounded in the operator norm (in which case $\nph$
is called a \emph{measurable Schur multiplier}), it is extended to
the space $\cl K(H_1,H_2)$ of all compact operators from $H_1$ into
$H_2$ by continuity. The map $S_{\varphi}$ is defined on the space
$\cl B(H_1,H_2)$ of all bounded linear operators from $H_1$ into
$H_2$ by taking the second dual of the constructed map on $\cl
K(H_1,H_2)$. A characterisation of measurable Schur multipliers,
extending Grothendieck's result, was obtained in
\cite{haag} and \cite{peller} (see also \cite {kp} and \cite{spronk}). Namely, a function
$\nph\in L^{\infty}(X\times Y)$ was shown to be a Schur multiplier
if and only if $\nph$ coincides almost everywhere with a function of
the form $\sum_{k=1}^{\infty} a_k(x) b_k(y)$, where $(a_k)_{k\in
\bb{N}}$ and $(b_k)_{k\in \bb{N}}$ are families of essentially
bounded measurable functions such that $\esssup_{x\in X}
\sum_{k=1}^{\infty} |a_k(x)|^2 < \infty$ and $\esssup_{y\in Y}
\sum_{k=1}^{\infty} |b_k(y)|^2 < \infty$.

Among the large number of applications of Schur multipliers is the description
of the space $M^{\rm cb}A(G)$ of completely bounded multipliers
(also known as Herz-Schur multipliers) of the Fourier algebra $A(G)$
of a locally compact group $G$, introduced by J. de Canni\`{e}re and U. Haagerup in \cite{ch}.
Namely, as shown by M. Bo\.{z}ejko and G. Fendler \cite{bf} ,
$M^{\rm cb}A(G)$ can be isometrically identified with the space of all
Schur multipliers on $G\times G$ of Toeplitz type. An alternative proof of this
result was given by P. Jolissaint \cite{j}.

Herz-Schur multipliers have been highly instrumental in operator algebra theory,
providing the route to defining and studying a number of
approximation properties of group C*-algebras and group von Neumann algebras
(we refer the reader to \cite{hk}, \cite{bo} and \cite{knudby}).
Here one uses the fact that every Herz-Schur multiplier on a locally compact group $G$
gives rise to a (completely bounded) map on the von Neumann algebra $\vn(G)$
of $G$, leaving invariant the reduced C*-algebra $C_r^*(G)$ of $G$.

In view of the large number of applications of
Herz-Schur multipliers in operator algebra theory, it is natural
to seek generalisations going beyond the context of group algebras.
The main goal of this paper is to extend the notion of Herz-Schur multipliers
to the setting of non-commutative dynamical systems.
Given a C*-algebra $A$, a locally compact group $G$, and an action $\alpha$
of $G$ on $A$, we define transformations on the (reduced) crossed product
$A\rtimes_{r,\alpha} G$ of $A$ by $G$, which, in the case $A = \bb{C}$, reduce to
the classical Herz-Schur multipliers and, in the case of a
discrete group $G$, to multipliers defined  recently by E. Bedos and R. Conti in \cite{bc}.   More generally, we introduce
Schur $A$-multipliers which, in the case $A = \bb{C}$, reduce to the classical
measurable Schur multipliers.
In Section \ref{s_smvna}, we establish a characterisation of Schur $A$-multipliers
that generalises the classical description of Schur multipliers
(Theorem \ref{th_chschura}).
We exhibit a large class of Schur $A$-multipliers
defined in terms of Hilbert $A$-bimodules, and show that it exhausts all Schur $A$-multipliers
in the case $A$ is finite-dimensional (Theorem \ref{c_hilmodc}).
In Section \ref{s_ggamma}, we prove a transference theorem in the new setting,
identifying isometrically the Hezr-Schur multipliers of
the dynamical system $(A,G,\alpha)$ with the invariant part of the Schur $A$-multipliers
(see Theorems \ref{th_tr} and \ref{th_invpart}).
We introduce bounded multipliers of $(A,G,\alpha)$ and relate them to
Herz-Schur $(A,G,\alpha)$-multipliers, extending a corresponding result from \cite{ch}.
In Section \ref{s_mwscp}, we provide a description of a more general
and closely related class of multipliers, namely, Herz-Schur multipliers associated with
weak* closed crossed products,
as the commutant of the scalar valued Herz-Schur multipliers associated with elements of
$M^{\rm cb}A(G)$ (Theorem \ref{th_commute}).
While in the case $A = \bb{C}$ this description is straightforward,
here we need to use structure theory of crossed products and some recent results from \cite{akt}.

The rest of the paper is devoted to special classes of Herz-Schur multipliers.
Namely, in Section \ref{s_scm}, we consider multipliers naturally associated with the
Haagerup tensor product of two copies of $A$, and multipliers defined on groupoids.
In the former case, we relate our notion to examples of Herz-Schur multipliers exhibited in \cite[Theorem 4.5]{bc}
in the case of a discrete group.
In the latter case, we show that completely bounded multipliers of the Fourier algebra
of a groupoid, defined in \cite{renault}, form a subclass of the class of Herz-Schur multipliers
introduced in the present work.

The results in Section \ref{s_cm} were our original motivation for the present paper.
Here, we consider the case of a locally compact abelian group $G$
and its canonical action $\alpha$ on the C*-algebra $C^*(\Gamma)$ of the dual group $\Gamma$.
We focus on a special class $\frak{F}(G)$ of Herz-Schur multipliers, which we call convolution multipliers,
and its natural subclass $\frak{F}_{\theta}(G)$ of weak* extendible convolution multipliers.
We show that the Fourier-Stieltjes algebras $B(G)$ and $B(\Gamma)$ can both be viewed
as subspaces of $\frak{F}_{\theta}(G)$, while $\frak{F}_{\theta}(G)$ is a subspace of their Fubini product.
When the crossed product of $C^*(\Gamma)\rtimes_{\alpha} G$ is canonically identified
with the space $\cl K(L^2(G))$ of all compact operators on $L^2(G)$,
the elements $u$ of $B(G)$ give rise to the measurable Schur multipliers corresponding to
$u$ {\it via} the aforementioned Bo\.{z}ejko-Fendler classical transference theorem,
while the elements of $B(\Gamma)$ correspond to a well-known class of
completely bounded maps, arising from a representation of the measure algebra $M(G)$
of $G$ on $\cl K(L^2(G))$, studied in a variety of contexts
in both operator algebra theory and quantum information theory, and by
a number of authors including F. Ghahramani \cite{g},
M. Neufang and V. Runde \cite{neurun}, M. Neufang, Zh.-J. Ruan and N. Spronk \cite{neuruaspro}
and E. St\o{}rmer \cite{stormer}.
The main result of the section are Theorem \ref{th_me} and \ref{th_me2}, where we
identify the set $\frak{F}_\theta(G)$, and an associated subset,
of convolution multipliers of $G$ as subsets of the joint commutant of the
two families described above.

The paper uses various notions from Operator Space Theory;
we refer the reader to \cite{blm}, \cite{er}, \cite{paulsen} or \cite{pisier} for the basics.
For background and notation on crossed products, which will be needed
in Sections \ref{s_ggamma}, \ref{s_scm} and \ref{s_cm}, we refer the reader to \cite{dw}.

\medskip

We finish this section with setting some notation.
If $\cl E$ and $\cl F$ are vector spaces, we let $\cl E\odot\cl F$
be their algebraic tensor product.
For a Banach space $\cl X$, we let
$\cl B(\cl X)$ (resp. $\cl K(\cl X)$)
be the algebra of all bounded linear (resp. compact) operators on $\cl X$, and
denote by $I_{\cl X}$ the identity operator on $\cl X$.
If $H$ and $K$ are Hilbert spaces, we denote by $H\otimes K$
their Hilbertian tensor product; for operators $S\in \cl B(H)$ and $T\in \cl B(K)$,
we let $S\otimes T$ be the bounded operator on $H\otimes K$ given by
$(S\otimes T)(\xi\otimes\eta) = S\xi \otimes T\eta$.
The (norm closed) spacial tensor product of two (norm closed)
operator spaces $\cl U\subseteq \cl B(H)$ and
$\cl V\subseteq \cl B(K)$ will be denoted by $\cl U\otimes\cl V$.
If $\cl U$ and $\cl V$ are weak* closed, their
weak* spacial tensor product will be denoted by $\cl U\bar\otimes\cl V$.

\section{Schur $A$-multipliers}\label{s_smvna}

Let $(X,\mu)$ be a standard measure space; this means that $\mu$ is a Radon measure
with respect to some
complete metrisable separable locally compact topology
(called an \emph{admissible} topology)
on $X$.
For $p = 1,2$ and a Banach space $\cl E$, we write
$L^p(X,\cl E)$ for the corresponding Lebesgue space of all
(equivalence classes of) weakly measurable
$p$-summable $\cl E$-valued functions on $X$ (see {\it e.g.} \cite[Appendix B]{dw}).
If $H$ and $K$ are separable Hilbert spaces and $\cl E\subseteq \cl B(H,K)$ is
a weak* closed subspace, let
$L^{\infty}(X,\cl E)$ be the space of all
(equivalence classes of) bounded
$\cl E$-valued functions $T$ on $X$ such that,
for every $\xi\in H$ and every $\eta\in K$,
the functions $x\to T(x)\xi$ and
$x\to T(x)^*\eta$ are weakly measurable.
Note that $L^{\infty}(X,\cl E)$ contains all bounded weakly measurable
functions from $X$ into $\cl E$.
Analogously to \cite[Chapter IV, Section 7]{takesaki1},
we often identify an element $g$ of $L^{\infty}(X,\cl E)$ with the operator $D_g$
from $L^2(X,H)$ into $L^2(X,K)$ given by
$(D_g\xi)(x) = g(x)(\xi(x))$, $x\in X$.

We write $\|\cdot\|_p$ for the norm on $L^p(X,\cl E)$, $p = 1,2,\infty$.
In the case $\cl E$ coincides with the complex field $\bb{C}$, we simply write $L^p(X)$.
If $f\in L^p(X)$ and $a\in \cl E$, we let $f\otimes a\in L^p(X,\cl E)$ be the function given by
$(f\otimes a)(x) = f(x) a$, $x\in X$.

{\it We fix throughout the section a separable Hilbert space $H$.}
For $a\in L^{\infty}(X)$, let $M_a \in \cl B(L^2(X))$ be the operator given by $M_a\xi = a\xi$;
set
$$\cl D_X = \{M_a : a\in L^{\infty}(X)\}.$$
Note that the identification
$L^2(X)\otimes H \equiv L^2(X,H)$
yields a unitary equivalence between
$L^{\infty}(X,\cl B(H))$ and $\cl D_X\bar\otimes \cl B(H)$ \cite[Theorem 7.10]{takesaki1}.

Let $(Y,\nu)$ be a(nother) standard measure space.
We equip the direct products $X\times Y$ and $Y\times X$ with the
corresponding product measures.
It is easy to see that, if $k\in L^2(Y\times X, \cl B(H))$ and $\xi\in L^2(X,H)$ then, for almost all $y\in Y$,
the function $x\to k(y,x)\xi(x)$ is weakly measurable;
moreover,
\begin{eqnarray}\label{eq_tou}
\int_X \ve k(y,x)\xi(x)\ve d\mu(x)
& \leq &  \int_X \ve k(y,x)\ve\ve\xi(x)\ve d\mu(x) \nonumber\\
& \leq & \|\xi\|_2 \left(\int_X \ve k(y,x)\ve^2 d\mu(x)\right)^{1/2}. 
\end{eqnarray}
It follows that the formula
\begin{equation}\label{d_defk}
(T_k\xi)(y)=\int_X k(y,x)\xi(x) d\mu(x), \ \ \  y\in Y,
\end{equation}
defines a (weakly measurable) function $T_k\xi : Y\to H$.

\begin{lemma}\label{l_Tk}
Let $k\in L^2(Y\times X, \cl B(H))$. Equation (\ref{d_defk}) defines
a bounded operator $T_k : L^2(X,H)\to L^2(Y,H)$ with
$\|T_k\|\leq \|k\|_2$.
Moreover, $T_k = 0$ if and only if $k = 0$ almost everywhere.
\end{lemma}
\begin{proof}
Let $\xi\in L^2(X,H)$. Then, by (\ref{eq_tou}),
\begin{eqnarray*}
\ve T_k\xi\ve^2
& = & \int_Y \ve T_k\xi(y)\ve^2 d\nu(y)\\
& \leq & \|\xi\|_2^2 \int_Y \int_X \ve k(y,x)\ve^2 d\mu(x) d\nu(y) =  \|k\|_2^2 \|\xi\|_2^2.
\end{eqnarray*}
Thus, $T_k$ is bounded and its norm does not exceed $\|k\|_2$.

It is clear that if $k = 0$ almost everywhere then $T_k = 0$.
Conversely,  suppose that
$T_k = 0$. Choose a countable dense subset $\{e_i\}_{i\in \bb{N}}$ of $H$.
If $\xi\in L^2(X)$ and $\eta\in L^2(Y)$ then
$$\int_{X\times Y} \langle k(y,x)e_i,e_j\rangle \xi(x)\overline{\eta(y)}dx dy
= \langle T_k(\xi\otimes e_i),\eta\otimes e_j\rangle = 0,$$
and it follows that $\langle k(y,x)e_i,e_j\rangle = 0$ almost everywhere, for all $i,j$.
Since $k(y,x)$ is a bounded operator,
$k(y,x) = 0$ for almost all $(x,y)$.
\end{proof}

If $\cl M\subseteq \cl B(H)$ is a C*-subalgebra, let
$$\cl S_2(Y\times X,\cl M) = \{T_k : k\in L^2(Y\times X,\cl M)\}.$$
Note that, if $w\in L^2(Y\times X)$ and $a\in \cl M$ then
\begin{equation}\label{eq_wta}
T_{w\otimes a} = T_w \otimes a.
\end{equation}
Letting
$$\cl K \stackrel{def}{=}\cl K(L^2(X),L^2(Y))$$
be the space of all compact operators from $L^2(X)$ into $L^2(Y)$, we have that
$\cl S_2\odot\cl M$ is norm dense in $\cl K\otimes \cl M$
(here $\cl S_2$ denotes the space of all Hilbert-Schmidt opertors from $L^2(X)$ into $L^2(Y)$).
We conclude that $\cl S_2(Y\times X,\cl M)$ is norm dense in $\cl K\otimes \cl M$
and equip it with the operator space structure
arising from its inclusion into $\cl K\otimes\cl M$.

{\it We fix throughout the section a non-degenerate separable C*-algebra $A\subseteq \cl B(H)$.}
If $B$ is a(nother) C*-algebra, we denote by $CB(A,B)$ the
space of all completely bounded linear maps from $A$ into $B$ and write $CB(A) = CB(A,A)$.
A function $\nph : X\times Y\to CB(A,\cl B(H))$ will be called \emph{pointwise measurable}
if, for every $a\in A$, the function $(x,y)\to \nph(x,y)(a)$
from $X\times Y$ into $\cl B(H)$ is weakly measurable.
Let $\nph : X\times Y\to CB(A,\cl B(H))$ be a bounded pointwise measurable function.
For $k\in L^2(Y\times X,A)$, let $\nph\cdot k : Y\times X \to \cl B(H)$ be
the function given by
$$(\nph\cdot k)(y,x) = \nph(x,y)(k(y,x)), \ \ \ (y,x)\in Y\times X.$$
It is easy to show that
$\nph\cdot k$ is weakly measurable;
since $\nph$ is bounded, $\nph\cdot k\in L^2(Y\times X, \cl B(H))$ and
$$\|\nph\cdot k\|_2 \leq \|\nph\|_{\infty}\|k\|_2.$$
Let
$$S_{\nph} : \cl S_2(Y\times X,A)\to \cl S_2(Y\times X,\cl B(H))$$
be the linear map given by
$$S_{\nph}(T_k) = T_{\nph\cdot k}, \ \ \ k\in L^2(Y\times X,A).$$
By Lemma \ref{l_Tk}, $S_\nph$ is well-defined.

\begin{definition}\label{d_schura}
A bounded pointwise measurable map
$$\varphi : X\times Y\to CB(A,\cl B(H))$$ will be called a
\emph{Schur $A$-multiplier} if the map
$S_{\varphi}$ is completely bounded.
\end{definition}

It follows from the discussion after Lemma \ref{l_Tk} that
a bounded pointwise measurable function
$\nph : X\times Y \to CB(A,\cl B(H))$ is a Schur $A$-multiplier
if and only if the map $S_{\nph}$ possesses a completely bounded extension to a map
from $\cl K \otimes A$ into $\cl K\otimes \cl B(H)$ (which will still be denoted by $S_{\nph}$).

We let $\frak{S}(X,Y; A)$ be the space of all Schur $A$-multipliers and endow it
with the norm
\begin{equation}\label{eq_non}
\|\nph\|_{\frak{S}} = \|S_{\nph}\|_{\cb};
\end{equation}
it follows from
Lemma \ref{l_Tk} that if
$S_{\nph} = 0$ then $\nph = 0$ almost everywhere, and hence
(\ref{eq_non}) indeed defines a norm on $\frak{S}(X,Y; A)$.

Note that Schur $\bb{C}$-multipliers coincide with the
classical (measurable) Schur multipliers \cite{peller}, \cite{kp}.

A special role in our considerations will be played by Schur $A$-multipliers $\nph$
for which $\nph(x,y)\in CB(A)$ for  all $(x,y)\in X\times Y$, that is, ones for which
the range of $\nph(x,y)$ is in $A$.
In this case, $S_{\nph}$ is a map on $\cl S_2(Y\times X,A)$.
We let $\frak{S}_0(X,Y; A)$ be the space of all such
Schur $A$-multipliers.
The next proposition shows that $\frak{S}_0(X,Y; A)$
does not depend on the faithful *-representation of $A$.

\begin{proposition}\label{p_indep}
Let $\theta : A\to \cl B(K)$ be a faithful *-representation of $A$ on a separable Hilbert space $K$.
A bounded pointwise measurable map
$\nph : X\times Y\to CB(A)$ is a Schur $A$-multiplier
if and only if the (bounded pointwise measurable)
map $\nph_{\theta} : X\times Y\to CB(\theta(A))$,
given by $\nph_{\theta}(x,y)(\theta(a)) = \theta (\nph(x,y)(a))$, $a\in A$,
is a Schur $\theta(A)$-multiplier.
Moreover, $\|\nph\|_{\frak{S}}=\|\nph_\theta\|_{\frak{S}}$.
\end{proposition}
\begin{proof}
Let $B = \theta(A)$.
Note that the map $\id\otimes\theta : \cl K\odot A \to \cl K \odot B$ given by
$(\id\otimes\theta)(T\otimes a) = T\otimes \theta(a)$ extends to a
complete isometry from $\cl K\otimes A$ onto $\cl K\otimes B$ \cite{blm}.
Let $k\in L^2(Y\times X,A)$; then the map $k_{\theta} = \theta\circ k$ belongs to $L^2(Y\times X,B)$.
We claim that
\begin{equation}\label{eq_otthe}
T_{k_{\theta}} = (\id\otimes\theta)(T_k).
\end{equation}
To see (\ref{eq_otthe}), note first that, by (\ref{eq_wta}), it holds when $k = k'\otimes a$ for some
$k'\in L^2(Y\times X)$ and $a\in A$; hence, by linearity, it holds if $k\in L^2(Y\times X)\odot A$.
By \cite[Proposition 7.4]{takesaki1},  there exists a sequence $(k_i)_{i\in \bb{N}}\subseteq L^2(Y\times X)\odot A$
such that $\|k_i - k\|_2\to_{i\to \infty} 0$. It follows that $\|\theta\circ k_i - \theta\circ k\|_2\to_{i\to \infty} 0$.
By Lemma \ref{l_Tk}, $T_{k_i}\to T_k$ and $T_{\theta\circ k_i}\to T_{\theta\circ k}$ in
the operator norm. Thus, $(\id\otimes\theta)(T_{k_i})\to (\id\otimes\theta)(T_k)$ in
the operator norm, and we conclude that $T_{k_{\theta}} = (\id\otimes\theta)(T_k)$.

By (\ref{eq_otthe}),
$$(\id\otimes\theta)(S_{\nph}(T_k)) = S_{\nph_{\theta}}(T_{k_{\theta}});$$
in other words,
$$S_{\nph_{\theta}} = (\id\otimes\theta)\circ S_{\nph}\circ (\id\otimes\theta^{-1}).$$
It follows that $S_{\nph}$ is completely bounded if and only if
$S_{\nph_{\theta}}$ is so and that, in this case,
$\|\nph\|_{\frak{S}} = \|\nph_\theta\|_{\frak{S}}$.
\end{proof}

Proposition \ref{p_indep} allows us to
view the elements of $\frak{S}_0(X,Y; A)$
independently of the particular faithful representation of $A$
on a separable Hilbert space; we will thus in the sequel
refer to $A$-valued Schur $A$-multipliers
without the need to specify a particular representation.

\begin{lemma}\label{l_modp}
Let $\nph\in  \frak{S}(X,Y; A)$. If $C\in \cl D_X$, $D\in \cl D_Y$ and $T\in \cl K\otimes A$ then
\begin{equation}\label{eq_modp}
S_{\nph}((D\otimes I_H)T(C\otimes I_H)) = (D\otimes I_H)S_{\nph}(T)(C\otimes I_H).
\end{equation}
\end{lemma}
\begin{proof}
By continuity and linearity, it suffices to establish (\ref{eq_modp}) in the case
$T = T_k\otimes a$, where $k\in L^2(Y\times X)$ and $a\in A$.
Assuming that $C = M_c$ and $D = M_d$, where $c\in L^{\infty}(X)$ and
$d\in L^{\infty}(Y)$,  we have that both
$S_{\nph}((D\otimes I_H)T(C\otimes I_H))$ and $(D\otimes I_H)S_{\nph}(T)(C\otimes I_H)$
coincide with $T_{k'}$, where $k'$ is the (weakly measurable)
function from $Y\times X$ into $\cl B(H)$ given by
$$k'(y,x) =  c(x)d(y) k(x,y) \nph(x,y)(a).$$
\end{proof}

\begin{lemma}\label{l_repkte}
Let $\cl E$ and $\cl L$ be separable Hilbert spaces and
$\theta : \cl K(\cl E)\otimes A\to \cl B(\cl L)$ be a non-degenerate *-representation.
Then there exists a separable Hilbert space $K$, a unitary operator $U : \cl L \to \cl E\otimes K$
and a non-degenerate *-representation $\rho : A\to \cl B(K)$
such that
$$U \theta(b\otimes a) U^* = b\otimes \rho(a), \ \ \ \ b\in \cl K(\cl E), a\in A.$$
\end{lemma}
\begin{proof}
Let $M(\cl K(\cl E)\otimes A)$ be the multiplier algebra of $\cl K(\cl E)\otimes A$.
There exists a unital *-homomorphism
$\hat{\theta} : M(\cl K(\cl E)\otimes A) \to \cl B(\cl L)$ extending $\theta$
(see {\it e.g.} \cite[Proposition 3.12.10]{ped2}).
The map $x\to \hat{\theta}(x\otimes I_H)$ is clearly a non-degenerate *-representation of
$\cl K(\cl E)$ on $\cl L$.
Thus, there exists a separable Hilbert space $K$ and a unitary operator
$U :  \cl L\to \cl E\otimes K$ such that
$$U \hat{\theta}(b\otimes I_H)U^* = b\otimes I_{K}, \ \ \ b\in \cl K(\cl E).$$
Let $\tilde{\theta} : M(\cl K(\cl E)\otimes A)\to \cl B(\cl E\otimes K)$ be given by
$$\tilde{\theta}(T) = U\hat{\theta}(T)U^*, \ \ \ T\in M(\cl K(\cl E)\otimes A).$$
For $a\in A$ and $b\in \cl K(\cl E)$, the operators
$\tilde{\theta}(b\otimes I_H)$ and $\tilde{\theta}(I_{\cl E}\otimes a)$ commute.
It follows that $\tilde{\theta}(I_{\cl E}\otimes a) = I_{\cl E}\otimes \rho(a)$, for some operator
$\rho(a)\in \cl B(K)$. Since $\tilde{\theta}$ is a unital *-homomorphism, the map
$\rho : A\to \cl B(K)$ is easily seen to be a non-degenerate *-homomorphism.
Moreover, if $b\in \cl K(\cl E)$ and $a\in A$ then
\begin{eqnarray*}
U \theta(b\otimes a) U^* & = & U \hat{\theta}(b\otimes I_H) \hat{\theta}(I_{\cl E}\otimes a) U^*
= U \hat{\theta}(b\otimes I_H)U^* U\hat{\theta}(I_{\cl E}\otimes a) U^*\\
& = &
\tilde{\theta}(b\otimes I_H) \tilde{\theta}(I_{\cl E}\otimes a) =
(b\otimes I_{K})(I_{\cl E}\otimes\rho(a)) = b\otimes\rho(a).
\end{eqnarray*}
\end{proof}

\begin{theorem}\label{th_chschura}
Let $\varphi:X\times Y\to CB(A,\Bd)$ be a bounded pointwise measurable function.
The following are equivalent:

(i) \ $\nph$ is a Schur $A$-multiplier;

(ii) there exist a separable Hilbert space $K$, a non-degenerate *-representa-tion
$\rho : A\to \cl B(K)$,
$V \in L^{\infty}(X,\cl B(H,K))$ and $W\in L^{\infty}(Y,\cl B(H,K))$
such that, for all $a\in A$,
$$\varphi(x,y)(a) = W^*(y)\rho(a)V(x),$$
for almost all $(x,y)\in X\times Y$.
\end{theorem}
\begin{proof}
(i)$\Rightarrow$(ii)
Suppose that $\nph\in \frak{S}(X,Y;A)$.
Let $\cl E = L^2(Y)\oplus L^2(X)$ and
$\Phi : \cl K(\cl E)\otimes A\to \cl K(\cl E)\otimes\cl B(H)$ be given by
\begin{equation}\label{eq_tbt}
\Phi\left(\left(\begin{matrix} x_{1,1} & x_{1,2}\\ x_{2,1} & x_{2,2}\end{matrix}\right)\otimes a\right) =
\left(\begin{matrix} 0 & S_{\nph}(x_{1,2}\otimes a)\\ 0 & 0\end{matrix}\right).
\end{equation}
It is clear that $\Phi$ is a completely bounded map with
$\|\Phi\|_{\cb} = \|S_{\nph}\|_{\cb}$. By the Haagerup-Paulsen-Wittstock Theorem, there exist a Hilbert space $\cl L$,
a non-degenerate *-homomorphism
$\theta : \cl K(\cl E)\otimes A\to \cl B(\cl L)$ and
operators $V_0,W_0\in \cl B(\cl E\otimes H,\cl L)$ such that
$$\Phi(T) = W_0^*\theta(T)V_0, \ \ \ T\in \cl K(\cl E)\otimes A.$$
As $\cl K(\cl E)\otimes A$ is separable, we may assume that $\cl L$ is separable.
By Lemma \ref{l_repkte}, there exist
a separable Hilbert space $K$, a unitary operator $U : \cl L \to \cl E\otimes K$
and a *-representation $\rho : A\to \cl B(K)$ such that
$$U\theta(b\otimes a) U^* = b\otimes \rho(a), \ \ \ \ b\in \cl K(\cl E), a\in A.$$
Let $\hat{W} = UW_0$ and $\hat{V} = UV_0$.
Then
$$\Phi(b\otimes a) = \hat{W}^*(b\otimes \rho(a)) \hat{V}, \ \ \ b\in \cl K(\cl E), a\in A.$$
Writing $\hat{V}$ and $\hat{W}$ in two by two matrix form and recalling (\ref{eq_tbt}),
we conclude that there exist bounded operators
$\tilde{V} : L^2(X,H)\to L^2(X,K)$ and
$\tilde{W} : L^2(Y,H)\to L^2(Y,K)$
such that
\begin{equation}\label{eq_tildevw}
S_{\nph}(b\otimes a) = \tilde{W}^*(b\otimes \rho(a)) \tilde{V}, \ \ \ b\in \cl K, a\in A.
\end{equation}

Let
$$\cl S \stackrel{\text{def}}{=}
\overline{[\{T\tilde{V}L^2(X,H) : T\in \cl K(L^2(X))\otimes \rho(A)\}]}.$$
Clearly, $\cl S$  is invariant under $\cl K(L^2(X))\otimes \rho(A)$.
Thus, the projection onto $\cl S$ has the form $I_{L^2(X)}\otimes E$, for
some projection $E\in \rho(A)'$.
Moreover,
\begin{equation}\label{eq_l2e}
\tilde{V} = (I_{L^2(X)}\otimes E)\tilde{V}.
\end{equation}
Setting $\tilde{\rho} = \id\otimes\rho$ (so that $\tilde{\rho}$ is a map
from $\cl K\otimes A$ into $\cl K\otimes \cl B(K)$),
by (\ref{eq_tildevw}) and (\ref{eq_l2e}), we now have
\begin{equation}\label{eq_snphwv}
S_{\varphi}(T) = \tilde{W}^*\tilde{\rho}(T)(I_{L^2(X)}\otimes E)\tilde{V}, \ \ \
T\in \cl K\otimes A.
\end{equation}
Note, further, that if $c\in L^{\infty}(X)$ and $d\in L^{\infty}(Y)$ then
\begin{equation}\label{eq_modrhot}
\tilde{\rho}((M_d^*\otimes I_H)T(M_c\otimes I_H)) = (M_d^*\otimes I_K)\tilde{\rho}(T)(M_c\otimes I_K).
\end{equation}

Let $W = (I_{L^2(Y)}\otimes E)\tilde{W}$.
Since
$$\tilde{\rho}(T)(I_{L^2(X)}\otimes E) = (I_{L^2(Y)}\otimes E)\tilde{\rho}(T),$$
we conclude from (\ref{eq_snphwv}) that
\begin{equation}\label{eq_wvti}
S_{\varphi}(T) = \tilde{W}^*(I_{L^2(Y)}\otimes E)\tilde{\rho}(T)\tilde{V}=W^*\tilde{\rho}(T)\tilde{V},
\end{equation}
for every $T\in \cl K\otimes A$.

Identities (\ref{eq_modrhot}) and (\ref{eq_wvti}) and Lemma \ref{l_modp} imply that
\begin{equation}\label{eqmodularity2}
W^*(M_d^*\otimes I_K)\tilde{\rho}(T)\tilde V=(M_d^*\otimes I_H)W^*\tilde{\rho}(T)\tilde V,
\ \ \ T\in \cl K\otimes A.
\end{equation}
Thus,
\[\ip{\tilde\rho(T)\tilde V\xi}{(M_d\otimes I_K) W\eta}
= \ip{\tilde\rho(T)\tilde V\xi}{W(M_d\otimes I_H)\eta}\!,\]
for all $\xi \in L^2(X,H)$ and all $\eta\in L^2(Y,H)$.
We conclude that
$$(I_{L^2(Y)}\otimes E)(M_d\otimes I_K) W=(I_{L^2(Y)}\otimes E)W(M_d\otimes I_H)$$
and hence
$(M_d\otimes I_K)W = W(M_d\otimes I_H)$ for all $d\in L^{\infty}(Y)$.
It follows easily that $W\in L^{\infty}(Y,\cl B(H,K))$ (see \cite[Theorem 7.10]{takesaki1}).
Let now
$$\cl T\stackrel{\text{def}}{=}
\overline{[\{TWL^2(Y,H) : T\in \cl K(L^2(Y))\otimes \rho(A)\}]}.$$
The projection onto $\cl T$
has the form  $I_{L^2(Y)}\otimes F$ for
some projection $F\in \rho(A)'$.
Letting
$V=(I_{L^2(X)}\otimes F)\tilde V$, and using similar arguments to the ones above, one shows that
$(M_c\otimes I_K)V=V(M_c\otimes I_H)$ for all $c\in L^\infty(X)$ and hence
that $V\in L^{\infty}(X,\cl B(H,K))$.
Note that $W = (I_{L^2(Y)}\otimes F)W$ and hence, by (\ref{eq_wvti}),
\begin{equation}\label{eq_vw}
S_{\nph}(T) =
W^*(I_{L^2(Y)}\otimes F)\tilde{\rho}(T) \tilde{V} = W^* \tilde{\rho}(T)(I_{L^2(Y)}\otimes F) \tilde{V}
= W^*\tilde{\rho}(T) V,
\end{equation}
for every $T\in\cl K\otimes A$.

Let $k\in L^2(Y\times X)$ and $a\in A$. For $\xi\in L^2(X,H)$ and
$\eta\in L^2(Y,H)$ we have
\begin{equation}\label{eqresult1}
\ip{S_{\varphi}(T_k\otimes a)\xi}{\eta} = \int_Y\int_X k(y,x)\langle \varphi(x,y)(a)\xi(x),\eta(y)\rangle d\mu(x)d\nu(y).
\end{equation}
On the other hand, by (\ref{eq_vw}),
\begin{eqnarray*}
\langle S_{\nph}(T_k\otimes a)\xi,\eta \rangle
& = &
\langle W^*(T_k\otimes\rho(a))V\xi,\eta\rangle\\
& = &
\langle (T_k\otimes\rho(a))V\xi,W\eta\rangle\\
& = &
\int_Y\int_X k(y,x) \langle \rho(a)(V(x)\xi(x)),W(y)\eta(y)\rangle d\mu(x) d\nu(y)\\
& = &
\int_Y\int_X k(y,x) \langle W(y)^*\rho(a)V(x)\xi(x),\eta(y)\rangle d\mu(x) d\nu(y).
\end{eqnarray*}
Comparing the last identity with (\ref{eqresult1}) and taking into account that these
identities hold for all $k\in L^2(Y\times X)$, we conclude that
\begin{equation}\label{eq_inpi}
\langle \varphi(x,y)(a)\xi(x),\eta(y)\rangle = \langle W(y)^*\rho(a)V(x)\xi(x),\eta(y)\rangle \ \
\mbox{almost everywhere},
\end{equation}
for all $\xi\in L^2(X,H)$ and all $\eta\in L^2(Y,H)$.
If the measures $\mu$ and $\nu$ are finite, take $\xi = \chi_X\otimes \xi_0$ and
$\eta = \chi_Y \otimes \eta_0$, where $\xi_0,\eta_0 \in H$.
The separability of $H$ and (\ref{eq_inpi}) imply that
$$\varphi(x,y)(a) = W(y)^*\rho(a)V(x), \ \mbox{ for almost all } (x,y)\in X\times Y.$$
If the measures $\mu$ and $\nu$ are not finite, the proof is completed by choosing
increasing sequences $(X_n)_{n\in \bb{N}}$ and $(Y_n)_{n\in \bb{N}}$, each of whose terms
has finite measure, and letting
$\xi = \chi_{X_n}\otimes \xi_0$ and
$\eta = \chi_{Y_n} \otimes \eta_0$, with $\xi_0,\eta_0 \in H$.

(ii)$\Rightarrow$(i)
The assumption shows that the mapping
$S_{\nph} : \cl S_2(Y\times X,A)\to \cl S_2(Y\times X,\cl B(H))$
satisfies
$$S_{\nph}(T_h\otimes a) = W^*(T_h\otimes \rho(a))V, \ \ h\in L^2(Y\times X), a\in A.$$
By linearity,
\begin{equation}\label{eq_simple}
S_{\nph}(T_k) = W^* T_{\rho\circ k}V,
\end{equation}
whenever $k\in L^2(Y\times X)\odot A$.

Let $k\in L^2(Y\times X,A)$ be arbitrary. By \cite[Proposition 7.4]{takesaki1}, there exists a sequence
$(k_i)_{i\in \bb{N}}\subseteq L^2(Y\times X)\odot A$
with $\|k_i - k\|_2\to_{i\to \infty} 0$.
Using (\ref{eq_otthe}), (\ref{eq_simple}), Lemma \ref{l_Tk} and the fact that $\nph$ is bounded, we obtain
\begin{eqnarray*}
S_{\nph}(T_k)
& = &
\lim_{i\to\infty} S_{\nph}(T_{k_i}) = \lim_{i\to\infty} W^* T_{\rho\circ k_i} V\\
& = &
W^* (\lim_{i\to\infty} T_{\rho\circ k_i})V =
W^* (\lim_{i\to\infty} \tilde{\rho}(T_{k_i})V =
W^*\tilde{\rho}(T_k)V.
\end{eqnarray*}
Thus, $S_{\nph}$ has a completely bounded extension to $\cl K\otimes \cl A$
(namely, the map $T\to W^*\tilde{\rho}(T))V$)
and hence $\nph$ is a Schur $A$-multiplier.
\end{proof}

\noindent {\bf Remarks (i) } The proof of Theorem \ref{th_chschura} shows that
if $\nph\in \frak{S}(X,Y;A)$ then the operator valued functions $V$ and $W$
can be chosen so that
$$\|\nph\|_{\frak{S}} = \esssup_{x\in X} \|W(x)\|\esssup_{y\in Y} \|V(y)\|.$$

\smallskip

{\bf (ii) } In the case $A = \bb{C}$, Theorem \ref{th_chschura} reduces to the
well-known characterisation of measurable Schur multipliers due to
U. Haagerup \cite{haag} and V. V. Peller \cite{peller} (see also \cite{kp}).
Indeed, in this case, $\rho$ is equal to the identity representation of $\bb{C}$
and hence $\nph$ has a representation of the form
\begin{equation}\label{eq_wvh}
\nph(x,y) = \langle w(y),v(x)\rangle,
\end{equation}
where $v : X\to K$ and $w : Y\to K$ are weakly measurable essentially bounded
functions, for some separable Hilbert space $K$.

\medskip

If, in Theorem \ref{th_chschura}, the operator valued functions $V$ and $W$
can be chosen to be weakly measurable, then we will say that the Schur $A$-multiplier
$\nph$ has a \emph{weakly measurable representation}.
In the next theorem
we exhibit a class of $A$-valued Schur $A$-multipliers possessing
a weakly measurable representation which exhausts all
such multipliers in the case $A$ is finite dimensional.
Recall that a Hilbert $A$-bimodule is a right Hilbert $A$-module $\cl N$,
equipped with a left $A$-module action given by
$a\cdot \xi \stackrel{def}{=} \theta(a)(\xi)$, $a\in A$, $\xi\in \cl N$,
for some *-representation $\theta$ of $A$ into the C*-algebra of all
adjointable operators on $\cl N$. As is customary in the literature on Hilbert modules,
we assume linearity on the second variable of the $A$-valued inner product,
denoted here by $\vip{\cdot}{\cdot}_A$.

\begin{theorem}\label{c_hilmodc}
Let $A\subseteq\Bd$ be a separable C*-algebra and $\varphi : X\times Y\to CB(A)$
be a bounded pointwise measurable function. Consider the conditions:

(i) there exists a countably generated Hilbert $A$-bimodule $\mathcal{N}$ and bounded
weakly measurable functions
$v : X\to\mathcal{N}$ and $w : Y\to\mathcal{N}$ such that
\begin{equation}\label{eq_spo}
\varphi(x,y)(a) = \vip{w(y)}{a\cdot v(x)}_A,\ \ \mbox{for almost all } (x,y)\in X\times Y
\end{equation}
for every $a\in A$.

(ii) $\varphi$ is a Schur $A$-multiplier possessing a weakly measurable representation.

\noindent
Then $(i)\Rightarrow (ii)$. If $A$ is finite-dimensional then $(i)\Leftrightarrow (ii)$.
\end{theorem}
\begin{proof}
(i)$\Rightarrow$(ii)
It follows for instance from  \cite[Example 2.8]{EKQR} that
there exist a separable
Hilbert space $K$, an isometry $\tau : \cl N\to \cl B(K)$
and a faithful *-representation $\pi : A\to \cl B(K)$ such that
$\tau(a\cdot z) = \pi(a)\tau(z)$, $\tau(z\cdot b) = \tau(z)\pi(b)$ and
$\pi(\langle z_1,z_2\rangle_A) = \tau(z_1)^*\tau(z_2)$, $z,z_1,z_2\in \cl N$, $a,b\in A$.
For all $a\in A$, we have that
$$\pi(\nph(x,y)(a)) =
\pi(\vip{w(x)}{a\cdot v(y)}_A) = \tau(w(x))^*\pi(a)\tau(v(y))  \ \mbox{ a.e.}$$
Moreover, the maps $\tau\circ v$ and $\tau\circ w$ are weakly measurable.
By Theorem \ref{th_chschura}, the map
$\nph_{\pi} : X\times Y\to CB(\pi(A))$, given by
$\nph_{\pi}(x,y)(\pi(a)) = \pi(\nph(x,y)(a))$, is a Schur $\pi(A)$-multiplier.
By Proposition \ref{p_indep}, $\nph$ is a Schur $A$-multiplier.

Assume now that $A$ is finite dimensional.
By Proposition \ref{p_indep}, we may identify $A$ with
the C*-algebra $\oplus_{k=1}^m M_{n_k}\subseteq \cl B(H)$,
where $M_n$ denotes, as customary, the $n$ by $n$ matrix algebra and $H= \oplus_{k=1}^m \mathbb C^{n_k}$, $n_k\in\mathbb N$, $k=1,\ldots,m$.

 Suppose that $\varphi$ is a Schur $A$-multiplier,
 $K$ is a separable Hilbert space,
$V : X\to \cl B(H,K)$, $W : Y\to \cl B(H,K)$
weakly measurable functions, and $\rho : A \to \cl B(K)$ a non-degenerate *-representation,
such that, for every $a\in A$, we have
$\varphi(x,y)(a) = W(y)^*\rho(a)V(x)$ for almost all $(x,y)\in X\times Y$.
The space $\cl B(H,K)$ is an operator $A$-bimodule
with respect to the actions
$a\cdot T \stackrel{\text{def}}{=}\rho(a)T$ and
$T\cdot a \stackrel{\text{def}}{=} Ta$, $a\in A,\ T\in\cl B(H,K)$.
Let $P_{k}$ be the projection in $\cl B(H)$ onto the
summand $\mathbb C^{n_k}$ and $\Psi(T)=\sum_{k=1}^mP_kTP_k$, $T\in \cl B(H)$.
Clearly, $\Psi$ is a completely positive projection from $\cl B(H)$ onto $A$.
We equip $\cl B(H,K)$ with the $A$-valued inner product given by
$\langle S,T\rangle_{\!A} = \Psi(S^*T)$.
As the projections $P_k$, $k = 1,\dots,m$, are mutually orthogonal
and $\sum_{k=1}^m P_k=I$, we have $\langle S,S\rangle_{\!A}=0$ if and only if $S=0$.
Moreover,
$$\langle S, T\cdot  a\rangle_{\!A} = \Psi(S^*Ta)=\Psi(S^*T)a =
\langle S, T\rangle_{\!A} a,$$
and hence $\cl N\stackrel{\text{def}}{=} \cl B(H,K)$ is a right Hilbert $A$-module.
In addition,
$$\langle a\cdot S, T\rangle_{\!A} = \Psi(S^*\rho(a)^*T) = \langle S, a^*\cdot T\rangle_{\!A},$$
showing that the map $\theta_a : S \to a\cdot S$ is adjointable and that
the map $a\to \theta_a$ is a *-representation;
thus, $\cl N$ is a Hilbert $A$-bimodule and $\varphi(x,y)(a)=\langle W(x), a\cdot V(x)\rangle_{\!A}$ for almost all $(x,y)\in X\times Y$.
As $H$ is finite dimensional and $K$ is separable, $\cl N$ is countably generated.
\end{proof}

Let $\cl B = \cl B(L^2(X),L^2(Y))$ for brevity.

\begin{proposition}\label{p_ext}
If $\nph\in \frak{S}(X,Y;A)$ then the map $S_{\nph}$ has a unique extension to a completely bounded
weak* continuous map from $\cl B\bar{\otimes} A^{**}$ into $\cl B\bar\otimes \cl B(H)$.
\end{proposition}
\begin{proof}
Let $P : \cl B(H)^{**}\to \cl B(H)$ be the canonical projection (that is, the adjoint of the
inclusion map of the trace class on $H$ into $\cl B(H)^*$). Then the map
$\id \otimes P : \cl B\bar\otimes \cl B(H)^{**}\to \cl B\bar\otimes \cl B(H)$ is weak* continuous and completely contractive
(see {\it e.g. } \cite{blm} and \cite[Proposition 7.1.6]{er}).

Let $\tilde{\cl K} = \cl K(L^2(Y)\oplus L^2(X))$ and $\tilde{\cl B} = \cl B(L^2(Y)\oplus L^2(X))$.
Write $P_X$ (resp. $P_Y$) for the projection from $L^2(Y)\oplus L^2(X)$ onto
$L^2(X)$ (resp. $L^2(Y)$).
Let $\Phi:\tilde{\cl K}\otimes A\to \tilde{\cl K}\otimes \cl B(H)$ be given by
\begin{equation}\label{eq_tbt2}
\Phi\left(\left(\begin{matrix} x_{1,1} & x_{1,2}\\ x_{2,1} & x_{2,2}\end{matrix}\right)\otimes a\right) =
\left(\begin{matrix} 0 & S_{\nph}(x_{1,2}\otimes a)\\ 0 & 0\end{matrix}\right).
\end{equation}
By \cite[Example 1]{huruya}, given a $C^*$-algebra $B$, there is a canonical normal *-isomorphism
\begin{equation}\label{eq_ksecd}
 (\tilde{\cl K}\otimes B)^{**} \cong \tilde{\cl B}\bar\otimes B^{**}.
\end{equation}
  Hence we may view the second dual $\Phi^{**}$ as a completely bounded map from
$\tilde{\cl B}\bar\otimes A^{**}$ to $\tilde{\cl B}\bar\otimes \cl B(H)^{**}$, extending $\Phi$.
As $\tilde{\cl K}\otimes A$ is weak* dense in $\tilde{\cl B}\bar\otimes A^{**}$ we have that for any $T\in \tilde{\cl B}\bar\otimes A^{**}$ there exists $\Psi(T)\in \cl B\bar\otimes \cl B(H)^{**}$ such that
$$\Phi^{**}(T)=\left(\begin{matrix} 0 & \Psi(T)\\ 0 & 0\end{matrix}\right).$$
In particular,
$$\Phi^{**}((P_Y\otimes \id)T(P_X\otimes \id))=\left(\begin{matrix} 0 & \Psi((P_Y\otimes\id)T(P_X\otimes\id))\\ 0 & 0\end{matrix}\right),$$
and the mapping $\tilde \Psi =
\Psi|_{\cl B\bar\otimes A^{**}}:\cl B\bar\otimes A^{**}\to \cl B\bar\otimes \cl B(H)^{**}$
is completely bounded and weak* continuous.
Hence the composition
$$(\id \otimes P)\circ \tilde\Psi : \cl B\bar\otimes A^{**} \to \cl B\bar\otimes \cl B(H)$$
is a completely bounded weak* continuous map, extending $S_{\nph}$.
The fact that
this extension is unique follows by weak* density.
\end{proof}

We will use the same symbol, $S_{\nph}$, to denote the map obtained in Proposition \ref{p_ext}.
We note that if $S_{\nph}$ satisfies equation (\ref{eq_vw}), that is, if
$S_\nph (S)=W^*(\id\otimes\rho)(S)V$ for all $S\in \cl K\otimes  A$,
then $S_{\nph}(T\otimes a) = W^*(T\otimes\rho(a))V$ for all $T\in\cl B$ and all
$a\in A^{**}$, where $\rho$ has been canonically extended to $A^{**}$.

While Proposition \ref{p_ext} implies that, if $\nph\in \frak{S}_0(X,Y;A)$, then
the map $S_{\nph}$ on $\cl K\otimes  A$ has a weak* continuous extension to
$\cl B\otimes  A^{**}$, an analogous extension is not guaranteed to exist in
representations of $A$ different from the universal one.
This motivates the following definition.

\begin{definition}\label{d_conthe}
Let $A$ be a separable C*-algebra and $\theta$ be a faithful *-representation
of $A$ on a separable Hilbert space. An element $\nph\in \frak{S}_0(X,Y;A)$
will be called a \emph{Schur $\theta$-multiplier} if the map
$S_{\nph_{\theta}} : \cl K\otimes\theta(A) \to \cl K\otimes\theta(A)$
can be extended to a weak* continuous map
on $\cl B\bar\otimes\theta(A)''$.
\end{definition}

The notion of a Schur $\theta$-multiplier will be used in the subsequent sections.

\section{Herz-Schur multipliers and transference}\label{s_ggamma}

In this section, we introduce and study Herz-Schur multipliers of crossed products.
We assume throughout that
$G$ is a locally compact group. Left Haar measure on $G$ will be denoted by $m_G$ or $m$ and
integration with respect to $m_G$
along the variable $s$ will be denoted by $ds$.
Let $\lambda^G : G\to \cl B(L^2(G))$ be the left regular representation of $G$;
thus, $\lambda^G_t\xi(s) = \xi(t^{-1}s)$, $\xi\in L^2(G)$, $s,t\in G$.
We write  $C^*_r(G)$ (resp. $\vn(G)$)
for the reduced group C*-algebra (resp. the von Neumann algebra) of $G$,
that is, for the closure in the norm topology (resp. in the weak* topology)
of $\lambda^G(L^1(G))$.
As customary, we let $A(G)$ (resp. $B(G)$, $B_{\lambda}(G)$) be the Fourier
(resp. the Fourier-Stieltjes, the reduced Fourier-Stieltjes) algebra of $G$.
We note the canonical identifications $A(G)^* = \vn(G)$, $C^*(G)^* = B(G)$
and $C_r^*(G)^* = B_{\lambda}(G)$ \cite{eym}.

Let $A$ be a separable C*-algebra. In this section, unless otherwise stated,
$H$ will denote the Hilbert space of the
universal representation of $A$; we consider $A$ as a C*-subalgebra of $\cl B(H)$.
Let $\alpha : G\rightarrow {\rm Aut}(A)$
be a continuous (with respect to point-norm topology) group homomorphism;
thus, $(A,G,\alpha)$ is a C*-dynamical system.
The space $L^1(G,A)$ is a *-algebra
with respect to the product $\times$ given by
$(f\times g) (t) = \int_{G} f(s) \alpha_s(g(s^{-1}t))ds$ and
the involution $*$ given by $f^*(s) = \Delta(s)^{-1} \alpha_s(f(s^{-1})^*)$.

Let $\pi : A\rightarrow \cl B(L^2(G,H))$ be the *-representation defined by
$(\pi(a)\xi)(t) = \alpha_{t^{-1}}(a)(\xi(t))$, $t\in G$, and $\lambda : G\rightarrow \cl B(L^2(G,H))$ be the
(continuous) unitary representation given by $(\lambda_t\xi)(s) = \xi(t^{-1}s)$, $s,t\in G$.
Note that
$$\pi(\alpha_t(a)) = \lambda_t \pi(a) \lambda_t^*, \ \ t\in G;$$
thus, the pair $(\pi,\lambda)$ is a covariant representation of $(A,G,\alpha)$ and
hence gives rise to a *-representation
$\pi\rtimes \lambda : L^1(G,A)\rightarrow \cl B(L^2(G,H))$ given by
$$(\pi\rtimes \lambda) (f) = \int_G \pi(f(s))\lambda_s ds, \ \ \ f\in L^1(G,A).$$
The \emph{reduced crossed product}
$A\rtimes_{\alpha,r} G$ of $A$ by $\alpha$ is, by definition,
the closure of $(\pi\rtimes \lambda)(L^1(G,A))$ in the operator norm of $\cl B(L^2(G,H))$ \cite[7.7.4]{ped2}.
We let $A\rtimes_{\alpha,r}^{w^*} G$ be the weak* closure of $A\rtimes_{\alpha,r} G$.

A bounded function $F : G\rightarrow \cl B(A)$ will be called \emph{pointwise measurable}
if, for every $a\in A$, the map $s\to F(s)(a)$ is a weakly measurable function from $G$ into $A$.
Suppose that $(\rho,\tau)$ is a covariant representation of
the dynamical system $(A,G,\alpha)$ on the Hilbert space $K$.
We say that $F$ is \emph{$(\rho,\tau)$-fiber continuous},
if the map
$$G\to \cl B(K), \ \ \ \ s\to \rho(F(s)(a))\tau_s,$$
is weak* continuous for every $a\in A$.
We will say that $F$ is \emph{fiber continuous} if
$F$ is $(\pi,\lambda)$-fiber continuous.
Note that if $F$ is bounded and point norm continuous then it is pointwise measurable
and fiber continuous.

We further say that $F$ is \emph{almost $(\rho,\tau)$-fiber continuous}
if, for every $\omega\in \cl B(K)_*$ and every $a\in A$, the function
$$s\to \langle \rho(F(s)(a))\tau_s, \omega\rangle$$
coincides, up to a null set, with a continuous function.
Almost $(\pi,\lambda)$-fiber continuous functions will be referred to simply
as \emph{almost fiber continuous}.

For each $f\in L^1(G,A)$, let
$F\cdot f \in L^1(G,A)$ be the function
given by $(F\cdot f) (s) = F(s)(f(s))$, $s\in G$.
It is easy to see that if $F$ is pointwise measurable then
$F\cdot f$ is weakly measurable
and hence $F\cdot f\in L^1(G,A)$ for every $f\in L^1(G,A)$;
in fact, $\|F\cdot f\|_1\leq \|F\|_{\infty} \|f\|_1$, where
$\|F\|_{\infty} = \sup_{s\in G}\|F(s)\|$.

\begin{definition}\label{d_hsmds}
A pointwise measurable function
$F : G\rightarrow CB(A)$
will be called a \emph{Herz-Schur $(A,G,\alpha)$-multiplier} if the map
$$S_F : (\pi\rtimes \lambda)(L^1(G,A)) \to (\pi\rtimes \lambda)(L^1(G,A))$$
given by
$$S_F((\pi\rtimes \lambda)(f)) = (\pi\rtimes \lambda)(F\cdot f)$$
is completely bounded.
\end{definition}

We denote by $\frak{S}(A,G,\alpha)$ the set of all Herz-Schur $(A,G,\alpha)$-multipliers.
If $F\in \frak{S}(A,G,\alpha)$ then the map $S_F$
extends to a completely bounded map on $A\rtimes_{r,\alpha} G$.
This (unique) extension will be denoted again by $S_F$.
We let $\|F\|_{\mm} = \|S_F\|_{\cb}$.

\medskip

\begin{remark}\label{herz_schur_ind} \rm
{\bf (i)} If $F_1, F_2\in \frak{S}(A,G,\alpha)$,
letting $F_1 + F_2 : G\to CB(A)$ and $F_1 F_2 : G\to CB(A)$ be given by
$(F_1 + F_2)(s) = F_1(s) + F_2(s)$ and $(F_1 F_2)(s) = F_1(s) \circ F_2(s)$, we see that
$S_{F_1 + F_2} = S_{F_1} + S_{F_2}$
and $S_{F_1 F_2} = S_{F_1} S_{F_2}$. Thus, $\frak{S}(A,G,\alpha)$
is an algebra with respect to the operations just defined.

\smallskip

\noindent {\bf (ii) }
Recall that a bounded continuous function $u : G\to \bb{C}$ is called
a \emph{completely bounded} (or \emph{Herz-Schur}) \emph{multiplier} \cite{ch}
of the Fourier algebra $A(G)$ of $G$ if
$uv\in A(G)$ for every $v\in A(G)$, and the map $m_u : v\to uv$ on $A(G)$ is
completely bounded.
The space of all Herz-Schur multipliers of $A(G)$ will be denoted as usual by $M^{\rm cb}A(G)$.
If $u\in M^{\rm cb}A(G)$ then
the dual $S_u$ of $m_u$ is a completely bounded (and weak* continuous)
linear map on the von Neumann algebra $\vn(G)$ of $G$, such that
$S_u(\lambda_t^G) = u(t)\lambda_t^G$, $t\in G$.
Moreover, $S_u$ leaves the reduced C*-algebra $C^*_r(G)$ of $G$ invariant, and
$$S_u\left(\int_G f(s)\lambda_s ds\right) = \int_G u(s) f(s)\lambda_s ds, \ \ f\in L^1(G).$$
The reduced crossed product of $\bb{C}$ by
the (unique) action $\alpha$ of a locally compact group $G$ on $\bb{C}$
coincides with $C^*_r(G)$. Identifying $\cl B(\bb{C})$ with $\bb{C}$ in the natural way,
we have that a bounded continuous function $u : G\to \bb{C}$ is a
Herz-Schur $(\bb{C},G,\alpha)$-multiplier if and only if $u$ is a Herz-Schur multiplier.

\smallskip

\noindent {\bf (iii) }
Suppose that $\theta : A\to \cl B(K)$ is a faithful *-representation of $A$ on a Hilbert space $K$.
Let $\pi^{\theta} : A\to \cl B(L^2(G,K))$ be given by
$(\pi^{\theta}(a)\xi)(t) = \theta(\alpha_{t^{-1}}(a))(\xi(t))$, $t\in G$,
while $\lambda^{\theta} : G\rightarrow \cl B(L^2(G,K))$ be given by
$(\lambda^{\theta}_t\xi)(s) = \xi(t^{-1}s)$, $s,t\in G$.
Then the pair $(\pi^{\theta},\lambda^{\theta})$ is a covariant representation of $(A,G,\alpha)$.
Since $A$ is assumed to be universally represented, up to a *-isomorphism,
$K$ is a closed subspace of $H$ that reduces $A$,
$\pi^{\theta}(a)$ is the restriction of $\pi(a)$ to $L^2(G,K)$,
while $\lambda^{\theta}_s$ is the restriction of $\lambda_s$ to $L^2(G,K)$.
In the sequel, we let
$A\rtimes_{\alpha,\theta} G = (\pi^{\theta}\rtimes \lambda^{\theta})(A\rtimes_{\alpha} G)$
and
$A\rtimes_{\alpha,\theta}^{w^*} G = \overline{A\rtimes_{\alpha,\theta} G}^{w^*}$.

By \cite[Theorem 7.7.5]{ped2}, the closure of
$(\pi^{\theta}\rtimes \lambda^{\theta})( L^1(G,A))$ is *-isomorphic to
$A\rtimes_{r,\alpha} G$ and a pointwise measurable function
$F : G\to CB(A)$ is a Herz-Schur $(A,G,\alpha)$-multiplier if and only if
the map
\begin{equation}\label{eq_Stheta}
S_{F}^{\theta} : (\pi^{\theta}\rtimes \lambda^{\theta})(f) \mapsto (\pi^{\theta}\rtimes \lambda^{\theta})(F\cdot f)
\end{equation}
is completely bounded. Thus, Herz-Schur $(A,G,\alpha)$-multipliers
can be defined starting with any faithful representation of $A$
instead of its universal representation.
\end{remark}

In the case $A = \bb{C}$, the maps on $C^*_r(G)$ associated with Herz-Schur multipliers
automatically have a weak* continuous extension to (completely bounded) maps on
the weak* closure $\vn(G)$ of $C^*_r(G)$. Such extension is not ensured to
exist in the general case -- this motivates the following definition.

\begin{definition}\label{d_wsthe}
Let $A$ be a separable C*-algebra, $K$ be a Hilbert space and
$\theta : A\to \cl B(K)$ be a faithful *-representation. A function
$F:G\to CB(A)$  will be called a
\emph{$\theta$-multiplier} if the map
$$\Phi_F^\theta : \pi^{\theta}(a)\lambda^{\theta}_t \mapsto  \pi^{\theta}(F(t)(a))\lambda^{\theta}_t, \ t\in G, \ a\in A,$$
has an extension to a bounded
weak* continuous map on $A\rtimes_{\alpha,\theta}^{w^*} G$.

A $\theta$-multiplier $F$ will be called a
\emph{Herz-Schur $\theta$-multiplier}
if the extension of $\Phi_F^\theta$ to $A\rtimes_{\alpha,\theta}^{w^*} G$
is completely bounded.
\end{definition}

We note that, in Definition \ref{d_wsthe}, we do not require the pointwise measurability of
the function $F$.
The weak* continuous extension of the map $\Phi_F^\theta$ therein will still be denoted
by the same symbol.

\begin{remark}\label{r_extl1}
{\rm Let $A$ be a separable C*-algebra, $K$ be a Hilbert space and
$\theta : A\to \cl B(K)$ be a faithful *-representation.
Suppose that
$F : G\to \cl B(A)$ is a
bounded  map and
$\Phi : A\rtimes_{\alpha,\theta}^{w^*} G \to A\rtimes_{\alpha,\theta}^{w^*} G$
is a bounded weak* continuous map such that, for almost all $t\in G$,
$$\Phi(\pi^{\theta}(a)\lambda^{\theta}_t) = \pi^{\theta}(F(t)(a))\lambda^{\theta}_t, \ \ a\in A.$$
Then, for any $\omega\in \cl B(L^2(G,K))_*$ and $f\in L^1(G,A)$,
the function
$s\mapsto \langle\pi^\theta(F(s)(f(s)))\lambda_s^\theta,\omega\rangle$ is measurable, and
$$\Phi((\pi^{\theta}\rtimes\lambda^{\theta})(f)) = (\pi^{\theta}\rtimes\lambda^{\theta})(F\cdot f),
\ \ \ f\in L^1(G,A).$$}
\end{remark}
\begin{proof}
Let $\omega\in \cl B(L^2(G,K))_*$ and $f\in L^1(G,A)$.
Since the function
$s\mapsto \langle\pi^\theta(f(s))\lambda_s^\theta,\Phi_*(\omega)\rangle$ is  measurable so is $s\mapsto \langle\pi^\theta(F(s)(f(s)))\lambda_s^\theta,\omega\rangle$. We have
\begin{eqnarray*}
\left\langle \Phi\left(\int \pi^{\theta}(f(s))\lambda_s^{\theta}ds\right), \omega\right\rangle
& = &
\left\langle \int \pi^{\theta}(f(s))\lambda_s^{\theta}ds, \Phi_*(\omega)\right\rangle\\
& = &
\int  \langle \pi^{\theta}(f(s))\lambda_s^{\theta},\Phi_*(\omega)\rangle ds\\
& = &
\int  \langle \Phi(\pi^{\theta}(f(s))\lambda_s^{\theta}),\omega \rangle ds\\
& = &
\int  \langle \pi^{\theta}(F(s)(f(s)))\lambda_s^{\theta},\omega \rangle ds\\
& = &
\left\langle \int \pi^{\theta}(F(s)(f(s)))\lambda_s^{\theta} ds,\omega \right\rangle\\
& = &
\langle(\pi^{\theta}\rtimes\lambda^{\theta})(F\cdot f),\omega\rangle.
\end{eqnarray*}
The claim follows.
\end{proof}

\begin{lemma}\label{blambda}
Let $\theta$ be a faithful *-representation of $A$ on a Hilbert space
$K$. Let $F : G\to CB(A)$ be a pointwise measurable map for which there exists $C>0$ such that
\begin{equation}\label{eq_pithe}
\|(\pi^\theta\rtimes\lambda^\theta)(F\cdot f)\|\leq C\|(\pi^\theta\rtimes\lambda^\theta)(f)\|,
\ \ \ f\in L^1(G,A).
\end{equation}
For $\omega\in \cl B(L^2(G,K))_*$, let
$g_{\omega}(s) = \langle\pi^\theta(F(s)(a))\lambda_s^\theta,\omega\rangle$, $s\in G$.
Then $g_{\omega}$ coincides with an element of $B_\lambda(G)$ up to a null set.
\end{lemma}
\begin{proof}
Let $\tilde{\pi}^\theta$ be the *-representation of $A$ on $L^2(G,K)\otimes L^2(G)$
given by
$\tilde\pi^\theta(a)=\pi^\theta(a)\otimes I_{L^2(G)}$, $a\in A$.
It is easy to see that the pair $(\tilde\pi^\theta,\lambda^\theta\otimes\lambda^G)$ is a covariant representation.
We first establish the following:

\medskip

\noindent {\bf Claim. }
{\it The representation
$\tilde\pi^\theta\rtimes(\lambda^\theta\otimes\lambda^G)$ is unitarily equivalent to a direct sum of copies of the representation $\pi^\theta\rtimes\lambda^\theta$.}

\medskip

\noindent {\it Proof of the Claim. }
Let $U_1: L^2(G)\otimes L^2(G)\to L^2(G,L^2(G))$ be the unitary operator given by
$$U_1(\xi\otimes \eta)(s)=\xi(s)\lambda_{s^{-1}}^G\eta, \ \ \xi,\eta\in L^2(G).$$
Let $(\eta_i)_{i\in I}$ be an orthonormal basis for $L^2(G)$
and let $L^2(G)^I$ be the direct sum of $I$-copies of $L^2(G)$.
Let $U_2 : L^2(G, L^2(G))\to L^2(G)^I$ be the operator given by
$$U_2f=(g_i)_{i\in I}, \ \mbox{ where } g_i(s) = \langle f(s),\eta_i\rangle, s\in G.$$
It is easy to see that $U_2$ is a unitary operator.
Set $U=U_2U_1$; thus, $U$ is a unitary operator from $L^2(G)\otimes L^2(G)$ onto  $L^2(G)^I$.
For an operator $T$ on $L^2(G)$,
we shall denote by $T^{\infty}$ its ampliation on $L^2(G)^I$, given by
$$T^{\infty}(\xi_i)_{i\in I} = (T\xi_i)_{i\in I}.$$
In what follows we will use some natural identifications:
of $K\otimes L^2(G)\otimes L^2(G)$ with $L^2(G,K)\otimes L^2(G)$,
of $L^2(G,K)^I$ with $L^2(G,K^I)$,
and of $K\otimes L^2(G)^I$ with $L^2(G,K)^I$,
the latter being
given by $\xi\otimes(\xi_i)_{\in I}\to(\tilde \xi_i)_{i\in I}$, where $\tilde\xi_i(s)=\xi_i(s)\xi$

Let $\xi_1 \in K$, $\xi_2, \xi_3\in L^2(G)$,
$a\in A$ and $t\in G$.
Let $\zeta_i$ be the function given by $\zeta_i(s) = \xi_2(s) \langle \xi_3,\lambda_s^G\eta_i\rangle$, $s\in G$.
For almost all $s\in G$ we have
\begin{eqnarray*}
&&(I_K\otimes U)\tilde\pi^\theta(a)(\xi_1\otimes\lambda_t^G\xi_2\otimes\lambda_t^G\xi_3)(s)\\
&&=(\theta(\alpha_{s^{-1}}(a))\xi_1(\lambda_t^G\xi_2)(s)\langle\lambda_{s^{-1}}^G\lambda_t^G\xi_3,\eta_i\rangle)_{i\in I}\\
&&=
(\theta(\alpha_{s^{-1}}(a))\xi_1(\lambda_t^G\xi_2)(s)\langle\xi_3,\lambda_{t^{-1}s}^G\eta_i\rangle)_{i\in I}\\
&&=(\theta(\alpha_{s^{-1}}(a))\xi_1 \lambda_t^G\zeta_i(s))_{i\in I}\\
&&= \pi^\theta(a)^{\infty}(\lambda_t^\theta(\xi_1\otimes\zeta_i)(s))_{i\in I}\\
&&= \pi^\theta(a)^{\infty}((\lambda_t^\theta\otimes I_{L^2(G)})(I_K\otimes U)\xi_1\otimes \xi_2\otimes\xi_3))(s).
\end{eqnarray*}
The calculations imply
$$(I_K\otimes U)(\tilde\pi^\theta\rtimes(\lambda^\theta\otimes\lambda^G))(f)
= (\pi^\theta\rtimes\lambda^\theta(f))^{\infty}(I_K\otimes U), \ \ f\in L^1(G,A),$$
and the Claim is proved.

\medskip

Fix $\omega\in \cl B(L^2(G,H))_*$ and write $g = g_{\omega}$.
Let $v\in B_\lambda(G)$ and $w$ be the linear functional on $\{\lambda^G(f):f\in L^1(G)\}$ given by
$$w(\lambda^G(f))=\int_G f(s)g(s)v(s)ds.$$
Let $f\in L^1(G)$, $a\in A$ and $\tilde f(s)=f(s)a$, $s\in G$; clearly, $\tilde f \in L^1(G,A)$.
Fix $\xi$, $\eta\in L^2(G)$  and let $\omega_{\xi,\eta}\in \cl B(L^2(G))^*$
be the vector functional given by
$\omega_{\xi,\eta}(T) = \langle T\xi,\eta\rangle$, $T\in \cl B(L^2(G))$. Then
\begin{eqnarray*}
\langle\lambda^G(fg)\xi,\eta\rangle&=& \int_G f(s)\langle\pi^\theta(F(s)(a))\lambda^\theta_s,\omega\rangle\langle\lambda_s^G\xi,\eta\rangle ds\\
&=& \left\langle\int_G( \pi^\theta(F(s)(\tilde f(s)))\otimes I_{L^2(G)})(\lambda^\theta_s\otimes\lambda_s^G)ds,\omega\otimes\omega_{\xi,\eta}\right\rangle\\
&=&\left\langle\int_G( \tilde\pi^\theta(F(s)(\tilde f(s))))(\lambda^\theta_s\otimes\lambda_s^G)ds,\omega\otimes\omega_{\xi,\eta}\right\rangle\\
&=&\langle\tilde\pi^\theta\rtimes(\lambda^\theta\otimes\lambda^G)(F\cdot\tilde f),\omega\otimes\omega_{\xi,\eta}\rangle.
 \end{eqnarray*}
Using the Claim, we have
$$\|\lambda^G(fg)\|\leq \|\omega\|\|\tilde\pi^\theta\rtimes(\lambda^\theta\otimes\lambda^G)(F\cdot\tilde f)\|
=\|\omega\|\|(\pi^\theta\rtimes\lambda^\theta)(F\cdot \tilde f)\|.$$
As $(\pi^\theta\rtimes\lambda^\theta)(F\cdot\tilde f)=\int f(s)\pi^\theta(F(s)(a))\lambda_s^\theta ds$,
using (\ref{eq_pithe}), we have
\begin{eqnarray*}
\|w(\lambda^G(f)\|&\leq&\|v\|_{B_\lambda(G)}\|\lambda^G(fg)\|
\leq C\|\omega\|\|v\|_{B_\lambda(G)}\|(\pi^\theta\rtimes\lambda^\theta)(\tilde f)\|\\
&\leq&  C\|\omega\|\|v\|_{B_\lambda(G)}\left\|\pi^\theta(a)\int f(s)\lambda_s^\theta ds\right\|\\
&\leq&
C\|\omega\|\|v\|_{B_\lambda(G)}\|\pi^\theta (a)\|\left\|\int f(s)\lambda_s^\theta ds\right\|\\
&\leq& C\|\omega\|\|v\|_{B_\lambda(G)}\|\pi^\theta(a)\|\|\lambda^\theta(f)\|\\
&=& C\|\omega\|\|v\|_{B_\lambda(G)}\|\pi^\theta(a)\|\|\lambda^G(f)\|.
\end{eqnarray*}

Hence, there exists $z\in B_\lambda(G)$ such that $\omega(\lambda^G(f))=\int f(s)z(s)ds$, $f\in L^1(G)$.
It follows that $z=vg$ almost everywhere. As this holds for any $v\in B_\lambda(G)$,
we have that $g$ is  almost everywhere equal to a  function from
$B_\lambda(G)$ (see \cite[Proposition 1.2]{ch}).
\end{proof}

\begin{remark}\label{newremark} \rm
For $\omega\in \cl B(L^2(G,K))_*$, let
$g_\omega(s)=\langle \pi^\theta(F(s)(a))\lambda_s^\theta,\omega\rangle$ and
let $b_\omega\in B_\lambda(G)$ be such that $b_\omega=g_\omega$ almost everywhere.
If $G$ is second countable and $K$ is separable,  then
under the assumption of the previous lemma
there exists a null subset $N\subseteq G$ such that
$g_\omega(t)=b_\omega(t)$ for any $t\in G\setminus N$ and $\omega\in \cl B(L^2(G,K))_*$.
Indeed, in this case we have that $\cl B(L^2(G,K)_*$ is separable. Let $\{\omega_n:n\in\mathbb N\}$ be a dense subset of $\cl B(L^2(G,K)_*$ and $a\in A$.
Let $N_n\subseteq G$ be a null set such that $g_{\omega_n}(s)=b_{\omega_n}(s)$ for all  $s\in G\setminus N_n$,
and $N=\cup_{n\in \mathbb N} N_n$.
Clearly, $N$ is a null set. If $\omega\in \cl B(L^2(G,K)_*$, let $\{\omega_{n(k)}\}_k$
be a subsequence converging to $\omega$ in norm.
Then $\{b_{\omega_{n(k)}}\}_k$ is a Cauchy sequence of bounded continuous functions:
letting $C = \|\omega\|\sup_{s\in G}\|F(s)\|$,
given $\varepsilon > 0$ there exists $L \in \bb{N}$
such that for any $l,k>L$, we have
$$|b_{\omega_{n(k)}}(t) - b_{\omega_{n(l)}}(t)|=|g_{\omega_{n(k)}}(t)-g_{\omega_{n(l)}}(t)|\leq C\|\omega_{n(k)}-\omega_{n(l)}\|\leq C\varepsilon,$$
whenever $t\in G\setminus N$.
As $b_{\omega_{n(k)}}$ is continuous, $| b_{\omega_{n(k)}}(t)-b_{\omega_{n(l)}}(t)|<C\varepsilon$ for every $t\in G$. Thus, the sequence
$\{b_{\omega_{n(k)}}\}_k$ converges to a continuous function, say $b$.
On the other hand, $b_{\omega_{n(k)}}(t)\to g_{\omega}(t)$ whenever $t\in G\setminus N$.
Therefore $g_{\omega}(t)=b(t)$ for $t\in G\setminus N$.
As $b_\omega$ and $b$ are continuous, and $b_\omega=g_\omega$ almost everywhere, we have  $b=b_\omega$, giving the statement.
\end{remark}

{\it For the rest of the section we will assume that
$G$ is a second countable locally compact group.}
In this case, the measure space $(G,m)$ is standard.

If $t\in G$, let us call a \emph{Dirac family at $t$} a
net $(f_U)_U\subseteq L^1(G)$ consisting of non-negative functions,
indexed by the directed set of all open neighbourhoods of $t$ with compact closure,
with $\supp f_U\subseteq U$ and $\|f_U\|_1 = 1$.

\begin{lemma}\label{l_point}
Let $(A,G,\alpha)$ be a C*-dynamical system, $F$ be a Herz-Schur $(A,G,\alpha)$-multiplier,
and $\theta$ be a faithful *-representation of $A$ on a separable Hilbert space $K$. Then there exists a null set
$N\subseteq G$ such that if $t\in G\setminus N$ and $(f_U)_{U}$ is a Dirac family at $t$ then
$$S_F^\theta((\pi^\theta\rtimes\lambda^\theta)(f_U\otimes a))\to_U \pi^\theta(F(t)(a))\lambda_t^\theta$$
in the weak* topology, for every $a\in A$.
\end{lemma}
\begin{proof}
Let $a\in A$.
By Lemma \ref{blambda} and Remark \ref{newremark} there exists a null set
$N\subseteq G$ such that for any $\xi$, $\eta\in L^2(G,H)$ there exists a continuous function $b_{\xi,\eta} : G\to \bb{C}$ such that
$$\langle \pi^\theta(F(s)(a))\lambda_s^\theta \xi,\eta\rangle = b_{\xi,\eta}(s) \ \ \
\mbox{ for  all } s\in G\setminus N.$$
Fix $\xi$, $\eta\in L^2(G,H)$. For $t\in G\setminus N$, set
$$C_U = \sup_{s\in U}|b_{\xi,\eta}(s) - b_{\xi,\eta}(t)|.$$
Since $b_{\xi,\eta}$ is continuous, $C_U\to_U 0$.
We have
\begin{eqnarray*}
& & \left|\langle S_F^\theta((\pi^\theta\rtimes\lambda^\theta)(f_U\otimes a))\xi,\eta\rangle - \langle\pi^\theta(F(t)(a))\lambda_t^\theta\xi,\eta\rangle\right|\\
& = &
\left|\int \langle \pi^\theta(F(s)(a))f_U(s)\lambda_s^\theta \xi,\eta\rangle ds
- \int \langle \pi^\theta(F(t)(a))f_U(s)\lambda_t^\theta\xi,\eta\rangle ds\right|\\
& \leq &
\int f_U(s) |\langle \pi^\theta(F(s)(a))\lambda_s^\theta \xi,\eta\rangle
- \langle \pi^\theta(F(t)(a))\lambda_t^\theta\xi,\eta\rangle| ds\\
& = &
\int f_U(s) |b_{\xi,\eta}(s) - b_{\xi,\eta}(t)|ds
\leq
C_U \int f_U(s)ds  = C_U \longrightarrow\mbox{}_U \ 0.
\end{eqnarray*}
The statement follows from the fact that
the weak operator topology and the weak* topology coincide on bounded sets.
\end{proof}

If $\nph : G\times G\to CB(A)$ is a bounded pointwise measurable function,
let $\cl T(\nph) : G\times G\to CB(A)$ be the function given by
$$\cl T(\nph)(s,t)(a) = \alpha_t(\nph(s,t)(\alpha_{t^{-1}}(a))), \ \ \ a\in A.$$
It is easy to see that, for each $a\in A$, the function $(s,t)\to \cl T(\nph)(s,t)(a)$
from $G\times G$ into $A$ is bounded and weakly measurable.
The inverse $\cl T^{-1}$ of $\cl T$ is given by
$\cl T^{-1}(\nph)(s,t)(a) = \alpha_{t^{-1}}(\nph(s,t)(\alpha_t(a)))$, $a\in A$.
For a map $F : G\to CB(A)$, let
$N(F):G\times G\to CB(A)$ be the function given by
$$N(F)(s,t) = F(ts^{-1}),$$
and
$$\cl N(F) = \cl T^{-1}(N(F));$$
thus,
$\cl N(F) : G\times G\to CB(A)$ is the function given by
$$\cl N(F)(s,t)(a) = \alpha_{t^{-1}}(F(ts^{-1})(\alpha_{t}(a))), \ \ \ a\in A, \ s,t\in G.$$
Note that if $F$ is pointwise measurable then so is $\cl N(F)$.

The next theorem is a dynamical system version of
the well-known description of Herz-Schur multipliers in
terms of Schur multipliers \cite{bf}.
Recall that, given a map $\nph : X\times Y\to CB(A)$ and a faithful *-representation
$\theta$ of $A$, we let $\nph_{\theta} : X\times Y\to CB(\theta(A))$ be the map given by
$\nph_{\theta}(x,y)(\theta(a)) = \theta(\nph(x,y)(a))$, $a\in A$.
Note that, if $\nph$ is pointwise measurable then so is $\nph_{\theta}$.

\begin{theorem}\label{th_tr}
Let $(A,G,\alpha)$ be a C*-dynamical system and $F : G\to CB(A)$ be a
pointwise measurable map. The following are equivalent:

(i) \ $F$ is a Herz-Schur $(A,G,\alpha)$-multiplier;

(ii) $\cl N(F)$ is a Schur $A$-multiplier.

\noindent Moreover, if (i) holds then $\|F\|_{\rm m} = \|\cl N(F)\|_{\frak{S}}$.
\end{theorem}
\begin{proof}
(i)$\Rightarrow$(ii)
Suppose that $F$ is a Herz-Schur $(A,G,\alpha)$-multiplier
and let $\theta$ be a faithful *-representation of $A$ on a separable Hilbert space $H'$.
By the Haagerup-Paulsen-Wittstock Theorem,
there exist a separable Hilbert space $K$, a *-representation $\rho : A\rtimes_{\alpha,\theta} G\to \cl B(K)$
and operators $V,W : L^2(G,H')\to K$ such that
\begin{equation}\label{eq_sf}
S_F^{\theta}(T) = W^*\rho(T) V, \ \ \ T\in A\rtimes_{\alpha,\theta} G,
\end{equation}
and $\|S_{F}\|_{\rm cb} = \|V\|\|W\|$.
Let $A\rtimes_{\alpha} G$ be the full crossed product of $A$ by $\alpha$,
$q : A\rtimes_{\alpha} G\to A\rtimes_{\alpha,\theta} G$ be the quotient map and
$\tilde{\rho} = \rho\circ q$; thus,
$\tilde{\rho} : A\rtimes_{\alpha} G \to \cl B(K)$ is a *-representation.
Let $\rho_A : A\to \cl B(K)$ be a  *-representation and
$\rho_G : G\to \cl B(K)$ be a strongly continuous unitary representation such that
$\tilde{\rho} = \rho_A\rtimes \rho_G$.
For $f\in L^1(G,A)$, we have
$$\tilde{\rho}(f) = \int \rho_A(f(s))\rho_G(s)ds.$$
Setting $T = (\pi^{\theta}\rtimes \lambda^{\theta})(f)$ in equation (\ref{eq_sf}), we have
\begin{equation}\label{eq_wfv}
\int \pi^{\theta}(F(s)(f(s))) \lambda_s^{\theta} ds = W^* \left(\int \rho_A(f(s))\rho_G(s) ds\right)V.
\end{equation}
Standard arguments
show that, if $a\in A$ and $(f_U)_U$ is a Dirac family at the point $t\in G$, then
\begin{equation}\label{eq_corho}
\int \rho_A((f_U\otimes a)(s))\rho_G(s) ds \longrightarrow\mbox{}_U \ \rho_A(a) \rho_G(t)
\end{equation}
in the weak operator topology.
Taking $f = f_U\otimes a$ in (\ref{eq_wfv}) and using Lemma \ref{l_point}, we obtain
a null set $N$ such that
\begin{equation}\label{eq_atpo}
\pi^{\theta}(F(t)(a))\lambda_t^{\theta} = W^*\rho_A(a) \rho_G(t) V, \text{ for all } t\in G\setminus N.
\end{equation}
For $s$ and $t$ in $G$, let
$\alpha(s), \beta(t) \in \cl B(L^2(G,H'),K)$ be given by
$$\alpha(s) = \rho_G(s^{-1}) V \lambda_{s}^{\theta}, \ \ \beta(t) = \rho_G(t^{-1}) W \lambda_{t}^{\theta};$$
then for every $\xi\in L^2(G,H')$, the functions $s\to \alpha(s)\xi$ and $s\to\beta(s)\xi$ are weakly continuous and hence $\alpha$ and $\beta$ belong to $L^{\infty}(G,$ $\cl B(L^2(G,$ $H'),$ $K))$.
Using (\ref{eq_atpo}), for all $(s,t)\in G\times G$ with $ts^{-1}\in G\setminus N$, we obtain
\begin{eqnarray*}
\beta(t)^*\rho_A(a)\alpha(s)
& = &
\lambda_{t^{-1}}^{\theta}W^*\rho_G(t)  \rho_A(a) \rho_G(s^{-1}) V \lambda_{s}^{\theta}\\
& = &
\lambda_{t^{-1}}^{\theta}W^* \rho_A(\alpha_t(a)) \rho_G(ts^{-1}) V \lambda_{s}^{\theta}\\
& = &
\lambda_{t^{-1}}^{\theta}\pi^{\theta}(F(ts^{-1})(\alpha_{t}(a)))\lambda_{ts^{-1}}^{\theta} \lambda_{s}^{\theta}\\
& = &
\pi^{\theta}(\alpha_{t^{-1}}(F(ts^{-1})(\alpha_{t}(a))) = \pi^{\theta}(\cl N(F)(s,t)(a)).
\end{eqnarray*}
As $\{(s,t):ts^{-1}\in N\}$ is a null set for the product measure $m\times m$,
by Theorem \ref{th_chschura},
$\cl N(F)$ is a Schur $A$-multiplier.
Moreover,
$$\esssup_{s\in G}\|\alpha(s)\| = \|V\| \ \mbox{ and } \ \esssup_{t\in G}\|\beta(t)\| = \|W\|$$
and hence
\begin{equation}\label{eq_normine}
\|\cl N(F)\|_{\frak{S}} \leq \|V\|\|W\| = \|F\|_{\rm m}.
\end{equation}

(ii)$\Rightarrow$(i)
Let $\theta : A\to \cl B(H')$ be a faithful *-representation, where $H'$ is
a separable Hilbert space.
Suppose that $\nph \stackrel{def}{=} \cl N(F)_{\theta}$ is a Schur $A$-multiplier.
Fix $f\in C_c(G,A)$.
For $\xi\in L^2(G,H')$, for almost all $t\in G$ we have
\begin{eqnarray*}
(\pi^{\theta}\rtimes \lambda^{\theta})(f)\xi (t)
& = &
\int \pi^{\theta}(f(s))\lambda_s^{\theta} \xi(t) ds\\
& = & \int \theta(\alpha_{t^{-1}}(f(s)))(\lambda_s^{\theta} \xi(t)) ds\\
& = &
\int \theta(\alpha_{t^{-1}}(f(s)))(\xi(s^{-1}t)) ds\\
& = & \int \Delta(r)^{-1} \theta(\alpha_{t^{-1}}(f(tr^{-1})))(\xi(r)) dr.
\end{eqnarray*}
Fix a compact set $K\subseteq G$. Then the function
\begin{equation}\label{kernel_op}
h_K : (t,r)\to \chi_{K\times K}(t,r) \Delta(r)^{-1} \theta(\alpha_{t^{-1}}(f(tr^{-1})))
\end{equation}
belongs to $L^2(G\times G, \theta(A))$.
Note that
\begin{equation}\label{eq_compK}
T_{h_K} = (M_{\chi_K}\otimes I_{H'})(\pi^{\theta}\rtimes \lambda^{\theta})(f) (M_{\chi_K}\otimes I_{H'}).
\end{equation}
We have
\begin{eqnarray*}
\nph\cdot h_K (t,r)
& = &
\chi_{K\times K}(t,r) \Delta(r)^{-1}  \theta(\alpha_{t^{-1}}(F(tr^{-1})(\alpha_{t}(\alpha_{t^{-1}}(f(tr^{-1}))))\\
& = &
\chi_{K\times K}(t,r) \Delta(r)^{-1}  \theta(\alpha_{t^{-1}}(F(tr^{-1})(f(tr^{-1})))).
\end{eqnarray*}
Let $\xi,\eta\in L^2(G,H')$ have
compact support
and $(K_n)_{n=1}^{\infty}$ be an increasing sequence of compact sets such that
$G = \cup_{n=1}^{\infty} K_n$
(such a sequence exists since $G$ is second countable and hence $\sigma$-compact).
Then
\begin{equation}\label{eq_hkn}
\langle S_{\nph}(T_{h_{K_n}})\xi,\eta\rangle
= \int_{K_n\times K_n} \Delta(r)^{-1}  \langle \theta(\alpha_{t^{-1}}(F(tr^{-1})(f(tr^{-1}))))\xi(r),\eta(t)\rangle drdt.
\end{equation}
On the other hand,
\begin{eqnarray*}
&& |\Delta(r)^{-1}  \langle \theta(\alpha_{t^{-1}}(F(tr^{-1})(f(tr^{-1}))))\xi(r),\eta(t)\rangle|\\
& \leq & \Delta(r)^{-1} \|F\|_{\infty} \|f(tr^{-1})\|\|\xi(r)\|\|\eta(t)\|,
\end{eqnarray*}
while
$$\int_{G\times G} \Delta(r)^{-1} \|f(tr^{-1})\|\|\xi(r)\|\|\eta(t)\|drdt
= \langle f'\ast \xi',\eta'\rangle,$$
where
$f', \xi',\eta' : G\to \bb{R}$ are the functions given by
$f'(s) = \|f(s)\|$, $\xi'(s) = \|\xi(s)\|$ and $\eta'(s) = \|\eta(s)\|$, $s\in G$.
Thus, an application of the Lebesgue Dominated Convergence Theorem
shows that the right hand side of
(\ref{eq_hkn}) converges to
$$\int \Delta(r)^{-1}  \langle \theta(\alpha_{t^{-1}}(F(tr^{-1})(f(tr^{-1}))))\xi(r),\eta(t)\rangle drdt
= \langle (\pi^{\theta}\rtimes\lambda^{\theta})(F\cdot f))\xi,\eta\rangle;$$
thus,
\begin{equation}\label{star}
\lim_{n\to\infty}\langle S_\nph(T_{h_{K_n}})\xi,\eta\rangle
= \langle (\pi^\theta\rtimes\lambda^\theta)(F\cdot f)\xi,\eta\rangle.
\end{equation}
By (\ref{eq_compK}), $T_{h_{K_n}} \to_{n\to\infty} (\pi^{\theta}\rtimes \lambda^{\theta})(f)$
in the weak* topology.
It follows that
\begin{eqnarray*}
|\langle (\pi^{\theta}\rtimes\lambda^{\theta})(F\cdot f))\xi,\eta\rangle|
& \leq & \limsup_{n\in \bb{N}} |\langle S_{\nph}(T_{h_{K_n}})\xi,\eta\rangle|\\
& \leq & \|\nph\|_{\frak{S}} \|\xi\|\|\eta\| \limsup_{n\in \bb{N}}\|T_{h_{K_n}}\|\\
& \leq &\|\nph\|_{\frak{S}} \|\xi\|\|\eta\| \|(\pi^{\theta}\rtimes\lambda^{\theta})(f)\|.
\end{eqnarray*}
Thus,
$$\|(\pi^{\theta}\rtimes\lambda^{\theta})(F\cdot f))\|\leq
\|\nph\|_{\frak{S}} \|(\pi^{\theta}\rtimes\lambda^{\theta})(f)\|,$$
and hence the map $S_F^{\theta}$ is bounded. Similar arguments show that, in fact,
$S_F^{\theta}$ is completely bounded and $\|S_{F}^{\theta}\|_{\rm cb}\leq \|\nph\|_{\frak{S}}$.
By Remark \ref{herz_schur_ind} (iii), $F$ is a Herz-Schur multiplier and
$$\|F\|_{\rm m} \leq \|\cl N(F)\|_{\frak{S}}.$$
The last inequality and (\ref{eq_normine}) show that
$\|F\|_{\rm m} = \|\cl N(F)\|_{\frak{S}}$
and the proof is complete.
\end{proof}

\begin{remark}\label{r_classt}
{\rm In the case $A = \bb{C}$, Theorem \ref{th_tr} reduces to
the classical transference theorem for Herz-Schur multipliers \cite{bf}:
a continuous function $u : G\to \bb{C}$ is a Herz-Schur multiplier of $A(G)$
if and only if $N(u)$ is a Schur multiplier on $G\times G$. }
\end{remark}

\begin{corollary}\label{c_weakst}
Let $\theta$ be a faithful *-representation of $A$ on a separable Hilbert space
and $F : G\to CB(A)$ be a pointwise measurable function.
The following are equivalent:

(i) \ $F$ is a Herz-Schur multiplier such that $S_F^\theta$ can be extended to a weak* continuous  map on $A\rtimes_{\alpha,\theta}^{w^*} G$;

(ii) there exists a weak* continuous completely bounded linear map $\Phi$ on
$A\rtimes_{\alpha,\theta}^{w^*} G$ such that for almost all $t\in G$
\begin{equation}\label{eq_Nweakst}
\Phi(\pi^{\theta}(a)\lambda_t^{\theta}) = \pi^{\theta}(F(t)(a))\lambda_t^{\theta}, \ \ \ a\in A.
\end{equation}

In particular, (i) holds true if $\cl N(F)$ is a Schur $\theta$-multiplier.
\end{corollary}

\begin{proof}
(i)$\Rightarrow$(ii)
Suppose that $F : G\to CB(A)$ is
a Herz-Schur multiplier and
let $\Phi$ be the (unique) weak* continuous
extension of $S_{F}^{\theta}$ to $A\rtimes_{\alpha,\theta}^{w^*} G$.
By Lemma \ref{l_point}, there exists a null set $N\subseteq G$ such that,
if $(f_U)_U$ is a Dirac family at $t\in G\setminus N$ then, for every $a\in A$,
$$\Phi((\pi^{\theta}\rtimes\lambda^{\theta})(f_U\otimes a))\to_U \pi^{\theta}(F(t)(a))\lambda_t^{\theta},$$
while, as can be easily checked,
$$(\pi^{\theta}\rtimes\lambda^{\theta})(f_U\otimes a)\to_U \pi^{\theta}(a)\lambda_t^{\theta}$$
(both limits are in the weak operator topology).
It follows that (\ref{eq_Nweakst}) holds for all $t\in G\setminus N$.

(ii)$\Rightarrow$(i)
By Remark \ref{r_extl1},
if $f\in L^1(G,A)$, then
$$\Phi\left(\int_G\pi^{\theta}(f(t))\lambda_t^{\theta}dt\right) =
\int_G\pi^{\theta}(F(t)(f(t)))\lambda_t^{\theta}dt,$$
{\it i.e.} $\Phi$ is a weak* continuous extension of $S_F^\theta$. Since $\Phi$ is completely bounded, Remark \ref{herz_schur_ind} (iii)
implies that $F\in \frak{S}(G,A,\alpha)$.

Suppose that the map $S_{\cl N(F)_{\theta}}$
has a weak* continuous extension to a map on $\cl B(L^2(G))\bar\otimes \theta(A)''$.
By the proof of Theorem \ref{th_tr},  if $G=\cup_{n=1}^\infty K_n$, where $K_n$ is an increasing sequence of compact subsets of $G$, $f\in C_c(G,A)$  and $T_{h_K}$ is the operator with the kernel $h_K$ given by (\ref{kernel_op}), then  $T_{h_{K_n}}\to (\pi^\theta\rtimes\lambda^\theta)(f)$  in the weak* topology.
As $S_{\cl N(F)_\theta}$ has a weak* continuous extension,  we have
$$\langle S_{\cl N(F)_\theta}(T_{h_{K_n}})\xi,\eta\rangle \to
\langle S_{\cl N(F)_\theta}((\pi^\theta\rtimes\lambda^\theta)(f))\xi,\eta \rangle,\ \ \ \xi,\eta\in L^2(G,H').$$
On the other hand, by (\ref{star}), $$\langle S_{\cl N(F)_\theta}(T_{h_{K_n}})\xi,\eta\rangle\to\langle S_F^\theta(\pi^\theta\rtimes\lambda^\theta(f))\xi,\eta\rangle.$$ Thus,
$S_F^{\theta}$ is the restriction of $S_{\cl N(F)_{\theta}}$ to $A\rtimes_{\alpha,\theta} G$,
and hence $S_F^\theta$ possesses a weak* continuous extension to $A\rtimes_{\alpha,\theta}^{w^*} G$.
\end{proof}

\begin{remark}\rm \label{herz_schur_theta}
We remark that if in Corollary  \ref{c_weakst} we assume also that
$F : G\to CB(A)$ is $(\pi^\theta,\lambda^\theta)$-fiber continuous then
condition (ii) is equivalent to $F$ being a Herz-Schur $\theta$-multiplier.
Therefore, in this case,
$F$ is a Herz-Schur   $\theta$-multiplier if and only if $F$ is a Hez-Schur multiplier such that $S_F^\theta$ possesses  a weak*-continuous extension to  $A\rtimes_{\alpha,\theta}^{w^*} G$.
\end{remark}

\begin{corollary}\label{c_norms}
Let $\theta$ be a faithful *-representation of $A$ on a separable Hilbert space
and $F : G\to CB(A)$ be a
Herz-Schur $\theta$-multiplier.
Then $\sup_{t\in G} \|F(t)\|_{\cb}\leq \|F\|_{\mm}$.
\end{corollary}
\begin{proof}
Immediate from Corollary \ref{c_weakst} and the fact that
$\lambda_s^{\theta}$ and $\pi^{\theta}$ are isometries.
\end{proof}

An {\it equivariant representation} of the dynamical system $(A,G,\alpha)$ on a Hilbert $A$-module
$\cl N$ is a pair $(\rho,\tau)$, where $\rho: A\to \cl L(\cl N)$ is a *-representation of $A$ on $\cl N$ and $\tau$ is a homomorphism from $G$ into the group $\cl I(\cl N)$ of all $\mathbb C$-linear, invertible, bounded maps
on $\cl N$, which satisfy:
\begin{enumerate}
\item
$\rho(\alpha_s(a))=\tau(s)\rho(a)\tau(s)^{-1}$ for all $s\in G$ and $a\in A$;
\item
$\alpha_s(\langle \xi,\eta\rangle_A)=\langle\tau(s)\xi,\tau(s)\eta\rangle_A$,
for all $s\in G$ and $\xi$, $\eta\in \cl N$;
\item
$\tau(s)(\xi\cdot a)=(\tau(s)\xi)\cdot a$, for all $s\in G$, $\xi\in\cl N$ and $a\in A$;
\item the map $s\mapsto\tau(s)\xi$ is continuous for every $\xi\in\cl N$.
\end{enumerate}
This definition was given in \cite{bc} for discrete twisted dynamical systems.
An example of such an equivariant representation can be obtained as follows. Let
 $$H_A^G=\{\xi:G\to A: \xi(\cdot)^*\xi(\cdot)\in L^1(G,A)\}.$$
 We have that $H_A^G$ is a Hilbert $A$-module with respect to the right action
 given by $(\xi\cdot a)(s):=\xi(s)a$ and the  inner product given by
 $\langle\xi,\eta\rangle_A=\int_G\xi(s)^*\eta(s)ds$.
 The regular equivariant representation
 of $(A,G,\alpha)$ on $H_A^G$ is the pair  $(\rho,\tau)$ given by
 $$\rho(a)\xi(h)=a\xi(h),\quad (\tau(t)\xi)(s)=\alpha_t(\xi(t^{-1}s)).$$
It is easy to check that $(\rho, \tau)$ satisfies the conditions (1)-(4).

The following corollary was proved in \cite[Theorem 4.8]{bc} for discrete dynamical systems using
different arguments.

\begin{corollary}
Let $(\rho,\tau)$ be an equivariant representation of $(A,G,\alpha)$ on a countably generated
Hilbert $A$-module $\cl N$, and let $\xi$, $\eta\in\cl N$. Define
$$F(t)(a)=\langle\xi,\rho(a)\tau(t)\eta\rangle_A, \ \ \ t\in G, a\in A.$$
Then $F$ is a Herz-Schur $(A,G,\alpha)$-multiplier.
\end{corollary}
\begin{proof}
By Theorem \ref{th_tr}, it suffices to show that $\cl N(F)$ is a Schur $A$-multiplier.
We have
\begin{eqnarray*}
\cl N(F)(s,t)(a)&=&\alpha_{t^{-1}}(F(ts^{-1})(\alpha_t(a)))=\alpha_{t^{-1}}(\langle\xi,\rho(\alpha_t(a))\tau(ts^{-1})\eta\rangle_A)\\
&=&\langle\tau(t^{-1})\xi,\tau(t)^{-1}\rho(\alpha_t(a))\tau(t)\tau(s^{-1})\eta\rangle_A\\
&=&
\langle\tau(t^{-1})\xi,\rho(a)\tau(s^{-1})\eta\rangle_A
\end{eqnarray*}
As
$$\|\tau(t)\xi\|^2=\|\langle\tau(t)\xi,\tau(t)\xi\rangle_A\|=\|\alpha_t(\langle\xi,\xi\rangle_A)\|=\|\langle\xi,\xi\rangle_A\|=\|\xi\|^2$$
for all $t\in G$, the statement follows from Theorem \ref{c_hilmodc}.
\end{proof}

We next identify the Schur multipliers of the form $\cl N(F)$ as the \lq\lq invariant''
part of $\frak{S}_0(G,G;A)$.
Let $\rho$ be the right regular representation of $G$ on $L^2(G)$,
{\it i.e.} $(\rho_r\xi)(s)=\Delta(r)^{1/2}\xi(sr)$, $\xi\in L^2(G)$, $s,r\in G$.
Let $\tilde{\alpha}_r = ({\rm Ad}\rho_r)\otimes\alpha_r$,
where, as usual, ${\rm Ad} U$ is the map given by ${\rm Ad} U(T) = UTU^*$;
we have that $\tilde{\alpha}_r$ is a *-automorphism of $\cl K(L^2(G)) \otimes A$
\cite[Theorem 11.2.9]{kr2}.

\begin{definition}\label{d_inhza}
A Schur $A$-multiplier $\nph : G\times G\to CB(A)$ will be called \emph{invariant} if
$S_{\nph}$ commutes with $\tilde{\alpha}_r$ for every $r\in G$.
\end{definition}

We denote by $\frak{S}_{\rm inv}(G,G;A)$ the set of all invariant Schur $A$-multipliers.
If $w$ is a (possibly vector-valued) function defined on $G\times G$,
for $r\in G$, we let $w_r$ be the function given by $w_r(s,t) = w(sr,tr)$.

\begin{lemma}\label{l_alpahr}
If $k\in L^2(G\times G,A)$ then
$\tilde{\alpha}_r(T_k) = T_{\tilde{k}}$,
where $\tilde{k}\in L^2(G\times G,A)$ is given by
$\tilde k(t,s)=\Delta(r)\alpha_r(k_r(t,s))$, $s,t\in G$.
\end{lemma}
\begin{proof}
First note that
if $k\in L^2(G\times G,A)$ then $\tilde{k}\in L^2(G\times G,A)$ and
\begin{equation}\label{eq_kktilde}
\|\tilde{k}\|_2 = \|k\|_2;
\end{equation}
indeed,
$$\|\tilde k\|_2 = \int\Delta^2(r)\|\alpha_r(k(tr,sr))\|^2dtds=
\int\|k(t',s')\|^2dt'ds'.$$
Let
$$\Theta, \Theta' : L^2(G\times G,A) \to \cl B(L^2(G,H))$$
be the maps defined by
$$\Theta(k) = T_{\tilde{k}}, \ \ \ \  \Theta'(k) = \tilde{\alpha}_r(T_k).$$
By Lemma \ref{l_Tk} and (\ref{eq_kktilde}), $\Theta$ and $\Theta'$ are continuous.

Suppose that $k \in L^2(G\times G,A)$ is given by $k = h\otimes a$, where
$h\in L^2(G\times G)$ and $a\in A$.
As $\rho_r^*\xi(s)=\Delta(r)^{-1/2}\xi(sr^{-1})$, $\xi\in L^2(G)$, $s\in G$,
we have
\begin{eqnarray*}
(\rho_rT_h\rho_r^*\xi)(s)
& = & \Delta(r)^{1/2}(T_h\rho_r^*)\xi(sr)=\Delta(r)^{1/2}\int h(sr,x)(\rho_r^*\xi)(x)dx\\
& = &\int h(sr,x)\xi(xr^{-1})dx=\int h(sr,yr)\xi(y)\Delta(r)dy,
\end{eqnarray*}
that is, $\rho_r T_h \rho_r^* = T_{\tilde{h}_r}$,
where $\tilde{h}_r(t,s) = \Delta(r)h(tr,sr)$, $s,t\in G$.
Thus,
$$\tilde{\alpha}_r(T_k) = \tilde{\alpha}_r(T_h\otimes a) = T_{\tilde{h}_r}\otimes \alpha_r(a),$$
and so $\Theta'(k) = \Theta(k)$.
As $h$ and $a$ vary, the functions $k$ span a dense subspace of $L^2(G\times G,A)$,
and hence $\Theta = \Theta'$ by continuity.
\end{proof}

In the proof of Theorem \ref{th_invpart} below, we will need the following improvement
of \cite[Lemma 3.9]{tt_p}.

\begin{lemma}\label{l_mt}
Let $\cl X$ be a separable Banach space and
$w : G\times G\to \cl B(\cl X)$ be a bounded function, such that, for every
$a\in \cl X$, the function $(s,t)\to w(s,t)(a)$ is weakly measurable, and
$w_r = w$ almost everywhere, for every $r\in G$.
Then there exists a bounded function $u : G\to \cl B(\cl  X)$ such that,
for every $a\in \cl X$, the function $s\to u(s)(a)$ is measurable and,
up to a null set, $w = N(u)$.
\end{lemma}
\begin{proof}
The map $\phi : G\times G\to G\times G$, given by $\phi(y,x) = (y,xy)$,
is continuous (and hence measurable) and bijective,
and Fubini's Theorem shows that it preserves null sets in both directions.
By assumption, for all $r\in G$, we have that
$w_r(s,x) = w(s,x)$ for almost all $(s,x)\in G\times G$. Thus,
$w_r(\phi(s,x)) = w(\phi(s,x))$ for almost all $(s,x)\in G\times G$, that is,
$w(sr,xsr) = w(s,xs)$ for almost all $(x,s)\in G\times G$.
We claim that $w(sr,xsr) = w(s,xs)$ for almost all $(x,s,r)\in G\times G\times G$.
In fact, let $\cl S\subseteq \cl X$ be a countable dense subset. For every
$a\in \cl S$, we have
\begin{eqnarray*}
& &\int_{G\times G\times G} \|w(sr,xsr)(a) - w(s,xs)(a)\|dxdsdr\\
& = &\int_G\left(\int_{G\times G} \|w(sr,xsr)(a) - w(s,xs)(a)\|dxds\right)dr = 0.
\end{eqnarray*}
Thus, there exists a null set $N_a\subseteq G\times G \times G$ such that
$w(sr,xsr)(a) = w(s,xs)(a)$ for all $(x,s,r)\not\in N_a$.
Let $N = \cup_{a\in \cl S} N_a$. Then
$w(sr,xsr)(a) = w(s,xs)(a)$ for all $(x,s,r)\not\in N$ and all $a\in \cl S$.
Since $w(sr,xsr)$ and $w(s,xs)$ are bounded operators on $\cl X$,
we conclude that $w(sr,xsr) = w(s,xs)$ for all $(x,s,r)\not\in N$.
Thus, there exists $s_0\in G$ such that
\begin{equation}\label{eq_s0}
w(s_0r, xs_0r) = w(s_0,xs_0), \mbox{ for almost all } (x,r) \in G\times G.
\end{equation}
For each $x\in G$, let $u(x) = w(s_0,xs_0)$.
Clearly, $u : G\to \cl B(\cl X)$ is a bounded function such that,
for every $a\in \cl X$, the function $x\to u(x)(a)$ is weakly measurable.
Now (\ref{eq_s0}) implies that
$w(y,xy) = u(x)$ for almost all $(x,y)\in G\times G$.
Letting $\tilde{u} : G\times G\to \cl X$ be the map given by $\tilde{u}(s,t) = u(t)$,
we thus have that $w(y,xy) = \tilde{u}(y,x)$ for almost all $(x,y)\in G\times G$.
It follows that $w(\phi^{-1}(y,xy)) = \tilde{u}(\phi^{-1}(y,x))$ for almost all $(x,y)\in G\times G$,
that is, $w(y,x) = u(xy^{-1})$ for almost all $(x,y)\in G\times G$.
The proof is complete.
\end{proof}

\begin{lemma}\label{l_clT}
Let $\nph\in \frak{S}_0(G,G;A)$. The following are equivalent:

(i) \ $\nph$ is an invariant Schur $A$-multiplier;

(ii) $\cl T(\nph)_r = \cl T(\nph)$ almost everywhere, for every $r\in G$.
\end{lemma}
\begin{proof}
Assume that $A$ is faithfully represented on a separable Hilbert space.

(i)$\Rightarrow$(ii)
Let $r\in G$, $a\in A$ and $k\in L^2(G\times G)$.
By Lemma \ref{l_alpahr},
$(S_{\nph}\circ \tilde{\alpha}_r)(T_k\otimes a) = T_{k_1}$,
where $k_1 : G\times G \to A$ is the function given by
$$k_1(t,s) =  \Delta(r)k(tr,sr)\nph(s,t)(\alpha_r(a)), \ \ \ s,t\in G,$$
and
$(\tilde{\alpha}_r\circ S_{\nph})(T_k\otimes a) = T_{k_2},$
where $k_2 : G\times G \to A$ is the function given by
$$k_2(t,s) =  \Delta(r)k(tr,sr)\alpha_r(\nph(sr,tr)(a)), \ \ \ s,t\in G.$$
By Lemma \ref{l_Tk}, $k_1 = k_2$ almost everywhere, and hence
$$\nph(sr,tr)(a) = \alpha_{r^{-1}}(\nph(s,t)(\alpha_r(a))),$$
for almost all $(s,t)\in G\times G$.
Thus, for every $a\in A$,
\begin{eqnarray*}
\cl T(\nph)(sr,tr)(a) & = & \alpha_{tr}(\nph(sr,tr)(\alpha_{r^{-1}t^{-1}}(a)))\\
& = & \alpha_{tr}(\alpha_{r^{-1}}(\nph(s,t)(\alpha_r(\alpha_{r^{-1}t^{-1}}(a))))\\
& = & \alpha_{t}(\nph(s,t)(\alpha_{t^{-1}}(a))) = \cl T(\nph)(s,t)(a)
\end{eqnarray*}
for almost all $(s,t)\in G\times G$. Since $A$ is separable, we conclude that
$\cl T(\nph)(sr,tr) = \cl T(\nph)(s,t)$ for almost all $(s,t)\in G\times G$.

(ii)$\Rightarrow$(i) follows by reversing the steps in the previous paragraph
and using the density in $\cl K(L^2(G))\otimes A$ of the linear span of the operators of the form
$T_k\otimes a$, with $k\in L^2(G\times G)$ and $a\in A$.
\end{proof}

\begin{theorem}\label{th_invpart}
The map $\cl N$ is a linear isometry from $\frak{S}(A,G,\alpha)$ onto
$\frak{S}_{\rm inv}(G,G;A)$.
\end{theorem}
\begin{proof}
By Theorem \ref{th_tr}, the map
$\cl N$ is a linear isometry from $\frak{S}(A,G,\alpha)$ into $\frak{S}_0(G,G;A)$.
By the definition of $\cl N$, we have that
$\cl T(\cl N(F))_r =\cl T(\cl N(F))$ almost everywhere for every $r\in G$ and every $F\in \frak{S}(A,G,\alpha)$.
By Lemma \ref{l_clT}, the image of $\cl N$ is in $\frak{S}_{\rm inv}(G,G;A)$.

It remains to show that $\cl N$ is surjective. To this end, let
$\theta$ be a faithful *-representation of $A$ on a separable Hilbert space $K$;
we identify $A$ with its image $\theta(A)$ under $\theta$
and let
$\nph\in \frak{S}_{\rm inv}(G,G;A)$.
By Lemmas \ref{l_clT} and \ref{l_mt},
there exists a bounded function $F : G\to \cl B(A)$ such that
$N(F) = \cl T(\nph)$ almost everywhere and such that, for every
$a\in A$, the function $s\to F(s)(a)$, is weakly measurable.
It follows that $\cl N(F) = \nph$ almost everywhere.
Since $\nph(x,y)$ is completely bounded for all $(x,y)$, we
have that $F(s)\in CB(A)$ for all $s\in G$. As $\nph$ is a Schur $A$-multiplier,
it follows from the proof of Theorem \ref{th_tr} that the map
$(\pi^\theta\rtimes\lambda^\theta)(f) \to (\pi^\theta\rtimes\lambda^\theta)(F\cdot f)$
on $(\pi^\theta\rtimes\lambda^\theta)(L^1(G,A))$
is completely bounded.
By Remark \ref{herz_schur_ind} (iii), $F\in\frak{S}(A,G,\alpha)$.
\end{proof}

We now consider bounded, as opposed to completely bounded, multipliers,
and use them to characterise Herz-Schur $\theta$-multipliers in the spirit of
\cite{ch} in Proposition \ref{p_bcb} below.
Let $\Gamma$ be a a locally compact group.
For a function $F : G\rightarrow CB(A)$, let
$F^{\Gamma} : \Gamma\times G\to CB(A)$ be given by $F^\Gamma(x,s) = F(s)$, $x\in \Gamma$, $s\in G$.
If $\alpha$ is an action of $G$ on $A$, we let $\alpha^\Gamma :  \Gamma\times G\to {\rm Aut}(A)$
be given by $\alpha^\Gamma_{(x,s)}(a) = \alpha_s(a)$, $a\in A$, $x\in \Gamma$, $s\in G$.
We have the following characterisation of Herz-Schur $\theta$-multipliers, similar to
\cite[Theorem 1.6]{ch}.

\begin{proposition}\label{p_bcb}
Let $F : G\rightarrow CB(A)$
and $\theta : A\to \cl B(K)$ be a faithful *-representation.
The following are equivalent:

(i) \ \ $F$ is a Herz-Schur $\theta$-multiplier;

(ii) \ for each locally compact group $\Gamma$, the function $F^\Gamma$
is a $\theta$-multiplier;

(iii) for $\Gamma=SU(2)$, the function $F^{\Gamma}$ is a $\theta$-multiplier.
\end{proposition}
\begin{proof}
Let $\pi^{\Gamma,\theta} : A\to \cl B(L^2(\Gamma\times G,K))$ be the *-representation
given by
$\pi^{\Gamma,\theta}(a)\xi(x,t) = \alpha^{\Gamma}_{(x^{-1},t^{-1})}(a)(\xi(x,t))$.
Identifying the Hilbert space $L^2(\Gamma\times G,K)$ with $L^2(\Gamma)\otimes L^2(G,K)$
in the natural way, we see that
$$\pi^{\Gamma,\theta}(a) = I_{L^2(\Gamma)}\otimes \pi^{\theta}(a), \ \ \ \ a\in A.$$
On the other hand,
$$\lambda_{(x,t)}^\theta = \lambda_x^{\Gamma}\otimes \lambda_t^{\theta}, \ \ \ \ x\in \Gamma, t\in G.$$

Suppose that $f\in L^1(\Gamma\times G,A)$ has the form
$f(x,s) = g(x) h(s)$, where $f \in L^1(\Gamma)$ and $h\in L^1(G,A)$.
We have
\begin{eqnarray}\label{eq_tenso}
\int_{\Gamma\times G} \pi^{\Gamma,\theta}(f(x,s))\lambda_{(x,s)}^{\theta} dx ds
& = &
\int_{\Gamma\times G} (g(x)I_{L^2(\Gamma)}\otimes \pi^\theta(f(s)))(\lambda_x^{\Gamma}\otimes \lambda_{s}^{\theta}) dx ds \nonumber\\
& = &
\left(\int_{\Gamma} g(x)\lambda_x^{\Gamma} dx\right) \otimes
\left(\int_G \pi^\theta(f(s)) \lambda_{s}^{\theta}ds\right).
\end{eqnarray}
Thus,
\begin{equation}\label{eq_gammar}
A\rtimes_{\alpha^{\Gamma},\theta}^{w^*} (\Gamma\times G)
= \vn(\Gamma)\bar{\otimes} (A\rtimes_{\alpha,\theta}^{w^*} G).
\end{equation}

(i)$\Rightarrow$(ii) Suppose that $F : G\to CB(A)$ is a Herz-Schur $\theta$-multiplier.
Since
the map $\Phi_F^{\theta}$ on $A\rtimes_{\alpha,\theta}^{w^*} G$,
given by
$\Phi_F^{\theta}(\pi^\theta(a)\lambda_t^\theta) = \pi^\theta(F(t)(a))\lambda_t^\theta$,
is completely bounded and weak* continuous,
$\id\otimes \Phi_F^{\theta}$ is a (completely) bounded map on
${\rm VN}(\Gamma)\bar{\otimes} (A\rtimes_{\alpha,\theta}^{w^*} G)$
(see {\it e.g.} \cite[Lemma 1.5]{ch}). Moreover,
\begin{equation}\label{eq_fgamma}
\pi^{\theta,\Gamma}(F^\Gamma(a))\lambda_{(x,s)}^{\theta} =(\id\otimes\Phi_F^\theta)(\pi^{\theta,\Gamma}(a)\lambda_{(x,s)}^{\theta}).
\end{equation}
It follows that the map $\Phi_{F^{\Gamma}}^{\theta}: \pi^{\theta,\Gamma}(a)\lambda_{(x,s)}^{\theta}\to \pi^{\theta,\Gamma}(F^\Gamma(a))\lambda_{(x,s)}^{\theta}$ extends to a bounded weak* continuous map
on $A\rtimes_{\alpha^\Gamma,\theta}^{w^*}(\Gamma\times G)$; in other words,
$F^{\Gamma}$ is a $\theta$-multiplier.

(ii)$\Rightarrow$(iii) is trivial.

(iii)$\Rightarrow$(i) We have that $\vn(SU(2)) \equiv \oplus_{n\in\bb{N}} M_n$, where $M_n$
is the $n$ by $n$ matrix algebra.
Hence
$$\vn(SU(2))\bar\otimes  (A\rtimes_{\alpha,\theta}^{w^*} G)
\equiv \oplus_{n=1}^\infty \left(M_n\otimes (A\rtimes_{\alpha,\theta}^{w^*}G)\right).$$
As $\Phi_{F^\Gamma}^{\theta}$ is a bounded weak* continuous map on
$\vn(SU(2))\bar\otimes  (A\rtimes_{\alpha,\theta}^{w^*} G)$,
equations (\ref{eq_gammar}) and (\ref{eq_fgamma}) now imply that
$\|\id_{M_n}\otimes \Phi_F^{\theta}\|\leq\|\Phi_{F^\Gamma}^{\theta}\|$ for all $n$ and hence
$\Phi_F^{\theta}$ is completely bounded.
\end{proof}

\section{Multipliers of the weak* crossed product}\label{s_mwscp}

In this section, we consider the weak* extendable Herz-Schur multipliers introduced in
Definition \ref{d_wsthe},
and characterise them
as the commutator of the \lq\lq scalar valued'' multipliers described in Proposition \ref{p_cbcl} below.
We fix, throughout the section, a $C^*$-dynamical system $(A,G,\alpha)$.
As before, $A$ is assumed to be separable, while $G$ is assumed to be second countable.
If $\theta : A\to \cl B(K)$ is a faithful *-representation, where $K$ is a separable Hilbert space,
we let $\alpha^{\theta} : G\to {\rm Aut}(\theta(A))$ be given by
$\alpha^{\theta}_t(\theta(a)) = \theta(\alpha_t(a))$, $t\in G$, $a\in A$.
We call $\alpha$ a \emph{$\theta$-action}, if $\alpha^{\theta}_t$ can be extended to
a weak* continuous automorphism (which we will denote in the same fashion)
of $\theta(A)''$, such that the map $s\to\alpha_s^\theta(x)$ from $G$ into $\theta(A)''$
is weak*-continuous for each $x\in \theta(A)''$.

\begin{proposition}\label{p_cbcl}
Let $u : G\to \bb{C}$ be a bounded continuous function, and let $F_u : G\to CB(A)$ be given by
$F_u(t)(a) = u(t)a$, $a\in A, t\in G$.
The following are equivalent:

(i) \ $F_u$ is a Herz-Schur $(A,G,\alpha)$-multiplier;

(ii) $u\in M^{\cb}A(G)$.

Moreover, if (i) holds then $F_u$ is a Herz-Schur $\theta$-multiplier
for every faithful representation $\theta$ of $A$ on a separable
Hilbert space.
\end{proposition}
\begin{proof}
Set $F = F_u$.  We have
\begin{equation}\label{eq_Nfst}
\cl N(F)(s,t)(a) = u(ts^{-1})a, \ \ a\in A, \ s,t\in G.
\end{equation}
We assume, without loss of generality, that $A$ is a non-degenerate
C*-subalgebra of $\cl B(H)$, for a separable Hilbert space $H$.

(i)$\Rightarrow$(ii)   By Theorems \ref{th_chschura} and \ref{th_tr}, there exist a separable Hilbert space $K$,
a non-degenerate *-representation $\rho : A\to \cl B(K)$ and elements
$V,W$ of $L^{\infty}(G,\cl B(H,K))$ such that
$$u(ts^{-1})a = W(t)^*\rho(a)V(s), \ \ \mbox{for almost all } s,t\in G
 \mbox{ and all } a\in A.$$
Let $(a_i)_{i=1}^{\infty}$ be a bounded approximate identity for $A$.  Then,
for a unit vector $\xi\in H$ and every $i\in \bb{N}$, we have
$$\langle u(ts^{-1}) a_i\xi,\xi\rangle = \langle \rho(a_i)V(s)\xi, W(t)\xi\rangle, \ \
\mbox{ for almost all } s,t\in G.$$
Since $A\subseteq \cl B(H)$ is non-degenerate and $\rho$ is a non-degenerate representation,
passing to a limit along $i$, we obtain
$$u(ts^{-1}) = \langle V(s)\xi, W(t)\xi\rangle,\ \ \mbox{ for almost all } s,t\in G.$$
By \cite{bf}, $u\in M^{\cb}A(G)$.

(ii)$\Rightarrow$(i) As $G$ is second countable, by \cite{bf}, there exist weakly measurable functions
$\xi, \eta : G\to \ell^2$ such that
$$u(ts^{-1}) = \langle \xi(s),\eta(t)\rangle, \ \ \mbox{ for almost all }s,t\in G.$$
Let $\rho : A\to \cl B(H^{\infty})$ be the countable
ampliation of the identity
representation of $A$. Write $\xi(s) = (\xi_i(s))_{i\in \bb{N}}$ and
$\eta(t) = (\eta_i(t))_{i\in \bb{N}}$, $s,t\in G$.
Let $V(s) : H\to H^{\infty}$ (resp. $W(t) : H\to H^{\infty}$)
be given by $V(s) = (\xi_i(s) I_H)_{i\in \bb{N}}$ (resp. $W(t) = (\eta_i(t) I_H)_{i\in \bb{N}}$).
Then $V, W \in L^{\infty}(G,\cl B(H,H^{\infty}))$ and
$$W(t)^*\rho(a)V(s) = \sum_{i=1}^{\infty}\xi_i(s)\overline{\eta_i(t)} a = u(ts^{-1}) a, $$
for almost all $s,t\in G$ and all  $a\in A$.
It follows by (\ref{eq_Nfst}) and
Theorems \ref{th_chschura} and \ref{th_tr} that $F_u$ is a Herz-Schur $(A,G,\alpha)$-multiplier.

Now suppose that $u\in M^{\cb}A(G)$ and
denote by $\Psi_u$ the weak* continuous completely bounded map
on $\cl B(L^2(G))$ corresponding to
the function $u$ {\it via} classical transference \cite{bf}  (see Remark \ref{r_classt}).
Let $\theta : A\to \cl B(K)$ be a faithful *-representation of $A$, for some separable Hilbert space $K$.
Note that $\cl N(F)$ is a Schur $\theta$-multiplier; indeed, we have that
$\cl N(F)_{\theta}(s,t)(\theta(a)) = u(ts^{-1})\theta(a)$, $a\in A$, and hence
$S_{\cl N(F)_{\theta}} = \Psi_u|_{\cl K(L^2(G))}\otimes \id_{\theta(A)}$.
It follows that $S_{\cl N(F)_{\theta}}$ is the restriction to
$\cl K(L^2(G))\otimes \theta(A)$ of the
weak* continuous map $\Psi_u\otimes \id_{\theta(A)''}$.
By Corollary \ref{c_weakst} and Remark \ref{herz_schur_theta},
$F_u$ is a Herz-Schur $\theta$-multiplier.
\end{proof}

In what follows we denote by $S_u^{\theta}$ the weak* continuous map on
$A\rtimes_{\alpha,\theta}^{w^*}G$ arising from the previous
proposition.

It is well-known that an essentially bounded function on $G$
that is invariant under right translations agrees almost everywhere with a constant function.
The next lemma is a dynamical system version of this fact.
For a *-representation $\theta : A\to \cl B(K)$ such that
$\alpha$ is a $\theta$-action, let $\bar\pi^{\theta}$ be the
*-representation of $\theta(A)''$ on $L^2(G,K)$ given by
$(\bar\pi^{\theta}(a)\xi)(s)=\alpha_s^{-1}(a)(\xi(s))$, $a\in \theta(A)''$,
$\xi\in L^2(G,K)$.
Recall that if $\rho:G\to\cl B(L^2(G))$ is the right regular representation of $G$ we write $\tilde\alpha_r$ for the map ${\rm Ad}\rho_r\otimes\alpha_r$ on $\cl K(L^2(G))\otimes\theta(A)$;
we have that the map $\tilde\alpha_r$ can be extended to a weak* continuous map $\cl B(L^2(G)\bar\otimes\theta(A)''$, denoted in the same fashion.

\begin{lemma}\label{l_intf}
Let $(A,G,\alpha)$ be a $C^*$-dynamical system and
$\theta : A\to \cl B(K)$ be a faithful *-representation, where $K$ is a separable Hilbert space,
such that $\alpha$ is a $\theta$-action.
Then
\begin{equation}\label{eq_ast}
(\cl D_G\bar\otimes \theta(A)'') \cap (A\rtimes_{\alpha,\theta}^{w^*} G) = \bar\pi^{\theta}(\theta(A)'').
\end{equation}
\end{lemma}
\begin{proof}
Suppose that $D$ belongs to both $L^{\infty}(G,\theta(A)'')$ and
$A\rtimes_{\alpha,\theta}^{w^*} G$.
By \cite[Theorem 1.2, Chapter II]{nakagami-takesaki},
there exists a null set $M\subseteq G$ such that
$\tilde{\alpha}_r(D) = D$ whenever $r\in M^c$.
However, $\tilde{\alpha}_r(D) = D_r$, where
$D_r \in L^{\infty}(G,\theta(A)'')$ is given by $D_r(s) = \alpha_r(D(sr))$, $s\in G$.
Let $\tilde{D}\in L^{\infty}(G,\theta(A)'')$ be defined by
$\tilde{D}(s) = \alpha_s(D(s))$, $s\in G$.
For every $r\in M^c$,
$$\tilde{D}(sr) = \alpha_{sr}(D(sr)) = \alpha_s(\alpha_r(D(sr)) = \alpha_s(D(s)) = \tilde{D}(s), \
\mbox{for almost all } s.$$
As in the proof of Lemma \ref{l_mt}, there exists $s_0\in G$ such that
$\tilde{D}(s_0r) = \tilde{D}(s_0)$ for almost all $r\in G$.
Thus, there exists $a\in \theta(A)''$ such that $\tilde{D}(t) = a$ for almost all $t\in G$,
and hence $D(t) = \alpha_{t^{-1}}(a)$ for almost all $t\in G$; in other words, $D = \bar{\pi}^{\theta}(a)$.
We thus showed  that the intersection on the left hand side of (\ref{eq_ast}) is
contained in $\bar{\pi}^{\theta}(\theta(A)'')$. The converse inclusion is trivial.
\end{proof}

Let $K$ be a Hilbert space. If $\omega\in \cl B(H)_*$, we let $L_{\omega}$
be the (unique) weak* continuous linear map from $\cl B(K\otimes H)$ into  $\cl B(K)$
such that $L_{\omega}(b\otimes a) = \omega(a)b$, $a\in \cl B(H)$, $b\in \cl B(K)$.
Recall that a weak* closed subspace $\cl U\subseteq \cl B(H)$
is said to have property $S_{\sigma}$ \cite{kraus_tams}
if
$$\cl V\bar\otimes\cl U = \{T\in \cl B(K)\bar\otimes\cl U : L_{\omega}(T)\in \cl V
\mbox{ for all } \omega\in \cl B(H)_*\},$$
for every weak* closed subspace $\cl V\subseteq \cl B(K)$.

For the proof of the next theorem, we recall that an operator $T\in \cl B(L^2(G))$
is said to be supported on a measurable subset $E\subseteq G\times G$ if
$M_{\chi_{\beta}}TM_{\chi_{\alpha}} = 0$ whenever $\alpha,\beta\subseteq G$
are measurable sets with $(\alpha\times\beta)\cap E = \emptyset$.
It is easy to see that the space of operators supported on the set
$\{(s,ts) : s\in G\}$ coincides with $\cl D_G\lambda_t^G$.

\begin{theorem}\label{th_commute}
Let $(A,G,\alpha)$ be a C*-dynamical system
and $\theta$ be a faithful *-representation of $A$ on a separable Hilbert space $K$
such that $\alpha$ is a $\theta$-action and
$\theta(A)''$ possesses property $S_{\sigma}$.
Let $\Phi$ be a completely bounded weak* continuous map on $A\rtimes_{\alpha,\theta}^{w^*} G$.
The following are equivalent:

(i) \ $\Phi S_{u}^\theta = S_{u}^\theta\Phi$ for all $u\in M^{\cb}A(G)$;

(ii) For each $t\in G$, there exists
completely bounded map
$F_t : \theta(A)''\to \theta(A)''$ such that $\Phi(\bar{\pi}^\theta(a)\lambda_t^\theta) = \bar{\pi}^\theta(F_t(a))\lambda_t^\theta$, $a\in \theta(A)''$.
\end{theorem}
\begin{proof}
(i)$\Rightarrow$(ii)
Given $u\in M^{\cb}A(G)$, let $\Psi_u$ be the weak* continuous completely bounded
map on $\cl B(L^2(G))$ corresponding to the Schur multiplier $N(u)$ (see Remark \ref{r_classt}).
We claim that
\begin{equation}\label{eq_ome}
(L_{\omega}\circ S_u^\theta)(T) = (\Psi_u\circ L_{\omega})(T), \ \ \
T\in A\rtimes_{\alpha,\theta}^{w^*} G, \ \omega\in \cl B(K)_*.
\end{equation}
For $S\in \cl B(K)$ and $T\in \cl B(L^2(G))$, we have
$$L_\omega((S\otimes T)(I_K\otimes\lambda_t^G))=\omega(S)T\lambda_t^G=L_\omega(S\otimes T)\lambda_t^G;$$
by weak* continuity, we obtain
\begin{equation}\label{eq_tG}
L_\omega(R(I_K\otimes\lambda_t^G)) = L_\omega(R)\lambda_t^G, \ \ \ R\in \cl B(L^2(G,K)).
\end{equation}
On the other hand, if $t\in G$ then
$\lambda_t^{\theta} = I_K\otimes\lambda_t^G$, and it follows that
\begin{equation}\label{eq_ttt}
L_{\omega}(S_u^\theta(\pi^\theta(a)\lambda_t^\theta)) = u(t)L_{\omega}(\pi^\theta(a)\lambda_t^\theta)
= u(t)L_{\omega}(\pi^\theta(a))\lambda_t^G, \ \ a\in A, t\in G.
\end{equation}
Since $\pi^\theta(a)\in L^{\infty}(G,\theta(A)'')$,
we have that $L_{\omega}(\pi^\theta(a))\in \cl D_G$, for every $\omega\in \cl B(K)_*$.
Since $\Psi_u$ is a $\cl D_G$-bimodule map, using equation (\ref{eq_tG}) we obtain
\begin{eqnarray}\label{eq_tttt}
\Psi_u(L_{\omega}(\pi^\theta(a)\lambda_t^\theta))
& = & \Psi_u(L_{\omega}(\pi^\theta(a))\lambda_t^G) =
L_{\omega}(\pi^\theta(a))\Psi_u(\lambda_t^G)\\
&  = &
u(t)L_{\omega}(\pi^\theta(a))\lambda_t^G.\nonumber
\end{eqnarray}
Equation (\ref{eq_ome}) follows from the weak* continuity of $L_{\omega}$ and $S_u^\theta$
(see Proposition \ref{p_cbcl}),
after comparing (\ref{eq_ttt}) and (\ref{eq_tttt}).

Let $a\in \theta(A)''$, $t\in G$ and $T = \bar{\pi}^\theta(a)\lambda_t^\theta$.
Set
$$J = \{u\in M^{\cb}A(G) : u(t) = 1\}.$$
If $u\in J$ then $S_u^{\theta}(T) = T$ and hence, by (\ref{eq_ome})
and the fact that $\Phi$ commutes with $S_u^{\theta}$, we have
$$\Psi_u(L_{\omega}(\Phi(T))) = L_{\omega}(\Phi(T)).$$
Thus, for every $u\in J$, the operator $L_{\omega}(\Phi(T))$ is $u$-harmonic in the sense of
\cite{neurun}.
It follows from \cite[Corollary 3.7]{akt} that $L_{\omega}(\Phi(T))$ is supported on the set
$\{(x,y)\in G\times G : yx^{-1}\in Z\}$, where $Z = \{s\in G : u(s) = 1, \mbox{ for all } u\in J\}$.
By the regularity of $A(G)$ and the fact that $A(G)\subseteq M^{\cb}A(G)$,
we have that $Z = \{t\}$ and hence, by (\ref{eq_tG}) and
the paragraph before the statement of the theorem,
$$L_{\omega}(\Phi(T)\lambda_{t^{-1}}^\theta) = L_{\omega}(\Phi(T))\lambda_{t^{-1}}^G\in \cl D_G.$$
Since this holds for every $\omega\in \cl B(K)_*$ and $\theta(A)''$ is assumed to possess property $S_{\sigma}$,
we conclude that $\Phi(T)\lambda_{t^{-1}}^\theta \in \cl D_G\bar\otimes \theta(A)''$.

On the other hand, $\Phi(T)\lambda_{t^{-1}}^\theta \in A\rtimes_{\alpha,\theta}^{w^*} G$.
By Lemma \ref{l_intf}, $\Phi(T)\lambda_{t^{-1}}^{\theta} = \bar{\pi}^\theta(a_t)$ for some $a_t\in \theta(A)''$,
and hence $\Phi(T) = \bar{\pi}^\theta(a_t))\lambda_t^\theta$.
Writing $F_t(a)=a_t$, we have
$\Phi(T) = \bar{\pi}^\theta(F_t(a)))\lambda_t^\theta$.
The map $F_t$ is linear and completely bounded since $\Phi$ is so.

(ii)$\Rightarrow$(i)
For $t\in G$ and $a\in \theta(A)''$, we have
$$\Phi(S_u^\theta(\bar\pi^\theta(a)\lambda_t^\theta)) = u(t) \Phi(\bar\pi^\theta(a)\lambda_t^\theta) =
u(t) \bar{\pi}^\theta(F_t(a))\lambda_t^\theta = S_u^\theta(\Phi(\bar\pi^\theta(a)\lambda_t^\theta)).$$
The commutation relations now follow by linearity and weak* continuity.
\end{proof}

\section{Two classes of multipliers}\label{s_scm}

In this section, we describe two special classes of Herz-Schur multipliers
and relate them to maps that have been studied previously.

\subsection{Multipliers from the Haagerup tensor product}
Multipliers of the type studied in this subsection have been considered in the
case of a discrete group in \cite{bc}.
Let $A$ be a separable non-degenerate C*-subalgebra of $\cl B(H)$,
where $H$ is a separable Hilbert space,
and $C_{\infty}(A)$ be the column operator space over $A$; thus, the elements of $C_{\infty}(A)$
are the sequences $(a_i)_{i\in \bb{N}}\subseteq A$ such that the series
$\sum_{i=1}^{\infty} a_i^* a_i$ converges in norm.
Recall \cite{blm} that the Haagerup tensor product $A\otimes_{\rm h} A$
consists, by definition, of all sums $u = \sum_{i=1}^{\infty} b_i\otimes a_i$,
where $(a_i)_{i\in \bb{N}}, (b_i^*)_{i\in \bb{N}} \in C_{\infty}(A)$.
Let $\beta : X\to C_{\infty}(A)$ and $\gamma : Y\to C_{\infty}(A)$ be
bounded weakly measurable functions.
Write $\beta(x) = (\beta_i(x))_{i\in \bb{N}}$, $x\in X$,
and $\gamma(y) = (\gamma_i(y))_{i\in \bb{N}}$, $y\in Y$.
It is clear that, in particular, $\beta_i\in L^\infty(X,A)$ and $\gamma_i\in L^\infty(Y,A)$
for each $i\in \bb{N}$.

Let $\nph_{\beta,\gamma} : X\times Y \to A\otimes_{\hh} A$ be given by
\begin{equation}\label{eq_gamma}
\nph_{\beta,\gamma}(x,y) = \sum_{i=1}^{\infty} \gamma_i(y)^*\otimes \beta_i(x), \ \ \ (x,y)\in X\times Y.
\end{equation}
Note that $A\otimes_{\rm h} A$ embeds canonically into $CB(A)$;
for an element $u = \sum_{i=1}^{\infty} b_i\otimes a_i$ of $A\otimes_{\rm h} A$,
the corresponding map
$\Phi_u : A\to A$ is given by $\Phi_u(a) = \sum_{i=1}^{\infty} b_i a a_i$, $a\in A$.
We thus view $\nph_{\beta,\gamma}(x,y)$ as a completely bounded map on $A$.
It is easy to see that the partial sums of (\ref{eq_gamma}) define
weakly measurable functions, and since the convergence of the series is in norm,
\cite[Lemma B.17]{dw} shows that the function
$\nph_{\beta,\gamma}$ is weakly measurable.
In particular, $\nph_{\beta,\gamma}$ is pointwise measurable.

\begin{proposition}\label{p_bc}
Let $\beta : X\to C_{\infty}(A)$ and $\gamma : Y\to C_{\infty}(A)$ be bounded
weakly measurable functions.
Then $\nph_{\beta,\gamma}$ is a Schur $\id$-multiplier.
Moreover,
\begin{equation}\label{eq_gammabeta}
S_{\nph_{\beta,\gamma}}(T) =
\sum_{i=1}^{\infty} \gamma_i^* T \beta_i, \ \ \ T\in \cl K\otimes A,
\end{equation}
where the series converges in norm.
\end{proposition}
\begin{proof}
First note that
$$\left\|\sum_{i=1}^{\infty} \beta_i^* \beta_i\right\|
= \text{essup}_{x\in X}\left\|\sum_{i=1}^{\infty} \beta_i^*(x)\beta_i(x)\right\|
= \text{essup}_{x\in X} \|\beta(x)\|^2,$$
and that a similar estimate holds for $\sum_{i=1}^{\infty} \gamma_i \gamma_i^*$.
It follows that the series on the right hand side of (\ref{eq_gammabeta}) converges in norm.
Let $k\in L^2(Y\times X, A)$, $\xi\in L^2(X,H)$ and $\eta\in L^2(Y,H)$. Then
\begin{eqnarray*}
\left\langle \sum_{i=1}^{\infty} \gamma_i^* T_k \beta_i \xi,\eta\right\rangle
& = &
\int_{X\times Y} \sum_{i=1}^{\infty} \langle k(y,x) \beta_i(x)\xi(x),\gamma_i(y)\eta(y)\rangle dxdy\\
& = &
\int_{X\times Y} \langle \nph_{\beta,\gamma}(x,y)(k(y,x))\xi(x),\eta(y)\rangle dxdy\\
& = &
\langle T_{\nph_{\beta,\gamma}\cdot k} \xi,\eta \rangle.
\end{eqnarray*}
It follows that $\nph_{\beta,\gamma}$ is a Schur $A$-multiplier.
Identity (\ref{eq_gammabeta}) now follows by boundedness.
Since the map expressed by the right hand side of (\ref{eq_gammabeta}) is
weak* extendible, we conclude that $\nph_{\beta,\gamma}$ is in fact a Schur $\id$-multiplier.
\end{proof}

Recall that a subset $E\subseteq G\times G$ is called \emph{marginally null} if
there exists a null set $M\subseteq G$ such that $E\subseteq (M\times G)\cup (G\times M)$.

\begin{proposition}\label{p_righti}
Let $\beta : X\to C_{\infty}(A)$ and $\gamma : Y\to C_{\infty}(A)$ be
bounded weakly measurable functions.
The following are equivalent:

(i) \ there exists $F\in\mathfrak S(A,G,\alpha)$ such that
$S_F^{\id}$ coincides with the restriction of $S_{\nph_{\beta,\gamma}}$
to $A\rtimes_{\alpha,\id} G$;

(ii) for every $a\in A$, the function $\nph_a : G\times G\to A$ given by
$$\nph_a(s,t) = \sum_{i=1}^{\infty} \alpha_t(\gamma_i(t))^* a\alpha_t(\beta_i(s)), \ \ s,t\in G,$$
has the property that, for every $r\in G$, $\nph_a(sr,tr) = \nph_a(s,t)$ for almost all $(s,t)$.

Moreover, if (i) holds then the map $S_F^{\id}$ has an extension to a bounded weak* continuous map on $A\rtimes_{\alpha,\id}^{w^*}G$.
\end{proposition}
\begin{proof}
(i)$\Rightarrow$(ii)
By Proposition \ref{p_bc},
the map $S_{\nph_{\beta,\gamma}}$ has a weak* continuous
extension to a completely bounded map on $\cl B(L^2(G))\bar\otimes A''$.
Since $S_F^{\id}$ is the restriction of $S_{\nph_{\beta,\gamma}}$,
it possesses a weak* continuous extension to a completely bounded map on
$A\rtimes_{\alpha,\id}^{w^*} G$.

Let $a\in \cl A$ and $s\in G$.
Note that, if $\beta_i^s\in L^{\infty}(G,A)$ is given by
$\beta_i^s(t) = \beta_i(s^{-1}t)$, then
$\lambda_s^{\id} \beta_i = \beta_i^s\lambda_s^{\id}.$
By Corollary \ref{c_weakst}, for almost all $s\in G$, we have
$$\pi^{\id}(F(s)(a)) = \left(\sum_{i=1}^\infty  \gamma_i^* \pi^{\id}(a) \lambda_s^{\id}\beta_i\right)(\lambda_s^{\id})^*
= \sum_{i=1}^\infty  \gamma_i^* \pi^{\id}(a)\beta_i^s, \  a\in A.$$
Therefore, if $\xi \in L^2(G,H)$ then, for almost all $s,t\in G$,
we have
\begin{eqnarray*}
\alpha_{t^{-1}}(F(s)(a))(\xi(t))
& = &
\left(\sum_{i=1}^\infty  \gamma_i^* \pi^{\id}(a) \beta_i^s\xi\right) (t)\\
& = &
\sum_{i=1}^\infty \gamma_i(t)^* \alpha_{t^{-1}}(a)\beta_i(s^{-1}t)(\xi(t)).
\end{eqnarray*}
A standard argument using the separability of $H$ now shows that,
for almost all $s,t\in G$, we have
$$\alpha_{t^{-1}}(F(s)(a)) = \sum_{i=1}^\infty \gamma_i(t)^* \alpha_{t^{-1}}(a)\beta_i(s^{-1}t),$$
and, since the series on the right hand side converges is norm,
\begin{equation}\label{eq_fsa}
\sum_{i=1}^\infty\alpha_t(\gamma_i^*(t))a\alpha_t(\beta_i(s^{-1}t)) = F(s)(a),
\end{equation}
{\it i.e.} $\nph_a(t,s^{-1}t)=F(s)(a)$ for almost all $s,t\in G$.
As the map  $(s,t)\mapsto (t,s^{-1}t)$ is continuous, bijective and  preserves null sets in both directions, we obtain
$\nph_a(s,t) = F(st^{-1})(a)$ for almost all $(s,t)\in G\times G$.
Hence, for each $r\in G$, $\nph_a(sr,tr)=\nph_a(s,t)$ almost everywhere on $G\times G$.

(ii)$\Rightarrow$(i)
As $\gamma(t), \beta(s) \in C_\infty(A)$ for all $(s,t)$, we have that
$$\nph_{\beta,\gamma}(s,t)(a)=\lim_{n\to\infty}\sum_{i=1}^n\gamma_i(t)^*a\beta_i(s)$$
in norm and, in particular,  $\nph_{\beta,\gamma}(s,t)(a)\in A$ for all $a\in A$. Hence
\begin{eqnarray*}
\alpha_t(\nph_{\beta,\gamma}(s,t)(\alpha_{t^{-1}}(a))
&=&\alpha_t\left(\lim_{n\to\infty}\sum_{i=1}^n\gamma_i(t)^*\alpha_{t^{-1}}(a)\beta_i(s)\right)\\
&=&\lim_{n\to\infty}
\sum_{i=1}^n\alpha_t(\gamma_i(t)^*)a\alpha_t(\beta_i(s))\\
& = &
\sum_{i=1}^\infty\alpha_t(\gamma_i(t))^*a\alpha_t(\beta_i(s))
= \nph_a(s,t).
\end{eqnarray*}
Thus $\cl T(\nph_{\beta,\gamma})(s,t)(a) = \nph_a(s,t)$ for all $s,t\in G$.
Fix $r\in G$ and let $\cl S\subseteq A$ be a countable dense subset.
Then, for every $a\in \cl S$ we have
\begin{eqnarray*}
\cl T(\nph_{\beta,\gamma})_r(s,t)(a) &=& \cl T(\nph_{\beta,\gamma})(sr,tr)(a) =
\nph_a(sr,tr) \\
&=& \nph_a(s,t) = \cl T(\nph_{\beta,\gamma})(s,t)(a),
\end{eqnarray*}
for almost all $(s,t)\in G\times G$.
It follows that there exists a set $E\subseteq G\times G$ whose complement is null,
such that $\cl T(\nph_{\beta,\gamma})_r(s,t)(a) = \cl T(\nph_{\beta,\gamma})(s,t)(a)$
for all $(s,t)$ and all $a\in \cl S$.
Fix $(s,t)\in E$. By the boundedness of the maps $\cl T(\nph_{\beta,\gamma})_r(s,t)$
and $\cl T(\nph_{\beta,\gamma})(s,t)$, we have that
$$\cl T(\nph_{\beta,\gamma})_r(s,t)(a) = \cl T(\nph_{\beta,\gamma})(s,t)(a),$$
for all $a\in A$. Thus, $\cl T(\nph_{\beta,\gamma})_r = \cl T(\nph_{\beta,\gamma})$ almost
everywhere, for all $r\in G$.
By Lemma \ref{l_clT}, Theorem \ref{th_invpart} and Proposition \ref{p_bc},
there exists $F\in \frak{S}(A,G,\alpha)$
such that $\cl N(F) = \nph_{\beta,\gamma}$ almost  everywhere.
\end{proof}

\subsection{Groupoid multipliers}\label{ss_gr}

In this subsection, we relate Herz-Schur multipliers to
the multipliers of the Fourier algebra of a groupoid. We refer the reader to
\cite{renault2} and \cite{rbook} for more details on the background, which we
now recall.

Let $G$ be a locally compact group
acting on a locally compact Hausdorff space $X$; thus,
we are given a map $X\times G\to X$, $(x,s)\to x s$,
jointly continuous and such that
$x (st) = (xs) t$ for all $x\in X$ and all $s,t\in G$.

The set $\cl G=X\times G$ is a groupoid,
where the set $\cl G^2$ of composable pairs is given by
$\cl G^2 = \{[(x_1,t_1),(x_2,t_2)]: x_2 = x_1t_1\}$,
and if $[(x_1,t_1),(x_2,t_2)] \in \cl G^2$, the product
$(x_1,t_1)\cdot (x_2,t_2)$ is defined to be $(x_1,t_1t_2)$,
while the inverse $(x,t)^{-1}$ of $(x,t)$ is defined to be $(xt, t^{-1})$.
The domain and range maps are given by
$$d((x,t)):=(x,t)^{-1}\cdot(x,t)=(xt,e), \ \ r((x,t)):=(x,t)\cdot(x,t)^{-1}=(x,e).$$
The unit space $\cl G_0$ of the groupoid,
which is by defnition equal to the common image of the maps $d$ and $r$,
can therefore be canonically identified with $X$.

Let $\lambda$ be the left Haar measure on $G$.
The groupoid $\cl G$ can be equipped with the Haar system
$\{\lambda^x:x\in X\}$, where $\lambda^x=\delta_x\times \lambda$ and
$\delta_x$ is the point mass at $x$.
The space $C_c(\cl G)$ of compactly supported continuous functions on $\cl G$
is a $*$-algebra with respect to the convolution product given by
$$(f\ast g)(x,t)=\int f(x,s)g(xs,s^{-1}t)ds,$$
and the involution given by $f^*(x,s)=\overline{f(xs,s^{-1})}$.
We equip $C_c(\cl G)$ with the norm
$$\|f\|_I=\text{max}\left\{\sup_{x\in X}\int|f|d\lambda^x, \sup_{x\in X}\int|f^*|d\lambda^x\right\}.$$
The completion of $C_c(\cl G)$ with respect to this norm is denoted by $L^I(\cl G)$,
and its enveloping $C^*$-algebra $C^*(\cl G)$ is called the \emph{groupoid $C^*$-algebra} of $\cl G$.

Let $A=C_0(X)$ and $\alpha_t(a)(x)=a(xt)$, $t\in G$, $x\in X$.
Then $\alpha:t\mapsto \alpha_t$ is a continuous homomorphism  from $G$ to ${\rm Aut}(A)$.
Identifying $C_c(X\times G)=C_c(\cl G)$ with a subspace of $C_c(G,A)$,
we see that the $*$-algebra structure on $C_c(G,A)$, associated with the action $\alpha$
(see the beginning of Section \ref{s_ggamma}),
coincides with the one on
$C_c(\cl G)$ except for the absence of the modular function in the definition of the involution.
However, the $C^*$-algebras $C^*(\cl G)$ and the
full crossed product $A\rtimes_{\alpha} G$ are isomorphic {\it via}
the map $\phi$ given by $\phi(f)(x,s)=\Delta^{-1/2}(s)f(x,s)$, $f\in C_c(X\times G)$.
In fact,
for $f,g\in C_c(X\times G)$, we have
\begin{eqnarray*}
\phi(f\ast g)(x,s)&=&\Delta^{-1/2}(s)(f\ast g)(x,s)=\Delta^{-1/2}(s)\int f(x,t)g(xt, t^{-1}s)dt,
\end{eqnarray*}
while
\begin{eqnarray*}
\phi(f)\times\phi(g)(x,s)&=&\int\phi(f)(x,t)\phi(g)(xt,t^{-1}s)dt\\&=&\int\Delta^{-1/2}(t)f(x,t)\Delta^{-1/2}(t^{-1}s)g(xt,t^{-1}s)dt\\
&=&\Delta^{-1/2}(s)\int f(x,t)g(xt,t^{-1}s)dt;
\end{eqnarray*}
hence, $\phi(f\ast g)=\phi(f)\times\phi(g)$.
In addition,
$$\phi(f^*)(x,s)=\Delta^{-1/2}(s)\overline{f(xs,s^{-1})},$$
while
\begin{eqnarray*}
\phi(f)^*(x,s)&=&\Delta^{-1}(s)\overline{\phi(f)(xs,s^{-1})}=\Delta^{-1}(s)\Delta^{-1/2}(s^{-1})\overline{f(xs,s^{-1})}\\
&=&
\Delta^{-1/2}(s)\overline{f(xs,s^{-1})},
\end{eqnarray*}
giving $\phi(f)^*=\phi(f^*)$.
By \cite[p. 9]{renault2}, the map $\phi$ extends to a
*-isomorphism from $C^*(\cl G)$ onto $A\rtimes_{\alpha} G$.

Let $\mu$ be a measure on $X$ and $\nu=\mu\times\lambda$; thus,
for a measurable subset $E$ of $X\times G$, we have
$\nu(E) = \int\lambda^x(E^x)d\mu(x)$
(for $x\in X$, we have set $E^x = E \cap (\{x\}\times G)$).
For a measurable subset $E$, set $\nu^{-1}(E)=\int\lambda^x((E^{-1})^x)d\mu(x)$.

Let $\text{Ind}(\mu)$ be the *-representation of $C_c(\cl G)$ on
$L^2(\cl G,\nu^{-1})$
given by
$$(\text{Ind}(\mu)(f)\xi)(x,t)=\int f(x,s)\xi(xs,s^{-1}t)ds, \ \ \ f\in C_c(\cl G).$$
One can check that $\|\text{Ind}(\mu)(f)\|\leq \|f\|_{I}$ \cite{renault2};
hence $\text{Ind}(\mu)$ can be extended to $C^*(\cl G)$.

If $\supp\mu=X$ then the map $f\mapsto M_f$,
where $M_f$ is the operator of  multiplication by $f$ on $L^2(X,\mu)$,
is a faithful *-representation $\theta$ of $C_0(X)$.
The corresponding regular representation $\pi^{\theta}\rtimes\lambda^\theta$
of the crossed product $A\rtimes_{\alpha} G$ on
$L^2(G,L^2(X,\mu))=L^2(X\times G,\mu\times\lambda)$ is given by
$$(\pi^{\theta}\rtimes\lambda^\theta)(f)\xi(x,t)=\int\alpha_{t^{-1}}(f(x,s))\xi(x,s^{-1}t)ds=\int f(xt^{-1},s)\xi(x,s^{-1}t)ds,$$
for $f\in C_c(X\times G)$ and $\xi\in L^2(X\times G, \mu\times\lambda)$.
Let $J\xi(x,t)=\xi(xt,t^{-1})$; then $J$ is a unitary operator from $L^2(\cl G,\nu)$ to $L^2(\cl G,\nu^{-1})$
with $J^{-1}\eta(x,t)=\eta(xt,t^{-1})$.
Let also $U\xi(x,t)=\Delta^{-1/2}(t)\xi(x,t^{-1})$; thus, $U$ is
a unitary operator on $L^2(\cl G,\nu)$. We have
$$(U^{-1}J^{-1}\text{Ind}(\mu)(f)JU\xi)(x,t)=\int f(xt^{-1},s)\xi(x,s^{-1}t)\Delta^{-1/2}(s)ds;$$
we thus see that $(\pi^{\theta}\rtimes\lambda^\theta)\circ\phi$ is unitarily equivalent to $\text{Ind}(\mu)$.

Let $I$ be the intersection of the kernels of $\text{Ind}(\mu)$ as $\mu$
varies over the measures of $X$.
The quotient $C^*(\cl G)/I$ is called the
reduced $C^*$-algebra of $\cl G$ and denoted by $C^*_{\rm red}(\cl G)$. It follows from
\cite[Proposition 2.17]{renault2} that $\text{Ind}(\mu)$ is a faithful representation of
$C^*_{\rm red}(\cl G)$ if $\supp\mu=X$. Therefore $C_{\rm red}^*(\cl G)$
is isomorphic to the reduced crossed product $C^*$-algebra of the $C^*$-dynamical system $(C_0(X), G, \alpha)$.

A measure $\mu$ on $X$ is called quasi-invariant if the measures
$\nu$ and $\nu^{-1}$ are equivalent.
It is known that $\mu$ is quasi-invariant if and only if
the measures $\mu$ and $\mu\cdot s$ are equivalent
for any $s\in G$ (here  $\mu\cdot s(E)=\mu(Es^{-1})$).
If $\delta(\cdot,s)$ is the Radon-Nikodym derivative $d(\mu\cdot s^{-1})/d\mu$
and $D$ is the Radon-Nikodym derivative $d\nu/d\nu^{-1}$ then $D(x,s)=\Delta(s)/\delta(x,s)$,
$x\in X$, $s\in G$ (see \cite[Chapter I, 3.21]{rbook}).
In what follows we will asume that $X$ possesses a quasi-invariant measure $\mu$ such that $\supp\mu=X$.

The groupoid $\cl G$ equipped with such a measure $\mu$ is called a
{\it measured groupoid} \cite{rbook}.
Next we would like to point out a connection between its multipliers, studied in \cite{renault},  and
Herz-Schur $(C_0(X),G,\alpha)$-multipliers.

The Hilbert space
$L^2(\cl G,\nu)$ carries a representation $\text{Reg}$ of $C_c(\cl G)$ defined by
$$(\text{Reg}(f)\xi)(x,s)=\int f(x,t)\xi(xt,t^{-1}s)D^{-1/2}(x,t)dt,$$
and unitarily equivalent to $\text{Ind}(\mu)$
{\it via} the unitary operator $V$ from $L^2(\cl G,\nu)$ to $L^2(\cl G,\nu^{-1})$ given by
$V\xi=D^{1/2}\xi$.
The von Neumann algebra $\text{VN}(\cl G)$ of $\cl G$
is defined to be the bicommutant $\text{Reg}(C_c(\cl G))''$ \cite[2.1]{renault}.

The Fourier algebra $A(\cl G)$ of the measured groupoid $\cl G$
was defined in \cite{renault} and is, similarly to the case where $\cl G$
is a group, a Banach algebra of complex-valued continuous functions on $\cl G$.
By \cite[Propsition 3.1]{renault}, the operator $M_\varphi$ of multiplication by
$\varphi\in L^\infty(\cl G)$ is a bounded linear map on $A(\cl G)$
if and only if the map $\text{Reg}(f)\to \text{Reg}(\varphi f)$, $f\in C_c(\cl G)$, is bounded.
The function $\varphi$ is in this case called a multiplier of $A(\cl G)$.
If the map $M_\varphi$ is moreover completely bounded then
$\varphi$ is called a completely bounded multiplier of $A(\cl G)$.
Following \cite{renault}, we denote by $MA(\cl G)$ (resp. $M_0A(\cl G)$)
the set of all multipliers (resp. completely bounded multipliers) of $A(\cl G)$.

For a bounded continuous function $\varphi : X\times G\to \bb{C}$ and $t\in G$, let
$F_\varphi(t)$ be the linear map on $C_0(X)$ given by
$F_\varphi(t)(a)(x)=\varphi(x,t)a(x)$, $a\in C_0(X)$, $x\in X$.

\begin{proposition}\label{groupoid}
Let $\varphi:X\times G \to \bb{C}$ be a bounded continuous function.
Then

(i) \ the map $F_\varphi$ is a $\theta$-multiplier if and only if $\varphi$ is a multiplier of $A(\cl G)$;

(ii) the map $F_\varphi$ is a Herz-Schur $(C_0(X),G,\alpha)$-multiplier if and only if $\varphi$ is a
completely bounded multiplier of $A(\cl G)$.
\end{proposition}
\begin{proof}
Both statements follow from the previous paragraphs,
Remark \ref{herz_schur_ind} (iii), the definition of
(Herz-Schur) $\theta$-multipliers and the fact
that $\|\text{Reg}(f)\|=\|(\pi^{\theta}\rtimes\lambda^\theta)(\phi(f))\|$, $f\in C_c(\cl G)$.
\end{proof}

The following statement gives the result of \cite[Proposition 3.8]{renault} in case $G$ is a locally compact
second countable group.

\begin{corollary}
Let $\theta:G\to \mathbb C$ be a bounded continuous function and
$\varphi : X\times G\to \bb{C}$ be the function given by
$\varphi(x,t) = \theta(t)$. Then $\varphi$ is a completely bounded multiplier of $A(\cl G)$
if and only if $\theta\in M^{\rm cb}A(G)$.
\end{corollary}
\begin{proof}
The statement follows from Proposition \ref{groupoid} and Proposition 4.1.
\end{proof}

The next corollary provides a new description of the completely bounded multipliers
of $A(\cl G)$.
We write $H = L^2(X,\mu)$.

\begin{corollary}
Let $\nph:X\times G\to\mathbb C$ be a bounded continuous function.
Assume that $\nph$ is a completely bounded multiplier of $A(\cl G)$.
Then there exist a separable Hilbert space $K$ and functions
$V,W\in L^{\infty}(G,\cl B(H,K))$ such that, for almost all $s,t\in G$, we have that
$W^*(t)V(s)\in\cl D_X$ and
$\nph(xt^{-1}, ts^{-1})=(W^*(t)V(s))(x)$, for almost all $x\in X$.
\end{corollary}
\begin{proof}
By Proposition \ref{groupoid}, $F=F_\nph$ is a Herz-Schur $(C_0(X),G,\alpha)$-multiplier.
By Theorem \ref{th_tr}, $\cl N(F)$ is a Schur  $C_0(X)$-multiplier.
We have
\begin{equation}\label{eq_clN}
\cl N(F)(s,t)(a)(x) = \varphi(xt^{-1},ts^{-1})a(x).
\end{equation}
Hence there exist a separable Hilbert space $K$, a non-degenerate $*$-represen-tation $\rho : C_0(X)\to \cl B(K)$ and
functions $V,W\in L^{\infty}(G,\cl B(H,K))$ such that
$$\varphi(xt^{-1},ts^{-1})a(x)\xi(x) = W^*(t)\rho(a)V(s)\xi(x), \ \ a\in C_0(X), \xi\in L^2(X,\mu).$$
Taking an approximate unit $(a_n)_{n\in\mathbb N}$ of $C_0(X)$ in
(\ref{eq_clN}) and letting $n\to\infty$, we obtain the statement.
\end{proof}

\section{Convolution multipliers}\label{s_cm}

{\it Throughout this section, we will assume that $G$ is an abelian locally compact group,
and will write the group operations additively.}
Let $A$ be a separable C*-algebra and $(A,G,\alpha)$ be a C*-dynamical system.
For a measure $\mu\in M(G)$,
let $\alpha_{\mu} : A\to A$ be the completely bounded map given by
$\alpha_{\mu}(a) = \int_G \alpha_r(a)d\mu(r)$ (see \cite{ss},\cite{stormer}).

\begin{definition}\label{d_cm}
A family $\Lambda = (\mu_t)_{t\in G}$, where $\mu_t\in M(G)$, $t\in G$, will be called a
\emph{convolution $(A,G,\alpha)$-multiplier} (or simply a \emph{convolution multiplier}),
if the map $F_{\Lambda} : G\to CB(A)$ given by $F_{\Lambda}(t) = \alpha_{\mu_t}$, $t\in G$,
is a Herz-Schur $(A,G,\alpha)$-multiplier.
\end{definition}

For a convolution multiplier $\Lambda = (\mu_t)_{t\in G}$, we let $\|\Lambda\|_{\mm} = \|F_{\Lambda}\|_{\mm}$.
Since $G$ is assumed to be abelian, we have that
$\alpha_\mu\circ\alpha_r=\alpha_r\circ\alpha_\mu$ for every $r\in G$
and every $\mu\in M(G)$.
It is well-known that in this case
the map $\alpha_\mu:A\to A$ lifts to a completely bounded map on th the crossed product;
the following proposition provides a concrete route to this fact.

\begin{proposition}\label{p_conmu}
Let $\mu\in M(G)$, $\mu_t = \mu$ for every $t\in G$, and $\Lambda = (\mu_t)_{t\in G}$.
Then $\Lambda$ is a convolution multiplier and $\|\Lambda\|_{\mm} \leq \|\mu\|$.
\end{proposition}
\begin{proof}
Note that $\pi(\alpha_r(a)) = \lambda_{r}\pi(a)\lambda_{r}^*$, $r\in G$.
Set $F = F_{\Lambda}$. If $a\in A$ then
$$\pi(F(s)(a)) = \pi(\alpha_{\mu}(a)) = \int_G \pi(\alpha_r(a))d\mu(r)
= \int_G  \lambda_{r}\pi(a)\lambda_{r}^* d\mu(r).$$
Hence, if $f\in L^1(G,A)$ then
\begin{eqnarray*}
S_{F}((\pi\rtimes\lambda)(f))
& = &
\int_G\pi(F(s)(f(s))\lambda_s  ds\\
& = &
\int_G\left(\int_G \lambda_r\pi(f(s))\lambda_r^* d\mu(r)\right)\lambda_sds\\
& = &
\int_G\int_G \lambda_r\pi(f(s))\lambda_s \lambda_r^* ds d\mu(r)\\
& = &
\int_G\lambda_r \left(\int_G \pi(f(s))\lambda_s ds \right) \lambda_r^* d\mu(r).
\end{eqnarray*}

The claims follow from the fact that the mapping $T\mapsto \int\lambda_rT\lambda_r^*d\mu(r)$ is a completely bounded map on $\cl B(L^2(G,H))$ with completely bounded norm dominated by $\|\mu\|$ (see \cite{ss}).
\end{proof}

In this section, we will be concerned with a special class of convolution multipliers,
which we now define.
Let $\Gamma$ be the dual group of $G$. The C*-algebra
$C^*(\Gamma)$ of $\Gamma$ is canonically *-isomorphic to its reduced C*-algebra
$C_r^*(\Gamma)$ (see {\it e.g.} \cite[Theorem 7.3.9]{ped2}).
We let $\theta : C^*(\Gamma)\to \cl B(L^2(\Gamma))$ be the associated (faithful) *-representation.
An element $s\in G$ will be viewed as a character (and, in particular, a unimodular function)
on $\Gamma$.
For $s\in G$, let
$\alpha_s : \lambda^{\Gamma}(L^1(\Gamma)) \to \lambda^{\Gamma}(L^1(\Gamma))$
be the map given by
$$\alpha_s(\lambda^{\Gamma}(f)) = \lambda^{\Gamma}(sf), \ \ f\in L^1(\Gamma).$$
Note that, if $f\in L^1(\Gamma)$, $s\in G$ and $x\in \Gamma$, then
\begin{equation}\label{eq_ano}
\alpha_s(\lambda^{\Gamma}(f)) =  M_s\lambda^{\Gamma}(f) M_{-s}.
\end{equation}
Thus, $\alpha_s$ extends canonically to an automorphism of $C_r^*(\Gamma)$.
By abuse of notation, we consider $\alpha_s$ as an automorphism of $C^*(\Gamma)$;
thus, $(C^*(\Gamma),G,\alpha)$ is a C*-dynamical system.
By (\ref{eq_ano}), $\alpha$ is a $\theta$-action.
Moreover, by \cite[Theorem 7.7.7]{ped2},
$C^*(\Gamma)\rtimes_{\alpha} G$ is *-isomorphic to the
C*-subalgebra $C^*(\Gamma)\rtimes_{\alpha,\theta} G$ of
$\cl B(L^2(G\times\Gamma))$.

Given a bounded measurable function
$\psi : G\times\Gamma\to\bb{C}$ and $t\in G$ (resp. $x\in \Gamma$),
let the function $\psi_t : \Gamma\to\bb{C}$ (resp. $\psi^x : G\to\bb{C}$)
given by $\psi_t(y) = \psi(t,y)$ (resp. $\psi^x(s) = \psi(s,x)$).
We call $\psi$ \emph{admissible} if
$\psi_t\in B(\Gamma)$ for every $t\in G$ and $\sup_t \|\psi_t\|_{B(\Gamma)} < \infty$.
Assuming that $\psi$ is addmissible, let
$F_\psi(t) : C_r^*(\Gamma)\to C_r^*(\Gamma)$ be the map given by
$$F_\psi(t)(\lambda^\Gamma(g))=\lambda^\Gamma(\psi_t g), \ \ \  g\in L^1(\Gamma).$$
By abuse of notation, we consider $F_\psi(t)$ as a map on $C^*(\Gamma)$.
Set
\begin{eqnarray*}
\frak{F}(G) = \{\psi : G\times \Gamma \to \bb{C} \ : \psi \text{ is admissible and}\\
 F_\psi  \mbox{ is a Herz-Schur $(C^*(\Gamma),G,\alpha)$-multiplier}\}
 \end{eqnarray*}
and
\begin{eqnarray*}
\frak{F}_{\theta}(G) =
\{\psi : G\times \Gamma \to \bb{C} \ : \psi \text{ is admissible and}\\F_{\psi} \mbox{ is a Herz-Schur } \theta\mbox{-multiplier}\}.
\end{eqnarray*}
Clearly, the space $\frak{F}(G)$ is an algebra with respect to the operations of pointwise addition and
multiplication, and $\frak{F}_{\theta}(G)$ is a subalgebra of $\frak{F}(G)$.
For $\psi\in \frak{F}(G)$, let $\|\psi\|_{\mm} = \|F_{\psi}\|_{\mm}$,
and use $S_{\psi}$ to denote the map $S_{F_{\psi}}$.

For $\mu\in M(G)$, set $\check{\mu}(x) = \int_G \langle x,s\rangle d\mu(s)$, $x\in \Gamma$.

\begin{proposition}\label{p_conv}
Let $\psi:G\times\Gamma\to\bb{C}$ be an addmissible function.
The following are equivalent:

(i) \ $\psi\in \frak{F}(G)$;

(ii) for each $t\in G$, there exists $\mu_t\in M(G)$ such that
$\psi(t,x)=\check{\mu}_t(x)$, $t\in G$, $x\in\Gamma$, and
the family $(\mu_t)_{t\in G}$ is a convolution $(C^*(\Gamma), G,\alpha)$-multiplier.
\end{proposition}
\begin{proof}
Note that, if $\mu\in M(G)$ and $g\in L^1(\Gamma)$ then
\begin{eqnarray}\label{eq_nst}
\alpha_{\mu}(\lambda^\Gamma(g))&=&\int\alpha_r(\lambda^\Gamma(g))d\mu(r)
=\int_G\left(\int_G\langle s, r\rangle g(s)\lambda_s^\Gamma ds\right)d\mu(r)\\
&=&
\int_G\check{\mu}(s)g(s)\lambda_s^\Gamma ds=\lambda^\Gamma(\check\mu g).\nonumber
\end{eqnarray}

(i)$\Rightarrow$(ii) If $\psi$ is admissible then,
for every $t\in G$, $\psi_t\in B(\Gamma)$ and hence, by Bochner's theorem,
there exists $\mu_t\in  M(G)$ such that $\psi_t = \check{\mu}_t$
(see {\it e.g.} \cite[Section I]{rudin}).
It follows from (\ref{eq_nst}) that the family $(\mu_t)_{t\in G}$ is a convolution multiplier.

(ii)$\Rightarrow$ (i) By (\ref{eq_nst}), $F_\psi(t)=\alpha_{\mu_t}$.
The claim now follows from the definition of a convolution multiplier.
\end{proof}

For a family $\Lambda = (\mu_t)_{t\in G}$ of measures in $M(G)$, let
$\psi_\Lambda : G\times \Gamma\to \bb{C}$ be the function given by
$\psi_\Lambda(t,x)=\check\mu_t(x)$.
Call $\Lambda$ admissible if the function $\psi_{\Lambda}$ is admissible.
By Proposition \ref{p_conv}, an admissibe family of measures
$\Lambda$ is a Herz-Schur $(C^*(\Gamma),G,\alpha)$-multiplier
if and only if $\psi_\Lambda\in \frak{F}(G)$.
By abuse of terminology, we will hence
call the elements of $\frak{F}(G)$ convolution multipliers.

\begin{corollary}\label{p_con}
Let $g\in L^{\infty}(\Gamma)$ and let $\psi : G\times \Gamma\rightarrow \bb{C}$ be given by
$\psi(s,x) = g(x)$, $s\in G$, $x\in \Gamma$. The following are equivalent:

(i) \ $\psi\in \frak{F}(G)$;

(ii) $g\in B(\Gamma)$.

\noindent Moreover, if (i) holds then $\psi\in \frak{F}_{\theta}(G)$.
\end{corollary}
\begin{proof}
The equivalence of (i) and (ii) follows from Propositions \ref{p_conmu} and \ref{p_conv}.
Suppose that $g\in B(\Gamma)$. Then the map on $C_r^*(\Gamma)$
corresponding to $g$ {\it via} classical transference
has a (completely bounded) weak* continuous extension $\Phi_g : \vn(\Gamma)\to \vn(\Gamma)$.
Thus, the restriction of the map $\Phi_g\otimes\id$ to $C^*(\Gamma)\rtimes_{\alpha,\theta}^{w^*} G$
is a weak* continuous extension of $S_{F_\psi}^\theta$.
By Remark \ref{herz_schur_theta}, $F_\psi$ is a Herz-Schur $\theta$-multiplier.
\end{proof}

It will be convenient, in the sequel, to denote by $S_g$ the map $S_{\psi}$,
where $\psi$ and $g$ are as in Corollary \ref{p_con}.

\medskip

The rest of the paper will be devoted to properties
of the spaces $\frak{F}(G)$ and $\frak{F}_{\theta}(G)$.
In the next theorem, we identify an elementary tensor $u\otimes h$, where $u\in B(G)$ and
$h\in B(\Gamma)$, with the function $(s,x)\to u(s)h(x)$, $s\in G$, $x\in \Gamma$.
Let $\frak{F}(B(G),B(\Gamma))$ be the complex vector space of all
separately continuous functions $\psi : G\times \Gamma\to \bb{C}$
such that, for every $s\in G$ (resp. $x\in \Gamma$),
the function $\psi_s : \Gamma\to\bb{C}$ (resp. $\psi^x : G\to\bb{C}$)
belongs to $B(\Gamma)$ (resp. $B(G)$).

\begin{theorem}\label{th_bgbg}
(i) \ \ The inclusions
$$B(G)\odot B(\Gamma)\subseteq \frak{F}_{\theta}(G)\subseteq \frak{F}(B(G),B(\Gamma))$$
hold.

(ii) \ Suppose that $\psi\in \frak{F}_{\theta}(G)$.
Then $\|\psi^x\|_{B(G)}\leq \|\psi\|_{\mm}$ for every $x\in \Gamma$
and $\|\psi_s\|_{B(\Gamma)}\leq \|\psi\|_{\mm}$ for every $s\in G$.

(iii) Let $\psi : G\times \Gamma\to \bb{C}$ be an admissible function, such that
the function $G\to B(\Gamma)$, sending $s$ to $\psi_s$, is continuous.
Suppose that $(\psi_k)_{k\in \bb{N}}\subseteq \frak{F}(G)$, $\sup_{k\in \bb{N}} \|\psi_k\|_{\infty} < \infty$
and $\psi_k\to \psi$ pointwise.
Then $\psi\in \frak{F}(G)$.
\end{theorem}
\begin{proof}
(i) The first inclusion follows from Proposition \ref{p_cbcl} and Corollary \ref{p_con}.

Let $\psi\in \frak{F}_{\theta}(G)$ and fix $x\in \Gamma$. The map $\Psi_{\psi_s}$ corresponding to $\psi_s$ {\it via} classical transference
satisfies the identities $\Psi_{\psi_s}(\lambda_x^{\Gamma}) = \psi_s(x)\lambda_x^{\Gamma}$, $x\in \Gamma$. Thus
\begin{eqnarray*}
\Phi_\psi^\theta(\bar{\pi}^\theta(\lambda_x^\Gamma)\lambda_s^\theta)&=&\bar{\pi}^\theta(\psi_s(x)\lambda_x^\Gamma)\lambda_s^\theta\\
&=&\psi_s(x)\bar{\pi}^\theta(\lambda_x^\Gamma)\lambda_s^\theta=
\bar{\pi}^\theta(\lambda_x^\Gamma)(\psi^x(s)\lambda_s^G\otimes I).
\end{eqnarray*}

On the other hand,
$$\bar{\pi}^\theta(\lambda_x^\Gamma)\lambda_s^\theta=\bar{\pi}^\theta(\lambda_x^\Gamma)(\lambda_s^G\otimes I).$$
It follows that the map
$$\lambda_s^G\to\psi^x(s)\lambda_s^G$$
is bounded,
and hence $\psi^x$ is a Herz-Schur multiplier giving
$\psi^x\in B(G)$.
The fact that $\psi_s\in B(\Gamma)$ for every $s\in G$ is implicit in the
definition of the space $\frak{F}(G)$.

(ii) The inequalities $\|\psi^x\|_{B(\Gamma)}\leq \|\psi\|_{\mm}$, $x\in \Gamma$, follow
from the proof of (i). The inequalities $\|\psi_s\|_{B(\Gamma)}\leq \|\psi\|_{\mm}$, $s\in G$,
follow from Corollary \ref{c_norms} and the fact that
$\|\psi_s\|_{B(\Gamma)} = \|F_{\psi}(s)\|_{\cb}$, $s\in G$.

(iii) Suppose that
$\|\psi_k\|_{\infty}\leq C$ for every $k\in \bb{N}$.
Let $\xi,\eta\in L^2(G\times\Gamma)$ and $f\in L^1(G,C_c(\Gamma))$.
A direct verification shows that
\begin{eqnarray*}
& & \langle S_{\psi_k}^{\theta}((\pi^{\theta}\rtimes\lambda^{\theta})(f))\xi,\eta\rangle
=
\langle (\pi^{\theta}\rtimes\lambda^{\theta})(\psi_k f)\xi,\eta\rangle\\
& = &
\int \overline{\langle t,x\rangle} \psi_k(t,y) f(t,y)\xi(s - t,x - y)\overline{\eta(s,x)} ds dt dx dy.
\end{eqnarray*}
On the other hand,
$$|\overline{\langle t,x\rangle} \psi_k(t,y) f(t,y)\xi(s - t,x - y)\overline{\eta(s,x)}|
\leq C |f(t,y)| |\xi(s - t,x - y)| |\eta(s,x)|,$$
and the $L^1$-norm of the latter function is equal to
$C\langle |f|\ast |\xi|,|\eta|\rangle$.
By the Lebesgue Dominated Convergence Theorem,
\begin{equation}\label{eq_tends}
\langle S_{\psi_k}^{\theta}((\pi^{\theta}\rtimes\lambda^{\theta})(f))\xi,\eta\rangle
\to \langle (\pi^{\theta}\rtimes\lambda^{\theta})(\psi f)\xi,\eta\rangle.
\end{equation}
Now let $f_{i,j}\in L^1(G,C_c(\Gamma))$, $i,j = 1,\dots,m$.
By (\ref{eq_tends}), the operator matrix
$(S_{\psi_k}^{\theta}((\pi^{\theta}\rtimes\lambda^{\theta})(f)))_{i,j}$
converges weakly to $((\pi^{\theta}\rtimes\lambda^{\theta})(\psi f))_{i,j}$.
The fact that the map $s\to \psi_s$ is continuous implies that
$\psi$ is weakly measurable.
Since $\|S_{\psi_k}^{\theta}\|_{\rm cb}\leq C$ for all $k$, we conclude that
$\psi\in \frak{F}(G)$ and that $\|\psi\|_{\mm}\leq C$.
\end{proof}

In view of Theorem \ref{th_bgbg}, it is natural to ask the following question.

\begin{question}
Can $\frak{F}_{\theta}(G)$ be characterised as a topological closure of $B(G)\odot B(\Gamma)$?
\end{question}

Let $CB_{w^*}(C^*(\Gamma)\rtimes_{\alpha,\theta} G)$ be the space of all
completely bounded maps on $C^*(\Gamma)\rtimes_{\alpha,\theta} G$
which admit a weak* continuous extension to
$C^*(\Gamma)\rtimes_{\alpha,\theta}^{w^*} G$.
Set
$$\cl S = \{S_{\psi}^{\theta} : \psi\in \frak{F}_{\theta}(G)\}.$$

As usual, if $\cl J$ is a family of linear transformations acting on a vector space,
we denote by $\cl J'$ its commutant.
We recall that, for $g\in B(\Gamma)$,
we let $S_g^{\theta}$ denote the map on $C_r^*(\Gamma)$ given by
$S_g^{\theta}(\lambda^{\Gamma}(f))=\lambda^{\Gamma}(g f)$, $f\in L^1(\Gamma)$.

\begin{theorem}\label{th_me}
We have
\begin{equation}\label{eq_eqah}
\cl  S = CB_{w^*}(C^*(\Gamma)\rtimes_{\alpha,\theta} G)\cap \{S_u^\theta, S_v^\theta  : u\in B(G), v\in B(\Gamma)\}'.
\end{equation}
In particular, $\cl S$ is a maximal abelian subalgebra of $CB_{w^*}(C^*(\Gamma)\rtimes_{\alpha,\theta} G)$.
\end{theorem}
\begin{proof}
Note that
$\cl S$ is a commutative subalgebra of $CB_{w^*}(C^*(\Gamma)\rtimes_{\alpha,\theta} G)$
and, by Theorem \ref{th_bgbg} (i),
contains the maps of the form $S_u^\theta$ and $S_v^\theta$, where $u\in B(G)$ and $v\in B(\Gamma)$.
It follows that it is contained in the commutant on the right hand side of (\ref{eq_eqah}).

Let $A = C^*(\Gamma)$.
To establish the reverse inclusion,
suppose that $\Phi \in CB_{w^*}(C^*(\Gamma)\rtimes_{\alpha,\theta} G)$
commutes with the operators of the form $S_u^\theta$ and $S_v^\theta$,
where $u\in B(G)$ and $v\in B(\Gamma)$.
Since $G$ is abelian, $\vn(G)$ possesses property $S_{\sigma}$ (see {\it e.g.} \cite[Theorem 1.9]{kraus_tams}).
By Theorem \ref{th_commute},
for each $t\in G$, there exists a
completely bounded map
$F_t : \vn(\Gamma)\to \vn(\Gamma)$ such that
$$\Phi(\bar{\pi}^{\theta}(a)\lambda_t^{\theta})
= \bar{\pi}^{\theta}(F_t(a))\lambda_t^{\theta}, \ \ \ a\in \vn(\Gamma).$$
If $v\in B(\Gamma)$, $f\in L^1(\Gamma)$ and $a = \lambda^{\Gamma}(f)$ then
$$\bar{\pi}^{\theta}(F_t(\lambda^{\Gamma}(vf)))\lambda_t^{\theta}
= \Phi(S_v^\theta(\bar{\pi}^{\theta}(a)\lambda_t^{\theta})) =
S_v^\theta(\Phi(\bar{\pi}^{\theta}(a)\lambda_t^{\theta})) =
\bar{\pi}^{\theta}(vF_t(\lambda^{\Gamma}(f))\lambda_t^{\theta}.$$
Thus, $F_t(\lambda^{\Gamma}(vf)) = vF_t(\lambda^{\Gamma}(f))$ for each $v\in B(\Gamma)$.
Let $\tilde F_t$ be the restriction of $F_t$ to $C_r^*(\Gamma)$.
Then $\tilde F_t^*$ is a map on $B(\Gamma)$. In particular,
$\tilde F_t^*(A(\Gamma))\subseteq B(\Gamma)$ and $\tilde F_t^*(vu)=v\tilde F_t^*(u)$ for any $u\in A(\Gamma)$ and $v\in B(\Gamma)$. It follows that $\tilde F_t^*(u)=\psi_tu$ for some function
$\psi_t:\Gamma\to\mathbb C$.
As $\psi_tu\in B(\Gamma)$ for all $u\in A(\Gamma)$,
by \cite[Theorem 3.8.1]{rudin}, $\psi_t\in B(\Gamma)$.
Hence, for $u\in A(\Gamma)$, we have
$$\langle F_t(\lambda^{\Gamma}(f)),u\rangle
= \langle\lambda^{\Gamma}(f), \tilde F_t^*(u)\rangle
= \langle\lambda^{\Gamma}(f),\psi_tu\rangle
= \langle \lambda^{\Gamma}(\psi_tf),u\rangle$$
and $F_t(\lambda^{\Gamma}(f))
= \lambda^{\Gamma}(\psi_tf)$.
The proof is complete.
\end{proof}

Our next aim is to identify
the joint commutant of two families of
completely bounded maps on $\cl K(L^2(G))$, in terms of multipliers of Herz-Schur type.
Recall that, for $a\in L^{\infty}(G)$, we denote by $M_a$ the operator on $L^2(G)$ given by
$M_a f = af$, $f\in L^2(G)$, and set
$$\cl C = \{M_a : a\in C_0(G)\}.$$
We let $\id$ be the identity representation of $\cl C$.
For $t\in G$, let $\beta_t : C_0(G)\to C_0(G)$ be given by $\beta_t(h)(s) = h(s - t)$, $h\in C_0(G)$.
By abuse of notation, we denote by $\beta_t$ the corresponding map on the C*-algebra $\cl C$.
Note that
\begin{equation}\label{eq_lG}
\beta_t(T) = \lambda_t^G T \lambda_{-t}^G, \ \ \ T\in \cl C,
\end{equation}
and that $(\cl C,G,\beta)$ is a C*-dynamical system.
Note also that
\begin{equation}\label{eq_muma}
\beta_{\mu}(M_a) = M_{\mu\ast a}, \ \ \ \mu\in M(G), a\in L^{\infty}(G).
\end{equation}
Note that $\cl C\rtimes_{\beta,\id} G$ is a C*-subalgebra of $\cl B(L^2(G\times G))$.

Let $\cl F : L^2(G)\to L^2(\Gamma)$ be the Fourier transform, so that
$\cl F\xi(x)=\int_G \langle x,s\rangle\xi(s)ds$, $\xi\in L^1(G)\cap L^2(G)$, $x\in \Gamma$.
Then
\begin{equation}\label{eq_tobeu}
\cl F^* M_t \cl F = \lambda_t^G \ \mbox{ and } \
\cl F^*\lambda^\Gamma(f)\cl F=M_{\hat f}, \ \ \ t\in G, f\in L^1(\Gamma),
\end{equation}
where $\hat f : G\to \bb{C}$ is the function given by
$\hat f(t)=\int_\Gamma \overline{\langle t,x\rangle}f(x)dx$.
In particular, $\cl F^* C_r^*(\Gamma)\cl F = \cl C$.
Moreover, if $f\in L^1(\Gamma)$ then
\begin{eqnarray}\label{eq_towli}
\cl F^*\alpha_t(\lambda^\Gamma(f))\cl F
= \beta_t(\cl F^*\lambda^\Gamma(f)\cl F),
\end{eqnarray}
giving
$$\cl F^*\alpha_t(T)\cl F = \beta_t(\cl F^* T\cl F), \ \ \ T\in C_r^*(\Gamma), t\in G.$$
Let $\tilde{\cl F} = I \otimes\cl F^*$; thus,
$\tilde{\cl F}$ is a unitary operator from $L^2(G \times \Gamma)$ onto $L^2(G\times G)$.
It is well-known that
$$\tilde{\cl F} (C^*(\Gamma)\rtimes_{\alpha,\theta} G) \tilde{\cl F}^* = \cl C\rtimes_{\beta,\id} G.$$

Let $\Lambda = (\mu_t)_{t\in G}$ be a family of measures in $M(G)$, such that the function
$\psi_{\Lambda}$ is in $\frak{F}(G)$.
Then $\psi_{\Lambda}$ gives rise to a Herz-Schur $(\cl C,G,\beta)$-multiplier given by
$$\pi(M_{\hat g})\lambda_t = \pi(\cl F^*\lambda^\Gamma(g)\cl F)\lambda_t
\mapsto \pi(\cl F^*F_{\psi_t}(\lambda^\Gamma(g))\cl F) \lambda_t
= \pi(M_{\mu_t\ast\hat g}) \lambda_t.$$
This observation was our motivation for the chosen terminology for
convolution multipliers.  Note that the convolution multipliers are of
different nature than Herz-Schur $(\cl C,G,\beta)$-multipliers considered in Section \ref{ss_gr}.

\medskip

The pair $(\id,\lambda^G)$ is a covariant representation of
$(\cl C,G,\beta)$ (see (\ref{eq_lG}));
in addition, $\id\rtimes\lambda^G$
is a faithful representation of $\cl C\rtimes_{\beta} G$ on $L^2(G)$ and its image
coincides with the algebra $\cl K(L^2(G))$ of all compact operators on $L^2(G)$
(see \cite{rieffel} and \cite{dw}).

For $\psi\in \frak{F}(G)$,
let
$\cl E_{\psi} : \cl K(L^2(G))\to \cl K(L^2(G))$ be the (completely bounded) map given by
$$\cl E_{\psi}((\id\rtimes\lambda^G)(\tilde{\cl F} T\tilde{\cl F}^*)) =
(\id\rtimes\lambda^G)(\tilde{\cl F} S_{\psi}^\theta(T)\tilde{\cl F}^*), \ \ \
T\in \cl C^*(\Gamma)\rtimes_{\alpha,\theta} G.$$
We extend $\cl E_{\psi}$ to a weak* continuous map on $\cl B(L^2(G))$, denoted in the same fashion.
For $r\in G$, let $\rho_r^G\in \cl B(L^2(G))$ be the corresponding
right regular unitary on $L^2(G)$, that is,
$\rho_r^G f (s) = f(s+r)$, $s,r\in G$, $f\in L^2(G)$.
For a measure $\mu\in M(G)$, consider the map
$\Theta(\mu) \in CB(\cl B(L^2(G))$, given by
\begin{equation}\label{eq_chv}
\Theta(\mu)(T) = \int_G \rho_r^G T \rho_{-r}^G d\mu(r), \ \ \ T\in \cl B(L^2(G)).
\end{equation}
It is easy to see that
\begin{equation}\label{eq_thetamu}
\Theta(\mu)(M_a) = M_{\mu\cdot a} = M_{\tilde{\mu}\ast a}, \ \ \ a\in L^{\infty}(G),
\end{equation}
where $(\mu\cdot a)(s)=\int_G a(s+r)d\mu(r)$ and $\tilde{\mu}$ is the
measure on $G$ given by $\tilde\mu(E)=\mu(-E)$.
Note that $\Theta(\mu)$ is a $\vn(G)$-bimodule map and leaves $\cl D_G$ invariant.

Recall that, for every $u\in B(G)$, the function $N(u)$ given by $N(u)(s,t) = u(t - s)$,
is a Schur multiplier \cite{bf} (see also Remrak \ref{r_classt}).
Thus, the corresponding map $\Psi_u : \cl B(L^2(G))\to \cl B(L^2(G))$ is
a completely bounded $\cl D_G$-bimodule map that leaves $\vn(G)$ invariant.

\begin{proposition}\label{p_onl2}
Suppose that $\mu\in M(G)$ and $u\in B(G)$. Let
$\psi_{\mu}$ and $\psi_u$ be the elements of $\frak{F}(G)$
given by $\psi_{\mu}(s,x) = \check{\mu}(x)$ and
$\psi_u(s,x) = u(s)$, $s\in G$, $x\in \Gamma$.
Then 

(i) \ $\cl E_{\psi_{\mu}} = \Theta(\tilde\mu)$, and

(ii) $\cl E_{\psi_u} = \Psi_u$.
\end{proposition}
\begin{proof}
Let $f\in C_c(G,C_c(\Gamma))$ be  given by
$f(s) = f_0(s) g$, for a certain $g\in C_c(\Gamma)$ and a certain $f_0\in C_c(G)$. Let $i:C_c(G,C_c(\Gamma))\to C^*(\Gamma)\rtimes_{\alpha,\theta} G $ be the embedding map given by
$$i(h)\xi(t)=\int_G\alpha_{-t}(\lambda^\Gamma(h(s)))\lambda_s^{\theta}\xi(t)ds, \ \ \xi\in L^2(G\times\Gamma),
h\in C_c(G\times \Gamma).$$
We have
$$i(f)\xi(t)=\int_G\alpha_{-t}(\lambda^\Gamma(g))f_0(s)\lambda_s^{\theta}\xi(t)ds$$
and, by (\ref{eq_tobeu}) and (\ref{eq_towli}),
\begin{eqnarray}\label{eq_iotaf}
\tilde{\cl F}i(f)\tilde{\cl F}^*\xi(t)
& = & \int_G{\cl F}^*\alpha_{-t}(\lambda^\Gamma(g)){\cl F}f_0(s)\lambda_s^{\theta}\xi(t)ds\nonumber\\
&=&\int_G\beta_{-t}({\cl F}^*(\lambda^\Gamma(g)){\cl F})f_0(s)\lambda_s^{\theta}\xi(t)ds\\
&=&\int_G\beta_{-t}(M_{\hat g})f_0(s)\lambda_s^{\theta}\xi(t)ds.\nonumber
\end{eqnarray}

Let  $T, T_{\mu} : G\to \cl C$ be the maps
given by $T(s) = M_{f_0(s)\hat g}$ and $T_\mu(s) = M_{f_0(s) \mu\ast \hat{g}}$, $s\in G$;
clearly, $T,T_{\mu}\in L^1(G,\cl C)$.
Let $j : L^1(G,\cl C)\to \cl C\rtimes_{\beta,\id} G$ be the canonical injection.
By (\ref{eq_iotaf}), $\tilde{\cl F}i(f)\tilde{\cl F}^* = j(T)$ and
$\tilde{\cl F}i(\psi_\mu f)\tilde{\cl F}^* = j(T_{\mu})$.

(i) Note that
$$S_{\psi_\mu}^{\theta}(i(f)) = i (\psi_\mu f) = i(f_0\otimes (\check{\mu}g)).$$
We have
\begin{eqnarray*}
\cl E_{\psi_{\mu}}(M_{\hat{g}}\lambda^G(f_0))
& = &
\cl E_{\psi_{\mu}}((\id\rtimes\lambda^G)(j(T)))
= \cl E_{\psi_{\mu}}((\id\rtimes\lambda^G)(\tilde{\cl F}i(f)\tilde{\cl F}^*)) \\
&=&(\id\rtimes\lambda^G)(\tilde{\cl F} S_{\psi_\mu}^\theta(i(f))\tilde{\cl F}^*)
= (\id\rtimes\lambda^G)(\tilde{\cl F} i(f_0\otimes\check\mu g)\tilde{\cl F}^*)\\
& = &
(\id\rtimes\lambda^G)(j(T_\mu)) = M_{\mu\ast\hat{g}}\lambda^G(f_0).
\end{eqnarray*}
By (\ref{eq_thetamu}) and the modularity of $\Theta(\mu)$ over $\vn(G)$ we now have
$$\cl E_{\psi_{\mu}}(M_{\hat{g}}\lambda^G(f_0))
= M_{\mu\ast\hat{g}}\lambda^G(f_0)=M_{\tilde\mu\cdot\hat{g}}\lambda^G(f_0)
= \Theta(\tilde\mu)(M_{\hat{g}}\lambda^G(f_0)).$$
Since the operators of the form $M_{\hat{g}}\lambda^G(f_0)$ span a dense
subspace of $\cl K(L^2(G))$, it follows that $\cl E_{\psi_{\mu}} = \Theta(\tilde\mu)$.

(ii)
Similarly to (i), we have
$$\cl E_{\psi_u}(M_{\hat g}\lambda^G(f_0))
= (\id\rtimes\lambda^G)(\tilde{\cl F}i(uf_0\otimes g)\tilde{\cl F}^*)
= M_{\hat g}\lambda^G(uf_0).$$
Since, by \cite{j}, $\lambda^G(uf_0) = \Psi_u(\lambda^G(f_0))$, and $\Psi_u$ is a $\cl C$-bimodule map,
we obtain that
$\cl E_{\psi_u}(M_{\hat g}\lambda^G(f_0)) = \Psi_u((M_{\hat g}\lambda^G(f_0))$.
The statement now follows by the density of the linear span of the operators
of the form $M_{\hat g}\lambda^G(f_0)$ in $\cl K(L^2(G))$.
\end{proof}

\begin{definition}
Let $(\theta,\tau)$  be a covariant representation of the dynamical system $(A,G,\alpha)$. We say that $F:G\to CB(A)$ is a Herz-Schur $(\theta,\tau)$-multiplier if the map
$$\theta(a)\tau_s\to \theta(F(s)(a)\tau_s, \ s\in G,\  a\in A$$
can be extended to a weak*-continuous completely bounded map on
the weak* closed hull of
$(\theta\rtimes\tau)(A\rtimes_{\alpha}G)$.
\end{definition}

Let $\frak{F}_{\id,\lambda^G}(G)$ be the set of admissible functions $\psi:G\times\Gamma\to\mathbb C$ such that the corresponding $F_{\psi}: G\to CB(\cl C)$ is a Herz-Schur $(\id,\lambda^G)$-multiplier for the
dynamical system $(\cl C, G,\beta)$.
\begin{theorem}\label{th_me2}
Let $ \cl E = \{\cl E_{\psi} : \psi\in \frak{F}_{\id,\lambda^G}(G)\}$. Then
$$\cl E = CB(\cl K(L^2(G)))\cap \{\Theta(\mu), \Psi_u : \mu\in M(G), u\in B(G)\}'.$$
In particular, $\cl E$ is a maximal abelian subalgebra of $CB(\cl K(L^2(G))$.
\end{theorem}
\begin{proof}
The  fact that $\cl E$ is contained in the right hand side
follows from the fact that
$\cl E$ is a commutative subalgebra of $CB(\cl K(L^2(G)))$
and, by Proposition \ref{p_onl2},
contains the maps $\Theta(\mu)$ and $\Psi_u$, where $u\in B(G)$ and $\mu \in M(G)$.

To prove the reverse inclusion, we modify the arguments
in the proof of Theorem \ref{th_commute}.
Let $\Phi\in CB(\cl K)$ commute with $\Psi_u$ and $\Theta(\mu)$,
for all $\mu\in M(G)$ and all $u\in B(G)$.
The map $\Phi$ has a unique extension   to a weak* continuous completely bounded map on
$\cl B(L^2(G))$, which  will be denoted by the same symbol.
Let $t\in G$, $a\in L^\infty(G)$ and $T=M_a\lambda_t^G$. Set
$$J=\{u\in B(G): u(t)=1\}.$$
Then for $u\in J$ we have
$$\Psi_u(M_a\lambda_t^G)=u(t)M_a\lambda_t^G=M_a\lambda_t^G.$$
As $\Phi \Psi_u = \Psi_u\Phi$, we obtain
$$\Psi_u (\Phi(T))=\Phi(T).$$
By \cite{akt} and the remark before
Theorem \ref{th_commute}, we conclude that $\Phi(T)\lambda_{-t}^G\in \cl D_G$.
Therefore, there exists $a_t\in L^\infty(G)$ such that
$\Phi(M_a\lambda_t^G)=M_{a_t}\lambda_t^G$.  Let $F_t(a)=a_t$, $t\in G$.
Then $F_t$ is a linear map on $\cl D_G$.
Since $\Phi$ is completely bounded and weak* continuous,
$F_t$ is so, too.

Since $\Phi$ commutes with $\Theta(\mu)$, $\mu\in M(G)$, we have
$$\Phi(M_{\mu\cdot a}\lambda_t^G)=M_{\mu\cdot F_t(a)}\lambda_t^G,$$
giving $F_t(\mu\cdot a)=\mu\cdot F_t(a)$.
The map $F_t$ is the adjoint of a bounded linear map
$\Psi_t : L^1(G)\to L^1(G)$ such that
$\Psi_t(\mu\ast f) = \mu\ast\Psi_t(f)$, $\mu\in M(G)$, $f\in L^1(G)$.
By \cite[Theorem 3.8]{rudin},
there exists $\nu_t\in M(G)$ such that $F_t(a)=\nu_t\ast a$, $a\in L^\infty(G)$.
Let $\psi(t,x) = \check\nu_t(x)$, $t\in G$, $x\in \Gamma$.
Then $\psi \in \frak{F}_{\id,\lambda^G}(G)$ and $\Phi$ is the weak* extension of  $\cl E_{\psi}$.
\end{proof}

\begin{remark}\rm
Assume $G$ is arbitrary and let $\beta_t:\cl C\to \cl C$ be given by $\beta_t(M_a)=\lambda_t^GM_a\lambda_{t^{-1}}^G$. Then $(\cl C,G,\beta)$ is a $C^*$-dynamical system. For a measure $\mu$ we consider the map $\Theta(\mu)\in CB(L^2(G))$ given by (\ref{eq_chv}).
We have $\Theta(\mu)(\cl C)\subseteq\cl C$. Moreover, for each $r\in G$ we have $\beta_r\circ\Theta(\mu)=\Theta(\mu)\circ\beta_r$:
\begin{eqnarray*}
\beta_r(\Theta(\mu)(M_a))
& = &
\lambda_r^G\left(\int_G\rho_s^GM_a\rho_{s^{-1}}^Gd\mu(s)\right) \lambda_{r^{-1}}^G\\
&=&\int_G\rho_s^{G}\lambda_r^GM_a\lambda_{r^{-1}}^G\rho_{s^{-1}}d\mu(s)=\Theta(\mu)(\beta_r(M_a)).
\end{eqnarray*}
Hence $\Theta(\mu)$ gives rise to a completely bounded map on
$\cl C\rtimes_{\beta,r}G$, {\it i.e. } the function $F$, given by
$F(t)(M_a):=\Theta(\mu)(M_a)$, is a Herz-Schur $(\cl C, G, \beta)$-multiplier.
For $\Lambda=\{\mu_t\}_{t\in G}$  we let $F_\Lambda(t)(M_a)=\Theta(\mu_t)(M_a)$.
The class of Herz-Schur multipliers  $F_\Lambda$  includes the convolution
multipliers examined in the present section, whose study will be pursued elsewhere. 
\end{remark}

\noindent {\bf Acknowledgement.} We would like to thank Sergey Neshveyev and Adam Skalski for
many helpful conversations during the preparation of this paper.


\begin{thebibliography}{99}

\bibitem{akt}
\textsc{M. Anoussis, A. Katavolos and I. G. Todorov},
\textit{Ideals of the Fourier algebra, supports and harmonic operators},
\textrm{to appear in the Math. Proc. Cambridge Plilos. Soc.}


\bibitem{bc}
\textsc{E. Bedos and R. Conti},
\textit{Fourier series and twisted C*-crossed products},
\textrm{J. Fourier Anal. Appl  21 (2015), 32-75}.


\bibitem{BS4}
{\sc M. S. Birman and M. Z. Solomyak},
\textit{ Double operator integrals in a Hilbert space},
\textrm{Int. Eq. Oper. Th. 47 (2003),  no. 2, 131-168}.

\bibitem{blm}
{\sc D. P. Blecher and C. Le Merdy}, {\it Operator algebras and their
modules -- an operator space approach}, {\rm Oxford University
Press, 2004}.


\bibitem{bf}
\textsc{M. Bo\.{z}ejko and G. Fendler},
{\it Herz-Schur multipliers and completely bounded multipliers of
the Fourier algebra of a locally compact group},
\textrm{Boll. Un. Mat. Ital. A (6) 2 (1984), no. 2, 297-302}.



\bibitem{bo}
\textsc{N. P. Brown and N. Ozawa},
{\it C*-algebras and finite dimensional approximations},
{\rm American Mathematical Society, 2008}.



\bibitem{ch}
\textsc{J. de Canniere and U. Haagerup},
\textit{Multipliers of the Fourier algebras of some simple Lie groups and their discrete subgroups},
\textrm{Amer. J. Math. 107 (1985), no. 2, 455-500}.




\bibitem{EKQR}
\textsc{S. Echterhoff, S. Kaliszewski, J. Quigg, I. Raeburn}
\textit{Naturality and induced representations},
\textrm{Bull. Austral. Math. Soc.  61  (2000),  no. 3, 415–-438}.

\bibitem{er}
{\sc E. Effros and Z.-J. Ruan},
{\it Operator Spaces},
{\rm Oxford University Press, 2000}.


\bibitem{eym}
{\sc  P. Eymard}, {\it L'alg\`{e}bre de Fourier d'un groupe localement
compact}, {\rm Bull. Soc. Math. France 92 (1964), 181-236}.

\bibitem{g}
{\sc F. Ghahramani},
{\it Isometric representations of $M(G)$ on $B(H)$},
{\rm Glasgow Math. J. 23 (1982), 119-122}.

\bibitem{Gro}
{\sc A. Grothendieck},
\textit{R\'{e}sum\'{e} de la th\'{e}orie m\'{e}trique des produits tensoriels topologiques},
\textrm{Boll. Soc. Mat. Sao-Paulo 8 (1956), 1-79}.


\bibitem{haag}
\textsc{U. Haagerup}, \textit{Decomposition of completely bounded
maps on operator algebras}, \textrm{unpublished manuscript}.


\bibitem{huruya}
\textsc{ T. Huruya},
{\it The second dual of a tensor product of C* -algebras, II},
{\rm Sci. Rep. Niigata Univ. Ser. A 11 (1974), 21-23}.

\bibitem{j}
\textsc{P. Jolissaint},
{\it A characterisation of completely bounded multipliers of Fourier algebras},
{\rm Colloquium Math. 63 (1992) 311-313}.


\bibitem{kr2}
\textsc{R. V. Kadison and J. R. Ringrose},
\newblock{\em Fundamentals of the theory of von Neumann algebras II},
{\rm Academic Press, 1986}.



\bibitem{kp}
\textsc{A. Katavolos and V. I. Paulsen},
\textit{On the ranges of bimodule projections},
\textrm{Canad. Math. Bull. 48 (2005) no. 1, 97-111}.

\bibitem{knudby}
\textsc{S. Knudby},
\textit{The weak Haagerup property},
\textrm{preprint, arXiv 1401.7541}.


\bibitem{kraus_tams}
{\sc J. Kraus},
{\it The slice map problem for $\sigma$-weakly closed subspaces of von Neumann algebras},
{\rm Trans. Amer. Math. Soc. 279 (1983), no. 1, 357-376}.

\bibitem{hk}
{\sc U. Haagerup and J. Kraus},
{\it Approximation properties for group C*-algebras and group von Neumann algebras},
{\rm Trans. Amer. Math. Soc. 344 (1994), no. 2, 667-699}.

\bibitem{renault2}
\textsc{P. S. Muhly and J. N. Renault},
\textit{C*-algebras of multivariable Wienner-Hopf operators},
\textrm{Trans. Amer. Math. Soc. 274 (1982), no. 1, 1-44}.


\bibitem{ms}
\textsc{P. S. Muhly and B. Solel},
{\it Tensor algebras over C*-correspondences: representations, dilations, and C*-envelopes},
{\rm J. Funct. Anal. 158 (1998), no. 2, 389-457}.


\bibitem{nakagami-takesaki}
\textsc{Y. Nakagami and M. Takesaki},
\newblock{\em Duality for crossed products of von Neumann algebras},
{\rm Springer-Verlag, 1979}.

\bibitem{neurun}
\textsc{M. Neufang and  V. Runde},
\newblock {\em Harmonic operators: the dual perspective},
\newblock {\rm {Math. Z.}} 255 (2007), no. 3, 669-690.

\bibitem{neuruaspro}
\textsc{M. Neufang,  Z-J.  Ruan and N. Spronk},
\newblock {\em Completely isometric representations of
$M_{cb}A(G)$ and  $UCB(\widehat G)^*$},
\newblock {\rm {Trans. Am. Math. Soc.}} 360 (2008), no. 3, 1133-1161.

\bibitem{paulsen}
{\sc V. I. Paulsen}, {\it Completely bounded maps and operator
algebras}, {\rm Cambridge University Press, 2002}.


\bibitem{ped2}
\textsc{G. K. Pedersen}, \textit{C*-algebras and their
automorphism groups}, \textrm{Academic Press, 1979}.

\bibitem{peller}
\textsc{V. V. Peller},
{\it Hankel operators in the theory of perturbations of unitary and selfadjoint operators},
{\rm Funct. Anal. Appl. 19 (1985), no. 2, 111-123}.


\bibitem{pimsner}
\textsc{M. Pimsner},
{\it A class of C*-algebras generalizing both Cuntz-Krieger algebras and crossed products by $\bb{Z}$},
{\rm \lq\lq Free Probability Theory'' (D. Voiculescu, Ed.), Fields Instit. Comm. (1997) 12, 189-212}.


\bibitem{Pi}
\textsc{G. Pisier}, \textit{Similarity problems and completely bounded maps},
\textrm{Springer-Verlag, 2001}.


\bibitem{pisier}
{\sc G. Pisier},
{\it Introduction to operator space theory},
{\rm Cambridge University Press, 2003}.


\bibitem{rbook}
{\sc J. Renault},
\textit{A groupoid approach to C*-algebras},
\textrm{Springer-Verlag, 1980}.


\bibitem{renault}
{\sc J. Renault},
{\it The Fourier algebra of a measured groupoid and its multipliers},
{\rm J. Funct. Anal. 145 (1997), no. 2, 455-490}.


\bibitem{rieffel}
{\sc M. A. Rieffel},
{\it On the uniqueness of the Heisenberg commutation relations},
{\rm Duke Math. J. 39 (1972), 745-752}.


\bibitem{rudin}
{\sc W. Rudin},
{\it Fourier analysis on groups},
{\rm John Wiley \& Sons, 1990}.

\bibitem{stt_clos}
{\sc V. S. Shulman, I. G. Todorov and L. Turowska},
{\it Closable multipliers},
{\rm Int. Eq. Oper. Th. 69 (2011), no. 1, 29-62}.


\bibitem{ss}
{\sc R. R. Smith and N. Spronk},
{\it Representations of group algebras in spaces of completely bounded maps},
{\rm Indiana Univ. Math. J. 54 (2005), no. 3, 873-896}.


\bibitem{spronk}
{\sc N. Spronk},
{\it Measurable Schur multipliers and completely bounded multipliers of the Fourier algebras},
{\rm Proc. London Math. Soc (3) 89 (2004), 161-192}.


\bibitem{stormer}
{\sc E. St\o{}rmer},
{\it Regular abelian Banach algebras of linear maps of operator algebras},
{\rm J. Funct. Anal. 37 (1980), no. 3, 331-373}.

\bibitem{takesaki1}
{\sc M. Takesaki},
{\it Theory of operator algebras I},
{\rm Springer-Verlag, 2003}.


\bibitem{tt_p}
{\sc I. G. Todorov and L. Turowska},
{\it Sets of $p$-multiplicity in locally compact groups},
{\rm Studia Math.  226 (2015), 75-93}.


\bibitem{tomiyama}
{\sc J. Tomiyama},
{\it On the projection of norm one in W*-algebras},
{\rm Proc. Japan Acad. 33 (1957), no. 10, 608-612}.


\bibitem{dw}
{\sc D. P. Williams},
{\it Crossed products of C*-algebras},
{\rm American Mathematical Society, 2007}.



\end{thebibliography}
\end{document}